\documentclass[a4paper]{amsart}

\usepackage[utf8]{inputenc}
\usepackage{amssymb}
\usepackage{amsthm}
\usepackage{amsmath}
\usepackage{mathtools}
\usepackage{enumitem}
\usepackage{nicefrac}
\usepackage{xcolor}
\usepackage{tikz-cd}
\usepackage{hyperref}

\DeclareFontFamily{U}{matha}{\hyphenchar\font45}
\DeclareFontShape{U}{matha}{m}{n}{
	<5> <6> <7> <8> <9> <10> gen * matha
	<10.95> matha10 <12> <14.4> <17.28> <20.74> <24.88> matha12
}{}
\DeclareSymbolFont{matha}{U}{matha}{m}{n}
\DeclareMathSymbol{\hash}{2}{matha}{"23}

\begin{document}
	
\theoremstyle{plain}
\newtheorem{lemma}{Lemma}
\numberwithin{lemma}{section}
\newtheorem{proposition}[lemma]{Proposition}
\newtheorem{corollary}[lemma]{Corollary}
\newtheorem{theorem}[lemma]{Theorem}
	
\theoremstyle{definition}
\newtheorem{definition}[lemma]{Definition}
\newtheorem{question}[lemma]{Question}

\theoremstyle{remark}
\newtheorem{remark}[lemma]{Remark}

\newcommand{\period}{\text{.}}
\newcommand{\comma}{\text{,}}

\newcommand{\rca}{\textsf{RCA}}
\newcommand{\aca}{\textsf{ACA}}
\newcommand{\sigmatwoind}{\Sigma^0_2\textsf{-IND}}
\newcommand{\zfc}{\textsf{ZFC}}

\newcommand{\n}{\mathbb{N}}
\newcommand{\on}{\textsf{On}}
\newcommand{\wo}{\textsf{WO}}
\newcommand{\atr}{\textsf{ATR}}
\newcommand{\pica}{\Pi^1_1\textsf{-CA}_0}
\newcommand{\pitr}{\Pi^1_1\textsf{-TR}_0}
\newcommand{\kriz}{K\v{r}\'{i}\v{z}}
\newcommand{\bta}{\textsf{B}}
\newcommand{\btaalt}{\textsf{\~{B}}}
\newcommand{\btal}{\textsf{B}^l}
\newcommand{\btalalt}{\textsf{\~{B}}^l}
\newcommand{\eqnf}{=_{\textsf{NF}}}
\newcommand{\phinf}{=_{\varphi\textsf{-NF}}}
\newcommand{\alphanf}{=_{\alpha\textsf{-NF}}}
\newcommand{\weak}{\textsf{w}}
\newcommand{\strong}{\textsf{s}}
\newcommand{\gordeev}{\textsf{g}}
\newcommand{\rec}{\textsf{r}}
\newcommand{\gapw}{\textsf{S}^\weak}
\newcommand{\gapg}{\textsf{S}^\gordeev}
\newcommand{\gaps}{\textsf{S}^\strong}
\newcommand{\gapr}{\textsf{S}^\rec} 

\newcommand{\REF}{\textcolor{red}{REF}}
\newcommand{\CHECK}{\textcolor{red}{CHECK}}

\newcommand{\treetype}{\mathfrak{B}}
\newcommand{\labeltype}{\mathfrak{L}}
\newcommand{\emptytype}{\mathfrak{E}}

\newcommand{\dom}{\textsf{dom}}

\newcommand{\badfin}{\textsf{Bad}_\textsf{fin}}

\newcommand{\hess}{\mathbin{\hash}}
\newcommand{\hessMul}{\mathbin{\rotatebox[origin=c]{-45}{$\hash$}}}

\title[Maximal order types for sequences with gap condition]{Maximal order types\\for sequences with gap condition}
\author[P. Uftring]{Patrick Uftring}

\address{University of the Bundeswehr Munich, Department of Computer Science, Werner-Heisenberg-Weg 39, 85579 Neubiberg, Germany}
\email{patrick.uftring@unibw.de}
\thanks{This research is partially funded by the Deutsche Forschungsgemeinschaft (DFG, German Research Foundation) -- Project number 460597863.}

\begin{abstract}
    Higman's lemma states that for any well partial order~$X$, the partial order~$X^*$ of finite sequences with members from~$X$ is also well. By combining results due to Girard as well as Sch\"{u}tte and Simpson, one can show that Higman's lemma is equivalent to \emph{arithmetical comprehension} over~$\rca_0$, the usual base system of reverse mathematics.

    By incorporating Friedman's \emph{gap condition}, Sch\"{u}tte and Simpson defined a slightly different order on finite number sequences with fewer comparisons. While it is still true that their definition yields a well partial order, it turns out that arithmetical comprehension is not enough to prove this fact.

    Gordeev considered a symmetric variation of this gap condition for sequences with members from arbitrary well orders. He could show, over~$\rca_0$, that his partial order on sequences is well (for any underlying well order) if and only if \emph{arithmetical transfinite recursion} is available.

	We present a new and simpler proof of this fact and extend Gordeev's results to weak and strong gap conditions as well as binary trees with weakly ascending labels.
	Moreover, we compute the maximal order types of all considered structures.
\end{abstract}

\maketitle

\section{Introduction}
Well partial orders generalize the concept of well orders: We call a partial order~${X}$ \emph{well} if and only if any sequence~${(x_n)_{n \in \n} \subseteq X}$ is \emph{good}, i.e., it has indices~${i < j}$ with~${x_i \leq x_j}$. Sequences without this property are called \emph{bad}. Of course, any well order is a (linear) well partial order. However, there are also interesting well partial orders that are not linear, e.g.~several orders on finite trees or even the graph minor relation on finite undirected graphs. The latter is a celebrated result due to Robertson and Seymour (cf.~\cite{RobertsonSeymour}).

Given a well partial order~${X}$, we can consider finite sequences~${X^*}$ together with Higman's famous embeddability relation on sequences (cf.~\cite{Higman}): Let~$|s|$ and~$s_i$ denote the length and the $i$-th member, respectively, of any sequence~${s \in X^*}$ for any index~${i < |s|}$. For~${s, t \in X^*}$, we write~${s \leq t}$ if and only if there is a strictly increasing map~${f: |s| \to |t|}$ such that~${s_i \leq t_{f(i)}}$ holds for any index~${i < |s|}$. This partial order on sequences~${X^*}$ is quite fundamental since it appears as a suborder of many other interesting examples, e.g., labeled trees or graphs. Moreover, it is easy to show that~${X^*}$ is a well partial order: We use the so-called \emph{minimal bad sequence argument}, which is due to Nash-Williams (cf.~\cite{NashWilliams}). Assume, for contradiction, that there is a bad sequence~${(s_n)_{n \in \n} \subseteq X^*}$. Then, we can assume that~${s_0}$ has minimal length among the set of all such bad sequences. Similarly, we can assume that~${s_1}$ has minimal length among all such bad sequences that begin with~${s_0}$, and so on. Now, we split each sequence~${s_n}$ into its head~${x_n}$ and its tail~${s'_n}$. Writing~${*}$ in order to denote the concatenation of sequences, we have~${s_n = \langle x_n \rangle * s'_n}$ for any index~${n \in \n}$. Using (a weak consequence of) Ramsey's theorem for pairs and two colors together with the assumption that~$X$ is a well partial order, we can find an increasing sequence of indices~${(n_i)_{i \in \n}}$ such that~${(s'_{n_i})_{i \in \n}}$ is bad. With this, one readily shows that~${(t_i)_{i \in \n}}$ with
\begin{equation*}
	t_i :=
	\begin{cases}
		s_i &\text{if~${i < n_0}$,}\\
		s'_{n_{i - n_0}} &\text{otherwise}
	\end{cases}
\end{equation*}
constitutes a bad sequence in~${X^*}$ that is even \emph{smaller} than our original~${(s_n)_{n \in \n}}$, i.e., we can find an index~${i \in \n}$ such that~${|t_i| < |s_i|}$ and~${|t_j| = |s_j|}$ hold for all~${j < i}$. In fact, we can simply choose~${i := n_0}$.
Clearly, this contradicts our assumption on~${(s_n)_{n \in \n}}$.

This minimal bad sequence argument is quite powerful: It does not only work for sequences but also for more complicated structures, e.g., labeled trees. But it has its limits, as any mathematical tool does, and it is exactly this limit that Friedman explored when he invented a new relation on certain trees: \emph{embeddability with gap condition} (cf.~\cite{SimpsonGap}). In the formal setting of reverse mathematics (cf.~\cite{Simpson}), he could show that even~${\pica}$, the strongest of the \emph{Big Five} (a collection of five subsystems of second-order arithmetic, which correspond to many results from ordinary mathematics), is not strong enough to prove that his construction yields a well partial order.

The system~${\pica}$ is characterized by a set existence principle for formulas of a certain form (``for all sets, some arithmetical statement holds''). This axiom is exactly as strong as the minimal bad sequence argument (cf.~\cite{MarconeMinimalBad}). Later, \kriz{} extended Friedman's results to trees with labels from arbitrary well orders (cf.~\cite{Kriz}). Freund could show that the minimality principle employed by \kriz{} is as strong as~${\pitr}$ (cf.~\cite{FreundKruskal}), a logical system that may use transfinite iterations of the axiom that characterizes~${\pica}$.

Instead of considering trees, we can also restrict ourselves to sequences. This case was first studied by Sch\"{u}tte and Simpson for sequences with members from finite linear orders (cf.~\cite{SchuetteSimpson}, see also \cite{RMWGapCondition} for further investigations). From there, Gordeev considered an extension to arbitrary well orders, quite similar in spirit to \kriz's extension of Friedman's original result to arbitrary well orders published in the same year (cf.~\cite{Gordeev}). Gordeev's definition is the one that we begin with:
\begin{definition}[Gordeev's gap condition]\label{def:gordeev_gap_sequence}
	Given a well order~${\alpha}$, we define the partial order~${\gapg_\alpha}$ of sequences~${\alpha^*}$ ordered using Gordeev's symmetric gap condition: For elements~${s, t \in \gapg_\alpha}$, we write~${s \leq_{\gapg_\alpha} t}$ or~${s \leq_\gordeev t}$ if and only if there exists a strictly increasing map~${f: |s| \to |t|}$ such that the following are satisfied:
	\begin{enumerate}[label=\roman*)]
		\item ${s_i \leq t_{f(i)}}$ for all~${i < |s|}$,
		\item ${s_i \leq t_j}$ or~${s_{i+1} \leq t_j}$ for all~${i < |s| - 1}$ and~${j}$ with~${f(i) < j < f(i+1)}$.
	\end{enumerate}
\end{definition}
The first property is nothing else than the usual requirement for embedding sequences. The second property is the gap condition.
It is called the \emph{symmetric} gap condition as it, in contrast to the original definitions, not only considers the right end of the gap with~${s_{i+1} \leq t_j}$ but also the left end with~${s_i \leq t_j}$. (This has the effect that~${s \leq_\gordeev t}$ holds if and only if~${s' \leq_\gordeev t'}$ does, for sequences~$s'$ and~$t'$ that result from~$s$ and~$t$, respectively, if we reverse the order of their members.) Gordeev's main result is that over the weak base theory of~${\rca_0}$, these partial orders~${\gapg_\alpha}$ are well for all well orders~${\alpha}$, if and only if the axiom of \emph{arithmetical transfinite recursion} is available.

As we will see in Lemma~\ref{lem:gapg_gapw_are_identical}, this partial order is not only isomorphic but, actually, \emph{equal} to the non-symmetric variant, which is known as the \emph{weak gap condition}:
\begin{definition}[Weak gap condition]\label{def:weak_gap_sequence}
	Given a well order~${\alpha}$, we define the partial order~${\gapw_\alpha}$ of sequences~${\alpha^*}$ ordered using the weak gap condition: For elements~${s, t \in \gapw_\alpha}$, we write~${s \leq_{\gapw_\alpha} t}$ or~${s \leq_\weak t}$ if and only if there exists a strictly increasing map~${f: |s| \to |t|}$ such that the following are satisfied:
	\begin{enumerate}[label=\roman*)]
		\item ${s_i \leq t_{f(i)}}$ for all~${i < |s|}$,
		\item ${s_{i+1} \leq t_j}$ for all~${i < |s| - 1}$ and~${j}$ with~${f(i) < j < f(i+1)}$.
	\end{enumerate}
\end{definition}
So far, we have only considered \emph{inner gaps}. Similar to the original definitions, we can also take all those members into account that lie to the left of anything that we map to. These members sit in the \emph{outer gap}. Additionally considering this new gap yields what is known as the \emph{strong gap condition}:
\begin{definition}[Strong gap condition]\label{def:strong_gap_sequence}
	Given a well order~${\alpha}$, we define the partial order~${\gaps_\alpha}$ of sequences~${\alpha^*}$ ordered using the strong gap condition: For elements~${s, t \in \gaps_\alpha}$, we write~${s \leq_{\gaps_\alpha} t}$ or~${s \leq_\strong t}$ if and only if there exists a strictly increasing map~${f: |s| \to |t|}$ such that the following are satisfied:
	\begin{enumerate}[label=\roman*)]
		\item ${s_i \leq t_{f(i)}}$ for all~${i < |s|}$,
		\item ${s_{i+1} \leq t_j}$ for all~${i < |s| - 1}$ and~${j}$ with~${f(i) < j < f(i+1)}$,
		\item ${s_0 \leq t_j}$ for all~${j < f(0)}$ if~${|s| > 0}$.
	\end{enumerate}
\end{definition}
The nice property of this strong variant is that its definition can equivalently be given in a concise recursive way (the equivalence of both definitions will be shown in Lemma~\ref{lem:gapg_gapw_are_identical}):
\begin{definition}[Strong gap condition, recursive definition]\label{def:strong_gap_sequence_recursive}
	For any well order~${\alpha}$, we define the partial order~${\gapr_\alpha}$ of sequences~${\alpha^*}$ ordered using the strong gap condition given by a recursive definition. We write~${\leq_{\gapr_\alpha}}$ or~${\leq_\rec}$ for the smallest relation satisfying the following properties:
	\begin{enumerate}[label=\roman*)]
		\item ${\langle \rangle \leq_\rec s}$ for all~${s \in \gapr_\alpha}$,
		\item ${\langle \beta \rangle * s \leq_\rec \langle \gamma \rangle * t}$ for~${\beta \leq \gamma \in \alpha}$ and~${s, t \in \gapr_\alpha}$ if~${s \leq_\rec t}$ or~${\langle \beta \rangle * s \leq_\rec t}$.
	\end{enumerate}
\end{definition}
Notice that, in contrast to the recursive definition of the simpler order on sequences employed by Higman's lemma, the relation ${\beta \leq \gamma}$ is even required in case of ${\langle \beta \rangle * s \leq_{\rec} t}$.

In Section~\ref{sec:ATR}, we extend Gordeev's result from the symmetric gap condition to the weak and strong gap conditions. Moreover, we use a different and, arguably, more transparent approach involving \emph{quasi-embeddings} for trees. We call a map~${f: X \to Y}$ between partial orders~${X}$ and~${Y}$ a quasi-embedding if it is order reflecting, i.e.,~${f(x) \leq f(x')}$ implies~${x \leq x'}$ for~${x, x' \in X}$. One can easily see that if~${Y}$ is a well partial order, then this property is reflected to~${X}$ via~${f}$.

Consider a simple number sequence, e.g.,~${s := \langle 2, 0, 1, 0, 3 \rangle}$. We can turn this sequence into a binary tree with labels from numbers as follows: Take the first occurrence of the smallest member ($0$). This will be the label of our root. Then, split our sequence into two subsequences at this member ($s_l := \langle 2 \rangle$ and~${s_r := \langle 1, 0, 3 \rangle}$). Recursively compute the binary trees for these sequences. The tree of the former sequence~${s_l}$ will be our left subtree, the tree of the latter sequence~${s_r}$ will be our right subtree. This results in the following tree:
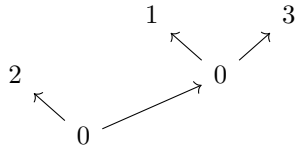
\begin{center}
\begin{tikzcd}[row sep = small, column sep = small]
	& & 1 & & 3\\
	2 & & & 0\ar[lu]\ar[ru] & \\
	& 0\ar[lu]\ar[rru] & & & \\
\end{tikzcd}
\end{center}
\vspace{-0.5em}
It turns out that this map from~${\gapg_{\omega}}$ into binary trees with labels in~${\omega}$ yields a quasi-embedding if we consider the embeddability relation used in Kruskal's theorem (cf.~\cite{Kruskal}) extended to ordered trees. Similarly, the question whether~${\gapg_{\alpha}}$ is a well partial order can now be reduced to the question whether binary trees with labels in~${\alpha}$ are a well partial order. The latter is answered positively by Kruskal's theorem for ordered trees, which is a principle that is proof-theoretically weaker than~${\pica}$ and, thus, also weaker than the minimal bad sequence argument.\footnote{In general, we can argue as follows: First, Kruskal's theorem (for ordered trees and labels from an arbitrary well order) can be proved using the minimal bad sequence argument and, hence,~${\pica}$. Now, the quantifier complexity of Kruskal's theorem makes it strictly weaker than~${\pica}$ over weak systems (see, e.g.,~\cite[Theorem~4.1.1]{UftringPhD}).}

We may also observe that our constructed tree has \emph{weakly ascending labels}, i.e., for any node with label~${\beta}$ all its children have labels~${\gamma}$ with~${\gamma \geq \beta}$. We make this precise with the following definition:
\begin{definition}[Binary trees with weakly ascending labels]
	Given a well order~${\alpha}$ and a (well) partial order~${X}$, we define the partial order~${\bta_{\alpha, X}}$ of binary trees with weakly ascending inner labels in~${\alpha}$ and leaf labels in~${X}$ as follows: The set contains terms
	\begin{enumerate}[label=\roman*)]
		\item ${x \star []}$ for~${x \in X}$,
		\item ${\beta \star [t_l, t_r]}$ for~${\beta \in \alpha}$ and~${t_l, t_r \in \bta_{\alpha, X}}$ such that~${\beta}$ is less than or equal to any \emph{inner label} in~${t_l}$ or~${t_r}$.
	\end{enumerate}
	The tree~${x \star []}$ is a \emph{leaf} and its label is a \emph{leaf label}. The tree~${\beta \star [t_l, t_r]}$ is an \emph{inner node} and its label~$\beta$ as well as all inner labels in~$t_l$ or~$t_r$ are its \emph{inner labels}. Moreover,~${t_l}$ is its \emph{left subtree}, whereas~${t_r}$ is its \emph{right subtree}.
	
	The partial order on~${\bta_{\alpha, X}}$ is the smallest relation~${\leq}$ satisfying
	\begin{enumerate}[label=\roman*), start=3]
		\item ${x \star [] \leq y \star []}$ for~${x \leq y \in X}$,
		\item ${s \leq \beta \star [t_l, t_r]}$ for~${\beta \in \alpha}$ and~${s, t_l, t_r \in \bta_{\alpha, X}}$ if~${s \leq t_l}$ or~${s \leq t_r}$,
		\item ${\beta \star [s_l, s_r] \leq \gamma \star [t_l, t_r]}$ for~${\beta \leq \gamma \in \alpha}$ as well as~${s_l, s_r, t_l, t_r \in \bta_{\alpha, X}}$ if~${s_l \leq t_l}$ and~${s_r \leq t_r}$.
	\end{enumerate}
\end{definition}
Using our method from before, we see that~${\gapg_\alpha}$ may be quasi-embedded into~${\bta_{\alpha, 1}}$.
In fact, we can give an even tighter characterization of our trees: By construction, we always look for the \emph{first occurrence} of the smallest member in our sequence. Thus, all inner labels to the left must be strictly greater. We can conclude that for any node~${\beta}$ all children in its \emph{left} subtree have labels~${\gamma}$ with~${\gamma > \beta}$:
\begin{definition}[Binary trees with weakly, but on left subtrees strictly, ascending labels]
		Given a well order~${\alpha}$ and a (well) partial order~${X}$, we define the partial order~${\btal_{\alpha, X}}$ to be the suborder of~${\bta_{\alpha, X}}$ containing the terms
		\begin{enumerate}[label=\roman*)]
			\item ${x \star []}$ for~${x \in X}$,
			\item ${\beta \star [t_l, t_r]}$ for~${\beta \in \alpha}$ and~${t_l, t_r \in \btal_{\alpha, X}}$ such that~${\beta}$ is strictly less than any inner label in~${t_l}$ and less than or equal to any inner label in~${t_r}$.
		\end{enumerate}
\end{definition}
Our proofs that~${\gapg_{\alpha}}$ and~${\gapw_{\alpha}}$ (but also~${\gaps_{\alpha}}$ and~${\btal_{\alpha, 1}}$) are well partial orders in~${\atr_0}$, i.e.~$\rca_0$ together with arithmetical transfinite recursion, will all reduce to the fact that~${\atr_0}$ already proves that~${\bta_{\alpha, 1}}$ is a well partial order, for any well order~${\alpha}$. This is the main result of Section~\ref{sec:ATR}.

In addition to the equivalence with arithmetical transfinite recursion, we want to compare ``the size'' of~$\gapg_\alpha$, $\gapw_\alpha$, $\gaps_\alpha$, $\bta_\alpha$, and~$\btal_\alpha$ for different well orders~$\alpha$. This can be done using \emph{maximal order types}.

Let~${X}$ be a well partial order. Then, its maximal order type is the supremum of all well orders~${\alpha}$ such that there exists a quasi-embedding from~${\alpha}$ into~${X}$. We denote this supremum by~${o(X)}$. By a famous result due to de Jongh and Parikh (cf.~\cite{deJonghParikh}), we know that this supremum is actually a maximum, i.e., there is a quasi-embedding from~${o(X)}$ into~${X}$.

With this definition, we have a notion to compare the sizes of our well partial orders.
\begin{theorem}\label{thm:maximal_order_types}
	For any ordinal~${\alpha}$, we have the following order types for~${\bta_{\alpha, 1}}$, ${\btal_{\alpha, 1}}$, ${\gapw_\alpha}$, ${\gapg_\alpha}$, and~${\gaps_\alpha}$:
	{\allowdisplaybreaks
	\begin{align*}
		&o(\bta_{\alpha, 1}) = F(\alpha) :=\\
		&\hspace{5em}\begin{cases}
			1 & \text{ if } \alpha = 0\comma\\
			\varphi_{1 + \gamma}(0) & \text{ if } \alpha = \omega^{\gamma} \text{ and } \gamma < \varphi_{\gamma}(0)\comma\\
			\varphi_{1 + \gamma}(F(\delta)) & \text{ otherwise if } \alpha \eqnf \omega^{\gamma} + \delta\period
		\end{cases}\\\\
		&o(\gapg_\alpha) = o(\gapw_\alpha) = o(\btal_{\alpha, 1}) = G(\alpha) :=\\
		&\hspace{5em}\begin{cases}
			1 & \text{ if } \alpha = 0\comma\\
			\omega & \text{ if } \alpha = 1\comma\\
			\omega^{\omega^{G(n)}} & \text{ if } \alpha = n + 1 \text{ and } 0 < n < \omega\comma\\
			\varphi_{\gamma}(0) & \text{ if } \alpha = \omega^{\gamma} \text{ and } 0 < \gamma < \varphi_{\gamma}(0)\comma\\
			\varphi_{\gamma}(G(\delta)) & \text{ otherwise if } \alpha \eqnf \omega^{\gamma} + \delta\period
		\end{cases}\\\\
		&o(\gaps_\alpha) = H(\alpha) :=\\
		&\hspace{5em}\begin{cases}
			1 & \text{ if } \alpha = 0\comma\\
			G(\alpha)^{\omega^{\gamma}} \cdot H(\delta) & \text{ if } \alpha \eqnf \omega^{\gamma} + \delta\period
		\end{cases}
	\end{align*}
	}
    Here, ``$\eqnf$'' restricts~$\delta$ to values strictly below~$\omega^{\gamma + 1}$.
\end{theorem}

This theorem will be proven in Section~\ref{sec:order_types}. We observe that the given characterization of our trees constructed from sequences with gap condition using~${\btal_{\alpha, 1}}$ is as tight as possible, at least when it comes to maximal order types. This answers an open question asked during ``Trends in Proof Theory 2023'' at the University of Ghent.

Nested applications of Veblen functions can be described using \emph{hyperations} (cf.~\cite[Corollary~4.10]{FDJoosten}) and have already been considered by Girard and Vauzeilles (cf.~\cite[Corollary~III.3.2]{GirardVauzeilles}. Maximal order types for similar trees using a different embeddability relation (that does not preserve infima) are independently being investigated by Alakh Dhruv Chopra.

This article corresponds to Chapter~3 of the author's PhD thesis~\cite{UftringPhD}, which also considers definitions of~${\gapg_X}$,~${\gapw_X}$,~${\gaps_X}$,~${\bta_{X, Y}}$, and~${\btal_{X, Y}}$ for proper well partial orders~$X$ and~$Y$.

\subsection*{Acknowledgments}

I would like to thank my PhD advisor Anton Freund as well as Alakh Dhruv Chopra, Fedor Pakhomov, and Andreas Weiermann for our discussions and their valuable input. Parts of this work have been funded by the Deutsche Forschungsgemeinschaft (DFG, German Research Foundation) – Project number 460597863.

\section{Basic properties of sequences with gap condition}

We prove the following basic properties with~${\rca_0}$ in mind. For a general introduction to reverse mathematics, we refer to \cite{Simpson}. All of these results can be repeated in set theory for arbitrary (and even uncountable) well orders.

The proofs are very detailed in order to be convincing in the context of~${\rca_0}$. In many cases, they are more tedious than insightful and can, therefore, safely be skipped.

\begin{lemma}
	For any well order~${\alpha}$ and sequences~${s, t_l, t, t_r \in \gapw_\alpha}$ with~${s \leq_\weak t}$, we also have~${s \leq_\weak t_l * t * t_r}$.
\end{lemma}
\begin{proof}
	This follows immediately from Definition~\ref{def:weak_gap_sequence}.
\end{proof}
This lemma is more of an observation than an actual result and we will use it implicitly from now on. It should be noted, though, that this property does not hold for the strong gap condition in general since members that are newly introduced via~$t_l$ might violate the outer gap condition: We have~${\langle 1 \rangle \leq_\strong \langle 1 \rangle}$ but~${\langle 1 \rangle \nleq_\strong \langle 0, 1 \rangle}$.

\subsection{The interplay between different definitions}

In this part, we study the relationship between the partial orders~${\gapg_\alpha}$, ~${\gapw_\alpha}$,~${\gaps_\alpha}$, and~${\gapr_\alpha}$, for arbitrary well orders~${\alpha}$.

\begin{lemma}\label{lem:gapg_gapw_are_identical}
	For any well order~${\alpha}$, the orders~${\gapg_\alpha}$ and~${\gapw_\alpha}$ are identical, i.e., the identity map between the underlying sets is an isomorphism.
\end{lemma}

\begin{proof}
	Given two arbitrary sequences~${s, t \in \gapg_\alpha}$, we prove that~${s \leq_\gordeev t}$ holds if and only if~${s \leq_\weak t}$ does. This yields our claim since the underlying sets of~${\gapg_\alpha}$ and~${\gapw_\alpha}$ are identical.
	
	By definition, it is clear that~${s \leq_\weak t}$ implies~${s \leq_\gordeev t}$. Therefore, assume that~${s \leq_\gordeev t}$ holds. This relation is realized by a function~${f: |s| \to |t|}$ that satisfies the properties demanded by Definition~\ref{def:gordeev_gap_sequence}. Let us denote by~${\sigma(f)}$ the sum of the range of~${f}$, i.e.,~${\sigma(f) := \sum_{i < |s|} f(i)}$. Using~${\Sigma^0_1}$-induction (actually already~${\Delta^0_1}$-induction), we can assume that for our choice of~${f}$, the value~${\sigma(f)}$ is maximal among all~${\sigma(g)}$ for realizers~${g}$ of~${s \leq_\gordeev t}$.
	
	We claim that~${f}$ realizes~${s \leq_\weak t}$. In order to derive a contradiction, assume that this is not the case and we find an index~${i < |s| - 1}$ together with an index~${j < |t|}$ such that~${f(i) < j < f(i+1)}$ and~${t_j < s_{i+1}}$ hold. Since~${f}$ is a valid realizer for~${s \leq_\gordeev t}$, this immediately entails~${s_i \leq t_j}$. Thus, we can define a new realizer~${g: |s| \to |t|}$ as follows:
	\begin{equation*}
		g(k) :=
		\begin{cases}
			j &\text{if } k = i\comma\\
			f(k) &\text{otherwise.}
		\end{cases}
	\end{equation*}
	First, from~${s_i \leq t_j}$ it is clear that~${s_k \leq t_{g(k)}}$ holds for all~${k < |s|}$. Next,~${g}$ is still strictly increasing since~${g(i) = j < f(i+1) = g(i+1)}$ holds. Finally, we prove that~${g}$ realizes~${s \leq_\gordeev t}$. We only need to check the gaps that have changed with respect to~${f}$, i.e., the gap between~${g(i-1)}$ and~${g(i)}$ (provided that~${i}$ is positive) as well as the gap between~${g(i)}$ and~${g(i+1)}$. The latter case is trivial since the interval between~${g(i)}$ and~${g(i+1)}$ is a subset of the interval between~${f(i)}$ and~${f(i+1)}$. In the former case, assume that~${i}$ is positive and there is some index~${l < |t|}$ that violates the gap condition, i.e.,~${g(i-1) < l < g(i)}$ holds together with~${t_l < s_{i-1}, s_i}$. If~${l}$ is strictly less than~${f(i)}$, we have~${f(i-1) < l < f(i)}$ because of~${f(i-1) = g(i-1)}$ and, thus, already a violation in the gap condition of~${f}$. If~${l}$ is equal to~${f(i)}$, then~${t_{f(i)} < s_i}$ violates the basic property~${s_i \leq t_{f(i)}}$ of~${f}$. If~${l}$ is strictly greater than~${f(i)}$, then we have~${f(i) < l < g(i) \leq f(i+1)}$. This entails one of~${t_l \geq s_i}$ or~${t_l \geq s_{i+1}}$ since~${f}$ is a valid realizer. With~${s_i \leq t_j < s_{i+1}}$, this implies~${t_l \geq s_i}$. However, this violates our assumption~${t_l < s_i}$.
	
	Finally, it is clear that~${\sigma(g) = \sigma(f) + (j - f(i)) > \sigma(f)}$ holds. This contradicts our assumption on the maximality of~${\sigma(f)}$.
\end{proof}
As an immediate consequence, we see that the maximal order types of~${\gapg_\alpha}$ and~${\gapw_\alpha}$ are identical, for any well order~$\alpha$.
We continue with a result that shows the connection between the weak and the strong gap condition. This will be important later, when we move to recursive definitions.
\begin{lemma}\label{lem:gap_strong_and_weak}
	Let~${\alpha}$ be a well order. Given~${s, t \in \gapw_\alpha}$, the following are equivalent:
	\begin{enumerate}[label=\alph*)]
		\item ${s \leq_\weak t}$
		\item ${\langle 0 \rangle * s \leq_\strong \langle 0 \rangle * t}$
		\item ${\langle \beta \rangle * s \leq_\strong \langle \gamma \rangle * t}$ for all~${\beta \leq \gamma \in \alpha}$ such that~${\beta}$ is less than or equal to any element in~${t}$.
	\end{enumerate}
\end{lemma}

\begin{proof}
	We begin with ``a)~${\Rightarrow}$ c)'':
	Assume that~${s \leq_\weak t}$ holds and is realized by a function~${f: |s| \to |t|}$.
	If~${s}$ is empty, then~${\langle \beta \rangle * s \leq_\strong \langle \gamma \rangle * t}$ holds trivially. Assume that~${s}$ is nonempty, i.e.,~${f(0)}$ is defined.
	We define~${g: |s|+1 \to |t|+1}$ with
	\begin{equation*}
		g(i) :=
		\begin{cases}
			f(0) & \text{if } i = 0\comma\\
			f(i - 1) + 1 & \text{otherwise.}
		\end{cases}
	\end{equation*}
	Clearly,~${g}$ is strictly monotone.
	We prove that~${g}$ is a realizer of~${\langle \beta \rangle * s \leq_\strong \langle \gamma \rangle * t}$. First, we show that~${(\langle \beta \rangle * s)_i \leq (\langle \gamma \rangle * t)_{g(i)}}$ holds for all~${i \leq |s|}$: In the case of~${i = 0}$, we have~${(\langle \beta \rangle * s)_0 = \beta \leq (\langle \gamma \rangle * t)_{g(0)}}$ by assumption on~${\beta}$. Otherwise, if~${i = j+1}$ holds, we derive the inequality~${(\langle \beta \rangle * s)_i = s_j \leq t_{f(j)} = (\langle \gamma \rangle * t)_{f(j)+1} = (\langle \gamma \rangle * t)_{g(i)}}$ by assumption on~${f}$.
	
	Finally, we prove that~${g}$ respects all gap conditions. Clearly, the outer gap condition holds since any index~${j}$ with~${j < g(0)}$ satisfies~${(\langle \beta \rangle * s)_0 = \beta \leq t_j}$ by assumption on~${\beta}$. Now, we consider an inner gap, i.e., let there be~${i < |s|}$ together with an index~${j}$ such that~${g(i) < j < g(i+1)}$ holds. Since~${g(0) = f(0)}$ and~${g(1) = f(1 -1) + 1 = f(0) + 1}$ do not leave a gap, we can assume that~${i}$ is positive. Thus, we have the inequality~${f(i-1)+1 < j < f(i)+1}$. By assumption,~${f}$ respects the gap condition, i.e., we have~${s_i \leq t_{j-1}}$. Hence,~${(\langle \beta \rangle * s)_{i+1} \leq (\langle \gamma \rangle * t)_{j}}$ entails our claim.
	
	The direction ``c)~${\Rightarrow}$ b)'' is trivial. Thus, we finish with ``b)~${\Rightarrow}$ a)'':
	Assume that~${\langle 0 \rangle * s \leq_\strong \langle 0 \rangle * t}$ holds and that it is realized by~${f: |s|+1 \to |t|+1}$. We define~${g: |s| \to |t|}$ with~${g(i) := f(i+1)-1}$. Clearly,~${g}$ is well-defined and strictly monotone as the same already holds for~${f}$. Now, we prove that~${g}$ realizes~${s \leq_\weak t}$. First, we have~${s_i = (\langle 0 \rangle * s)_{i+1} \leq (\langle 0 \rangle * t)_{f(i+1)} = t_{f(i+1)-1}}$.
	
	Finally, we prove that~${g}$ respects the gap conditions. There is no outer gap condition since we only want to show~${s \leq_\weak t}$. We consider an inner gap, i.e., let~${i < |s| - 1}$ together with an index~${j}$ such that~${g(i) < j < g(i+1)}$ holds. This entails~${f(i+1) < j+1 < f(i+2)}$. Thus, we can arrive at our claim using the inequality~${s_{i+1} = (\langle 0 \rangle * s)_{i+2} \leq (\langle 0 \rangle * t)_{j+1} = t_j}$.
\end{proof}

\begin{lemma}
	The realizer-based relation~${\leq_\strong}$ and the recursively defined relation~${\leq_\rec}$ coincide.
\end{lemma}

\begin{proof}
	Let~${s, t \in \gapw_\alpha}$ be such that~${s \leq_\strong t}$ holds. We prove by induction on~${|s| + |t|}$ that this implies~${s \leq_\rec t}$. If~${s}$ is empty, we are already finished. Therefore, assume that~${s}$ is nonempty and the relation~${s \leq_\strong t}$ is realized by~${f: |s| \to |t|}$. Since~${s}$ is nonempty, we can assume it to be of the form~${s = \langle s_0 \rangle * s'}$ for~${s' \in \gapw_\alpha}$.
	
	If~${f(0)}$ is positive, then it can easily be seen that~${i \mapsto f(i) - 1}$ is a valid realizer for~${s \leq_\strong t'}$ if~${t}$ is of the form~${t = \langle t_0 \rangle * t'}$: First, we have~${s_i \leq t_{f(i)} = t'_{f(i) - 1}}$ for indices~${i < |s|}$. Next, consider an index~${j}$ in the outer gap, i.e.,~${j < f(0) - 1}$. Using the original outer gap, we have~${s_0 \leq t_{j+1}}$ and, thus,~${s_0 \leq t'_j}$. Finally, consider an index~${j}$ in an inner gap, i.e., there is some index~${i < |s| - 1}$ with~${f(i)-1 < j < f(i+1)-1}$. We have~${s_{i+1} \leq t_{j+1} = t'_j}$. By our induction hypothesis,~${s \leq_\strong t'}$ yields~${s \leq_\rec t'}$. Now, we only have to show that~${s_0 \leq t_0}$ holds in order to derive~${s \leq_\rec t}$. This easily follows from~${0 < f(0)}$, i.e.,~${0}$ is in the outer gap realized by~${f}$, which entails~${s_0 \leq t_0}$.
	
	If~${f(0) = 0}$ holds, then we claim that~${g: |s|-1 \to |t|-1}$ with~${g(i) := f(i+1)-1}$ is a realizer for~${s' \leq_\strong t'}$ if we write~${t = \langle t_0 \rangle * t'}$: First, we observe the inequality~${s'_i = s_{i+1} \leq t_{f(i+1)} = t'_{f(i+1)-1}}$ for all~${i < |s'|}$. Next, consider an index~${j}$ in the outer gap, i.e.,~${j < g(0)}$. This yields~${f(0) = 0 < j + 1 < f(1)}$. Thus, we have the inequality~${s'_0 = s_1 \leq t_{j+1} = t'_j}$. Finally, consider an index~${j}$ in some inner gap, i.e., there is some index~${i < |s'| - 1}$ with~${g(i) < j < g(i+1)}$. This entails the strict inequalities~${f(i+1) < j+1 < f(i+2)}$ and, thus,~${s'_{i+1} = s_{i+2} \leq t_{j+1} = t'_j}$. We conclude that~${s' \leq_\strong t'}$ holds and, by induction hypothesis, this yields~${s' \leq_\rec t'}$. Finally, we have to show the inequality~${s_0 \leq t_0}$. However, this is directly given by the fact that our realizer~${f}$ maps the index~${0}$ (of~${s}$) to the index~${0}$ (of~${t}$).
	
	Let~${s, t \in \gapw_{\alpha}}$ such that~${s \leq_\rec t}$ holds. Similar to before, we prove by induction on~${|s| + |t|}$ that this implies~${s \leq_\strong t}$. Again, we can assume that~${s}$ is nonempty and, hence, of the form~${s = \langle s_0 \rangle * s'}$ for some sequence~${s' \in \gapw_\alpha}$ since~${\langle \rangle \leq_\strong t}$ holds trivially.
	
	Assume that~${s \leq_\rec t}$ holds because we have~${s_0 \leq t_0}$ and~${s' \leq_\rec t'}$ for~${t = \langle t_0 \rangle * t'}$. By induction hypothesis, this yields a realizer~${f: |s'| \to |t'|}$ of~${s' \leq_\strong t'}$. We define a function~${g: |s| \to |t|}$ as follows:
	\begin{equation*}
		g(i) :=
		\begin{cases}
			0 & \text{if } i = 0\comma\\
			f(i-1)+1 & \text{otherwise.}
		\end{cases}
	\end{equation*}
	Clearly,~${g}$ is strictly monotone. We prove that~${g}$ is a valid realizer for~${s \leq_\strong t}$: First, we have~${s_0 \leq t_0 = t_{g(0)}}$. For positive indices~${i < |s|}$, we get the inequality~${s_i = s'_{i-1} \leq t'_{f(i-1)} = t_{f(i-1)+1}}$. We can skip the outer gap condition since there are no indices~${j < g(0) = 0}$. Finally, we show that all inner gaps are respected. Let~${i < |s| - 1}$ and~${j}$ be indices such that~${g(i) < j < g(i+1)}$ holds. In the case where~${i = 0}$ holds, this entails~${j-1 < f(0)}$ and, thus,~${s_1 = s'_0 \leq t'_{j-1} = t_{j}}$ by the outer gap condition of~${f}$. If~${i}$ is positive, then we have~${f(i-1) < j-1 < f(i)}$ and, therefore,~${s_{i+1} = s'_i \leq t'_{j-1} = t_j}$ by some inner gap condition of~${f}$.
	
	Assume that~${s \leq_\rec t}$ holds because we have both~${s_0 \leq t_0}$ and~${s \leq_\rec t'}$ for~${t = \langle t_0 \rangle * t'}$. By induction hypothesis, this yields a realizer~${f: |s| \to |t'|}$ of~${s \leq_\strong t'}$. We define a function~${g: |s| \to |t|}$ with~${g(i) := f(i) + 1}$. Clearly,~${g}$ is strictly monotone. We prove that~${g}$ is a valid realizer for~${s \leq_\strong t}$. First, we have~${s_i \leq t'_{f(i)} = t_{f(i) + 1} = t_{g(i)}}$ for all indices~${i < |s|}$. Next, we prove that~${g}$ respects the outer gap: Let~${j}$ be an index with~${j < g(0)}$. If~${j = 0}$, then we have~${s_0 \leq t_0}$ simply by assumption. If~${j}$ is positive, then this yields~${j - 1 < f(0)}$. Thus, by the outer gap condition of~${f}$, we get~${s_0 \leq t'_{j-1} = t_j}$. Finally, we show that~${g}$ respects all inner gaps: Let~${i < |s| - 1}$ and~${j}$ be indices such that~${g(i) < j < g(i+1)}$ holds. Then, we derive~${f(i) < j-1 < f(i+1)}$. Thus, by the inner gap condition of~${f}$, this yields the inequality~${s_{i+1} \leq t'_{j-1} = t_j}$.
\end{proof}

From now on, we will mostly continue working with the recursive definition. It allows us to present relatively short proofs for the following results:

\subsection{Constructing sequences}

In this part, we consider results that let us combine multiple relations of sequences from~${\gapw_\alpha}$ or~${\gaps_\alpha}$ into relations of larger sequences, for well orders~${\alpha}$.

\begin{lemma}\label{lem:gap_concat_weak_and_strong}
	Let~${\alpha}$ be a well order. Then~${s_l \leq_\weak t_l}$ and~${s_r \leq_\strong t_r}$ imply the relation~${s_l * s_r \leq_\weak t_l * t_r}$ for sequences~${s_l, s_r, t_l, t_r \in \gapw_{\alpha}}$ provided that~${s_r}$ is empty or begins with an element that is less than or equal to any element in~${t_l}$. The relation holds with respect to~${\leq_\strong}$ if we assume~${s_l \leq_\strong t_l}$.
\end{lemma}

Usually, we invoke this lemma for sequences~${s_r}$ that are empty or begin with~${0}$.

\begin{proof}
	First, we assume that~${s_l \leq_\strong t_l}$ and~${s_r \leq_\strong t_r}$ hold. Our claim is that this entails~${s_l * s_r \leq_\strong t_l * t_r}$ under the specified conditions. We proceed by induction on the length of~${t_l}$. If~${t_l}$ is empty, then so must be~${s_l}$ and our claim follows immediately. We continue under the assumption that~${t_l = \langle \beta \rangle * t'_l}$ holds for some~${\beta \in \alpha}$ and~${t'_l \in \gaps_{\alpha}}$. We apply induction on the length of~${s_l}$: Let this sequence be empty. If~${s_r}$ is also empty, we are done. Otherwise, we recall that the first member of~${s_r}$ must be less than or equal to any element in~${t_l}$. Thus, a quick induction yields~${s_r \leq_\strong t_l * t_r}$, which is our claim since~${s_l}$ is empty. If~${s_l}$ is nonempty, i.e., of the form~${s_l = \langle \gamma \rangle * s'_l}$ for some~${\gamma \in \alpha}$ and~${s'_l \in \gaps_{\alpha}}$, there are two possible reasons for~${s_l \leq_\strong t_l}$: If we have~${\gamma \leq \beta}$ and~${s_l \leq_\strong t'_l}$, then we apply our induction hypothesis, which yields~${s_l * s_r \leq_\strong t'_l * t_r}$. With~${\gamma \leq \beta}$, this results in our claim. Otherwise, if we have~${\gamma \leq \beta}$ and~${s'_l \leq_\strong t'_l}$, we apply our induction hypothesis, which yields~${s'_l * s_r \leq_\strong t'_l * t_r}$. Again, with~${\gamma \leq \beta}$, this results in our claim.
	
	Secondly, we assume that~${s_l \leq_\weak t_l}$ and~${s_r \leq_\strong t_r}$ hold. Our claim is that this entails~${s_l * s_r \leq_\weak t_l * t_r}$. By Lemma~\ref{lem:gap_strong_and_weak}, we know that the first assumption implies~${\langle 0 \rangle * s_l \leq_\strong \langle 0 \rangle * t_l}$. We cannot apply our result for the strong gap condition immediately since the first member in~$s_r$ might not be less than or equal to~$0$. Instead, we argue as follows: By definition, there are two possible cases: In the first case, we have~${s_l \leq_\strong t_l}$. Using the argument from the previous paragraph, we arrive at our claim. In the second case, we have~${\langle 0 \rangle * s_l \leq_\strong t_l}$. By the previous paragraph, we conclude~${\langle 0 \rangle * s_l * s_r \leq_\strong t_l * t_r}$, which leads to~${\langle 0 \rangle * s_l * s_r \leq_\strong \langle 0 \rangle * t_l * t_r}$ and, by Lemma~\ref{lem:gap_strong_and_weak}, to our claim.
\end{proof}

\begin{lemma}
	Let~${\alpha}$ be a well order. Then~${s \leq_\strong t}$ implies~${u * s \leq_\strong u * t}$ for sequences~${s, t, u \in \gaps_{\alpha}}$.
\end{lemma}

\begin{proof}
	We proceed by a simple induction along the length of~${u}$: If~${u}$ is empty, our claim is clear. Otherwise, if~${u}$ is of the form~${u = \langle u_0 \rangle * u'}$, we know by induction hypothesis that~${u' * s \leq_\strong u' * s}$ holds. Finally, this entails~${u * s \leq_\strong u * t}$ by the (recursive) definition of~${\gaps_\alpha}$.
\end{proof}

\subsection{Destructing sequences}

In this part, we consider results that let us break apart relations of sequences from~${\gapw_\alpha}$ or~${\gaps_\alpha}$ into relations of smaller sequences, for well orders~${\alpha}$.

\begin{lemma}\label{lem:gap_split_weak_weak_or_strong}
	Let~${\alpha}$ be a well order. Given sequences~${s, t_l, t_r \in \gapw_{\alpha}}$ such that the relation~${s \leq_\weak t_l * t_r}$ holds, we can split~${s}$ into sequences~${s_l, s_r \in \gapw_\alpha}$, i.e.,~${s = s_l * s_r}$ with~${s_l \leq_\weak t_l}$ and~${s_r \leq_\weak t_r}$. The latter inequality also holds with respect to~${\leq_\strong}$ provided that~${s_l}$ is nonempty.
\end{lemma}

\begin{proof}
	We begin by assuming~${s \leq_\strong t_l * t_r}$ and show that this entails both~${s_l \leq_\strong t_l}$ and~${s_r \leq_\strong t_r}$ for sequences~${s_l, s_r \in \alpha^*}$ with~${s = s_l * s_r}$. For this, we proceed by induction on the length of~${t_l}$: If~${t_l}$ is empty, then we choose~${s_l := \langle\rangle}$ and~${s_r := s}$. Our claim follows immediately. Otherwise, let~${t_l}$ be of the form~${t_l = \langle t_0 \rangle * t'_l}$. We continue by induction on the recursive definition of~${s \leq_\strong t_l * t_r}$. If this relation holds since~${s}$ is empty, then we can only define~${s_l := s_r := \langle \rangle}$ and our claim is trivial. Next, assume that~${s}$ is nonempty and we have both~${s_0 \leq t_0}$ and~${s \leq_\strong t'_l * t_r}$. The induction hypothesis yields two sequences~${s_l}$ and~${s_r}$ with~${s = s_l * s_r}$ such that~${s_l \leq_\strong t'_l}$ and~${s_r \leq_\strong t_r}$ hold. If~$s_l$ is nonempty, then~${s_0 \leq t_0}$ leads to~${s_l \leq_\strong t_l}$ (which, otherwise, would be trivial) and we are done. Finally, let~${s}$ be of the form~${s = \langle s_0 \rangle * s'}$ for~${s' \in \gapw_\alpha}$ such that~${s_0 \leq t_0}$ and~${s' \leq_\strong t'_l * t_r}$ hold. The induction hypothesis yields sequences~${s'_l}$ and~${s_r}$ with~${s' = s'_l * s_r}$ such that~${s'_l \leq_\strong t'_l}$ and~${s_r \leq_\strong t_r}$ hold. With~${s_0 \leq t_0}$, the former entails~${\langle s_0 \rangle * s'_l \leq_\strong t_l}$. Thus, we choose~${s_l := \langle s_0 \rangle * s'_l}$ and are finished.
	
	Now, assume~${s \leq_\weak t_l * t_r}$. We want to show that this entails~${s_l \leq_\weak t_l}$ and~${s_r \leq_\weak t_r}$ for some sequences~${s_l}$ and~${s_r}$ with~${s = s_l * s_r}$. By Lemma~\ref{lem:gap_strong_and_weak}, our assumption implies~${\langle 0 \rangle * s \leq_\strong \langle 0 \rangle * t_l * t_r}$. Thus, our previous considerations yield sequences~${s'_l}$ and~${s'_r}$ with~${\langle 0 \rangle * s = s'_l * s'_r}$ such that~${s'_l \leq_\strong \langle 0 \rangle * t_l}$ and~${s'_r \leq_\strong t_r}$ hold. If~${s'_l}$ is empty, then choose~${s_l := \langle \rangle}$ and~${s_r := s}$. Clearly, we have~${s_r = s \leq_\weak \langle 0 \rangle * s = s'_r \leq_\weak t_l * t_r}$. Otherwise, if~${s'_l}$ is nonempty, we know that it must be begin with a zero. Thus, there is a unique sequence~${s_l}$ such that~${\langle 0 \rangle * s_l = s'_l}$ holds. We conclude that~${s_l \leq_\weak s'_l \leq_\strong t_l}$ and~${s_r := s'_r \leq_\strong t_r}$ hold and, hence, our claim.
\end{proof}

\begin{lemma}\label{lem:gap_remove_tail}
	Let~${\alpha}$ be a well order. Let~${s_l, s_r, t_l, t_r \in \gaps_\alpha}$ be sequences such that~${t_r}$ is empty or begins with an element that lies strictly below all elements in~${s_l}$. Then, the inequality~${s_l * s_r \leq_\strong t_l * t_r}$ implies~${s_l \leq_\strong t_l}$.
\end{lemma}

\begin{proof}
	If~${s_l}$ is empty, our claim holds trivially. Assume that~${s_l = \langle s_0 \rangle * s'_l}$ holds for some sequence~${s'_l \in \gaps_\alpha}$. Assume that~${t_l}$ is empty. Should~${t_r}$ also be empty, then~${s_l * s_r \leq_\strong \langle \rangle}$ must lead to a contradiction since~${s_l}$ is nonempty. Otherwise, if~${t_r}$ is nonempty, our assumption yields~${(t_r)_0 < s_0}$. This clearly contradicts~${s_l * s_r \leq_\strong t_r}$.
	
	Now, let~${t_l}$ be of the form~${t_l = \langle t_0 \rangle * t'_l}$ for some sequence~${t'_l \in \gaps_\alpha}$. We proceed by induction on the recursive definition of~${s_l * s_r \leq_\strong t_l * t_r}$. We already excluded the case where~${s_l * s_r}$ is empty. Therefore, assume that~${s_0 \leq t_0}$ and~${s_l * s_r \leq_\strong t'_l * t_r}$ hold. By induction hypothesis, this yields~${s_l \leq_\strong t'_l}$. With~${s_0 \leq t_0}$, we arrive at our claim. Finally, assume that~${s_0 \leq t_0}$ and~${s'_l * s_r \leq_\strong t'_l * t_r}$ hold. By induction hypothesis, this results in~${s'_l \leq_\strong t'_l}$ and, with~${s_0 \leq t_0}$, our claim.
\end{proof}

\begin{lemma}\label{lem:gap_remove_head}
	Let~${\alpha}$ be a well order and let~${s, t, u \in \gaps_\alpha}$ be sequences such that~${t}$ is empty or begins with an element that lies strictly below all elements in~${u}$. Then, the inequality~${u * s \leq_\strong u * t}$ implies~${s \leq_\strong t}$.
\end{lemma}

\begin{proof}
	We prove a slightly stronger claim: Let~${u, v \in \gaps_\alpha}$ be sequences with~${|u| \geq |v|}$ such that~${t}$ is either empty or begins with an element that lies strictly below all elements in~${u}$. Then, we claim that~${u * s \leq_\strong v * t}$ implies~${s \leq_\strong t}$.
	
	First, if~${t}$ is empty, then a short induction reveals that the same most hold for~${s}$ (since, in this case,~${u * s}$ is a longer sequence than~${v * t}$) and our claim follows trivially. Therefore, we continue under the assumption that~${t}$ is nonempty.
	
	If~${u}$ is empty, then the same must hold for~${v}$. Thus, our claim is trivial. Now, assume that~${u}$ is nonempty and of the form~${u = \langle u_0 \rangle * u'}$ for some sequence~${u' \in \gaps_\alpha}$. We proceed by induction on the recursive definition of~${u * s \leq_\strong v * t}$: Clearly, we can skip the case where~${u * s}$ is empty since we already assume that~${u}$ is nonempty. Moreover, we can continue under the assumption that~${v}$ is nonempty since, otherwise, both remaining cases demand~${u_0 \leq t_0}$, which is a clear violation of the conditions that we impose on~${t}$. Let~${v}$, therefore, be of the form~${v = \langle v_0 \rangle * v'}$ for some sequence~${v' \in \gaps_\alpha}$. Assume that~${u * s \leq_\strong v * t}$ holds because of~${u_0 \leq v_0}$ and~${u * s \leq_\strong v' * t}$. We have~${|u| \geq |v| > |v'|}$. Thus, our induction hypothesis yields our claim. Finally, assume that the inequality holds because of both~${u_0 \leq v_0}$ and~${u' * s \leq_\strong v' * t}$. Similarly, we have~${|u'| = |u| - 1 \geq |v| - 1 = |v'|}$. Again, by induction hypothesis, we arrive at our claim.
\end{proof}

\begin{lemma}\label{lem:gap_remove_head_weak}
	Let~${\alpha}$ be a well order and let~${s, t, u \in \gapw_\alpha}$ be sequences such that the inequality~${u * s \leq_\weak u * t}$ holds. Then, this implies~${s \leq_\weak t}$.
\end{lemma}

\begin{proof}
	Again, we prove a slightly stronger claim:
	We begin by showing that the relation~${u * s \leq_\strong v * t}$ for sequences~${s, t, u, v \in \gapw_\alpha}$ with~${|u| \geq |v|}$ entails~${s \leq_\weak t}$. If~${u}$ is empty, then the same must hold for~${v}$ and our claim is trivial. Also, if~${v}$ is empty, we arrive at our claim via~${s \leq_\weak u * s \leq_\strong t}$. Thus, we can assume both~${u}$ and~${v}$ to be nonempty, i.e., they are of the form~${u = \langle u_0 \rangle * u'}$ and~${v = \langle v_0 \rangle * v'}$ for sequences~${u', v' \in \gapw_\alpha}$. We proceed by induction on the recursive definition of our assumed inequality.
	
	We can skip the case where~${u * s}$ is empty since we already handled the case where~${u}$ is empty. Thus, assume that~${u_0 \leq v_0}$ and~${u * s \leq_\strong v' * t}$ hold. We have the strict inequality~${|u| \geq |v| > |v'|}$ and, therefore, our claim by induction hypothesis. Similarly, assume that~${u_0 \leq v_0}$ and~${u' * s \leq_\strong v' * t}$ hold. Then, we have the inequality~${|u'| = |u| - 1 \geq |v| - 1 \geq |v'|}$ and, again, our claim follows from the induction hypothesis.
	
	Finally, assume that~${u * s \leq_\weak u * t}$ holds. By Lemma~\ref{lem:gap_strong_and_weak}, this implies the relation~${\langle 0 \rangle * u * s \leq_\strong \langle 0 \rangle * u * t}$. Therefore, our previous considerations yield~${s \leq_\weak t}$.
\end{proof}

\section{Simplified construction for~${\atr_0}$}\label{sec:ATR}

In this section, we extend the following result due to Gordeev (cf.~\cite{Gordeev}): Over the system~${\rca_0}$, the order~${\gapg_{\alpha}}$ is a well partial order for any well order~${\alpha}$ if and only if the principle of \emph{arithmetical transfinite recursion} holds. The definition of this axiom and more on the ensuing system~${\atr_0}$ can be found in \cite[Chapter~V]{Simpson}. For our purposes, we work directly with an alternative characterization using the \emph{Veblen hierarchy}. The following definition closely follows the one given in \cite{RW11}:
\begin{definition}
	Given a well order~${\alpha}$, we define a linear order~${\varphi(\alpha, 0)}$. We begin with the terms of its underlying set, which we define together with a function~${h: \varphi(\alpha, 0) \to \alpha}$ and a set~${H \subseteq \varphi(\alpha, 0)}$:
	\begin{enumerate}[label=\roman*)]
		\item ${0 \in \varphi(\alpha, 0)}$, with~${h(0) := 0}$ and~${0 \notin H}$,
		\item ${\varphi(\beta, x) \in \varphi(\alpha, 0)}$ for~${\beta \in \alpha}$ and~${x \in \varphi(\alpha, 0)}$ if~${h(x) \leq \beta}$; with~${h(\varphi(\beta, x)) := \beta}$ and~${\varphi(\beta, x) \in H}$,
		\item ${x_0 + \dots + x_{n-1}}$ for~${x_i \in \varphi(\alpha, 0)}$ and~${n \geq 2}$ satisfying~${x_i \geq x_j}$ for~${i < j < n}$; with~${h(x_0 + \dots + x_{n-1}) := 0}$ and~${x_0 + \dots + x_{n-1} \notin H}$.
	\end{enumerate}
	We call terms in~${H}$ \emph{indecomposable}. All other terms are \emph{decomposable}.
	Using simultaneous recursion, the order on~${\varphi(\alpha, 0)}$ is given as follows:
	\begin{enumerate}[label=\roman*)]
		\item ${0 < x}$ for all~${x \in \varphi(\alpha, 0)}$ with~${x \neq 0}$,
		\item ${x_0 + \dots + x_{n-1} < y_0 + \dots + y_{m-1}}$ if and only if~${n < m}$ as well as~${x_i = y_i}$ for all~${i < n}$, or there is some~${i < \min(n, m)}$ with~${x_i < y_i}$ and~${x_j = y_j}$ for all~${j < i}$,
		\item ${x_0 + \dots + x_{n-1} < \varphi(\beta, y)}$ if and only if~${x_0 < \varphi(\beta, y)}$,
		\item ${\varphi(\beta, x) < y_0 + \dots + y_{m-1}}$ if and only if~${\varphi(\beta, x) \leq y_0}$,
		\item ${\varphi(\beta, x) < \varphi(\gamma, y)}$ if and only if one of the following holds:
		\begin{itemize}
			\item ${\beta < \gamma}$ and~${x < \varphi(\gamma, y)}$,
			\item ${\beta = \gamma}$ and~${x < y}$,
			\item ${\beta > \gamma}$ and~${\varphi(\beta, x) < y}$.
		\end{itemize}
	\end{enumerate}
\end{definition}
The definition of~${\atr_0}$ that we will employ in this article is then given by the following theorem (cf.~\cite{FriedmanATR, RW11}) due to Friedman:
\begin{theorem}\label{thm:atr_veblen}
	The following are equivalent over~${\rca_0}$:
	\begin{enumerate}[label=\alph*)]
		\item Arithmetical transfinite recursion,
		\item ${\varphi(\alpha, 0)}$ is well-founded for any well order~${\alpha}$.
	\end{enumerate}
\end{theorem}
This result is essential for the characterization of arithmetical transfinite recursion using sequences with gap condition. It should be noted, however, that it was only proved after Gordeev's initial work on such sequences. To be more precise, while Gordeev showed that arithmetical transfinite recursion proves all partial orders of sequences with gap condition to be well, he did not show the other direction in its full strength but restricted himself to $\Pi^1_1$-consequences. In light of Theorem~\ref{thm:atr_veblen}, the full equivalence follows immediately from Gordeev's work. Hence, we still attribute this result to Gordeev.

The following theorem extends Gordeev's result to the orders~${\gapg_{\alpha}}$,~${\gapw_{\alpha}}$,~${\gaps_{\alpha}}$,~${\btal_{\alpha, 1}}$, and~${\bta_{\alpha, 1}}$ for well orders~${\alpha}$:
\begin{theorem}\label{thm:ATR_equivalent_gap_sequences_and_trees}
	The following are equivalent over~${\rca_0}$:
	\begin{enumerate}[label=\alph*)]
		\item Arithmetical transfinite recursion,
		\item ${\gapg_\alpha}$ is a well partial order for any well order~${\alpha}$,
		\item ${\gapw_\alpha}$ is a well partial order for any well order~${\alpha}$,
		\item ${\gaps_\alpha}$ is a well partial order for any well order~${\alpha}$,
		\item ${\btal_{\alpha, 1}}$ is a well partial order for any well order~${\alpha}$,
		\item ${\bta_{\alpha, 1}}$ is a well partial order for any well order~${\alpha}$.
	\end{enumerate}
\end{theorem}
As already mentioned during the introduction, in the author's PhD thesis, this result is also proved for a possible generalization of~${\gapg_{X}}$,~${\gapw_{X}}$,~${\gaps_{X}}$,~${\btal_{X, Y}}$, and~${\bta_{X, Y}}$ for (proper) well partial orders~$X$ and~$Y$ (cf.~\cite[Theorem~3.3.1]{UftringPhD}).

Similar results can also be derived if one considers arbitrary finite trees but with a different embeddability relation (that does not preserve infima). In fact, Friedman and Weiermann proved a characterization for such trees that also involves the Veblen hierarchy (cf.~\cite{FriedmanWeiermann}). Similar questions are currently being studied by Chopra and Pakhomov (cf.~\cite{ChopraPakhomov}). Moreover, a characterization with respect to better quasi orders for such trees with strictly ascending labels is being investigated by Alakh Dhruv Chopra, Fedor Pakhomov, Philipp
Provenzano, and Giovanni Sold\`{a}.

Before we continue, we need to introduce a further order:
\begin{definition}
    Given a well order~$\alpha$, we define a linear order~$\omega^\alpha$. It contains the terms
    \begin{equation*}
        x := \omega^{x_0} + \dots + \omega^{x_{n-1}}
    \end{equation*}
    for~${n \in \n}$ if~${x_i \in \omega^\alpha}$ for all~${i < n}$ and~${x_i \geq x_j}$ for all~${i < j < n}$. If~${n = 0}$ holds, we usually write~${x = 0}$ instead of an empty sum.
    Simultaneously, we define the order using a second term
    \begin{equation*}
        y := \omega^{y_0} + \dots + \omega^{y_{m-1}}\text{.}
    \end{equation*}
    We have~${x < y}$ if and only if~${n < m}$ holds together with~${x_i = y_i}$ for all~${i < n}$, or there exists some common~${i < \min(n, m)}$ with~${x_i < y_i}$ as well as~${x_j = y_j}$ for all~${j < i}$.
\end{definition}
Similar to the characterization of~$\atr_0$ via the Veblen hierarchy, it is a classical result due to Girard (cf.~\cite[Section~II.5]{Girard}, see also~\cite[Theorem~2.6]{Hirst}) that~$\aca_0$, i.e.~$\rca_0$ together with arithmetical comprehension, can be characterized using orders~$\omega^\alpha$ for well orders~$\alpha$.
\begin{definition}
	For any well order~${\alpha}$, we define the sum~${x + y}$ of any two elements~${x, y \in \omega^\alpha}$. Given~${x := \omega^{x_0} + \dots + \omega^{x_{n-1}}}$ and~${y := \omega^{y_0} + \dots + \omega^{y_{m-1}}}$, we set
	\begin{equation*}
		x + y := \omega^{x_0} + \dots + \omega^{x_i} + \omega^{y_0} + \dots + \omega^{y_{m-1}}\comma
	\end{equation*}
	where~${i < n}$ is the largest index such that~${x_i \geq y_0}$ holds. If~${y}$ is~${0}$, we set~${i := n-1}$. Otherwise, we take~${i := -1}$ if no such index exists.

    The same definition can be applied to~${x, y \in \varphi(\alpha, 0)}$. In this case, we write indecomposable elements~${\varphi(0, z)}$ and~${\varphi(\beta, z)}$ for ${\beta > 0}$ (where~$0$ is the smallest element in~$\alpha$) as~${\omega^z}$ and~$\omega^{\varphi(\beta, z)}$, respectively. Decomposable elements are then written as sums of such~$\omega$-terms.
\end{definition}
Addition enjoys the usual properties of ordinal arithmetic as associativity and, e.g., one can easily show in~${\rca_0}$ that~${x < y}$ implies~${z + x < z + y}$ for~${x, y, z \in \omega^\alpha}$ and any well order~${\alpha}$. (The same holds, if we consider~${x, y, z \in \varphi(\alpha, 0)}$.)

We derive Theorem~\ref{thm:ATR_equivalent_gap_sequences_and_trees} using several propositions and lemmas.
The following first step matches a similar result by Gordeev (cf.~\cite[Section~II.4]{Gordeev}):
\begin{proposition}\label{prop:embed_varphi_into_weak_gap_sequence}
	The system~$\rca_0$ proves that for any well order~${\alpha}$, there is a quasi-embedding from~${\varphi(\alpha, 0)}$ into~${\gapw_{\omega^\alpha}}$.
\end{proposition}
For this proof, we introduce new notation:
\begin{definition}
	Consider a well order~${\alpha}$ on which addition is defined. We may write~${\beta + s}$ for any~${\beta \in \alpha}$ and~${s \in \gapw_\alpha}$ in order to denote the sequence in~${\gapw_\alpha}$ that results from~${s}$ if each member~${\gamma}$ in~${s}$ is replaced by~${\beta + \gamma}$.
\end{definition}
\begin{lemma}\label{lem:left_addition_ATR}
	The system~$\rca_0$ proves the following:
	Let~${\alpha}$ be a well order. Consider~${s, t \in \gapw_{\omega^\alpha}}$ and~${\beta, \gamma \in \omega^\alpha}$.
	\begin{enumerate}[label=\alph*)]
		\item ${\beta + s \leq_\weak \beta + t}$ implies~${s \leq_\weak t}$.
		\item ${\omega^\beta + s \leq_\weak \gamma + t}$ with~${\omega^\beta > \gamma}$ implies~${\omega^\beta + s \leq_\weak t}$.
	\end{enumerate}
\end{lemma}
\begin{proof}
	By definition of sequences with (weak) gap condition, it suffices to prove these with~${s}$ and~${t}$ replaced by elements from~${\omega^\alpha}$:
	\begin{enumerate}[label=\alph*)]
		\item Assume~${\beta + \lambda \leq \beta + \rho}$ for~${\beta, \lambda, \rho \in \omega^\alpha}$. Assume, for contradiction, that the strict inequality~${\lambda > \rho}$ holds. It is easy to show that addition on~${\omega^\alpha}$ entails the contradiction~${\beta + \lambda > \beta + \rho}$.
		\item Assume~${\omega^\beta + \lambda \leq \gamma + \rho}$ for~${\beta, \gamma, \lambda, \rho \in \omega^\alpha}$. Assume, for contradiction, that the inequality~${\omega^\beta + \lambda > \rho}$ holds. Hence,~${\gamma + \omega^\beta + \lambda = \omega^\beta + \lambda}$ holds by definition of addition. Thus, we arrive at~${\omega^\beta + \lambda = \gamma + \omega^\beta + \lambda > \gamma + \rho}$.\qedhere
	\end{enumerate}
\end{proof}

\begin{proof}[Proof of Proposition~\ref{prop:embed_varphi_into_weak_gap_sequence}]
	It~${\alpha}$ is empty, the claim is trivial. Assume that~${\alpha}$ is nonempty.
	We define a map~${f: \varphi(\alpha, 0) \to \gapw_{\omega^{\alpha}}}$ as follows:
	\begin{equation*}
		f(x) :=
		\begin{cases*}
			\langle 0 \rangle & \text{if~${x = 0}$,}\\
			f(x_0) * \langle 0 \rangle * f(x_1 + \dots + x_{n-1}) & \text{if~${x = x_0 + \dots + x_{n-1}}$,}\\
			\omega^\beta + f(y) & \text{if~${x = \varphi(\beta, y)}$.}
		\end{cases*}
	\end{equation*}
	We prove that~${f}$ is a quasi-embedding: Let~${x, y \in \varphi(\alpha, 0)}$ be such that~${f(x) \leq_\weak f(y)}$ holds. By induction along the combined amount of symbols in~${x}$ and~${y}$, we prove that this entails~${x \leq y}$.
	For~${y = 0}$, a short induction reveals that this entails~${x = 0}$. In the following, we assume that both~${x}$ and~${y}$ are different from~${0}$.
	
	Assume that~${x = x_0 + \dots + x_{n-1}}$ and~${y = y_0 + \dots + y_{m-1}}$ are decomposable. By Lemma~\ref{lem:gap_split_weak_weak_or_strong}, we can find two sequences~${s_l}$ and~${s_r}$ with~${s_l * s_r = f(x)}$ such that~${s_l \leq_\weak f(y_0)}$ and~${s_r \leq_\weak \langle 0 \rangle * f(y_1 + \dots + y_{m-1})}$ hold. If~${s_l}$ is empty, then the inequality~${f(x) \leq_\weak \langle 0 \rangle * f(y_1 + \dots + y_{m-1})}$ must hold. Moreover,~${f(x)}$ begins with a positive member since~${x_0}$ is indecomposable and, thus, a short induction reveals that~${f(x_0)}$ must be nonempty and only consist of positive members. Therefore, we conclude that already~${f(x) \leq_\weak f(y_1 + \dots + y_{m-1})}$ holds. By induction hypothesis, this leads to~${x \leq y_1 + \dots + y_{m-1} < y}$. Otherwise, if~${s_r}$ is empty, we have the relation~${f(x) \leq_\weak f(y_0)}$. By induction hypothesis, this entails~${x \leq y_0 < y}$. Now if both~${s_l}$ and~${s_r}$ are nonempty, Lemma~\ref{lem:gap_split_weak_weak_or_strong} already yields~${s_r \leq_\strong \langle 0 \rangle * f(y_1 + \dots + y_{m-1})}$ with respect to the strong gap condition. By definition, this means that~${s_r}$ must begin with a~${0}$. Hence,~${f(x_0)}$ is an initial segment of~${s_l}$, i.e.,~${f(x_0) \leq_\weak s_l \leq_\weak f(y_0)}$ holds. By induction hypothesis, this yields~${x_0 \leq y_0}$. If the inequality is strict, i.e.~if we have~${x_0 < y_0}$, we are already done. Otherwise, assume~${x_0 = y_0}$. This implies~${f(x_0) * \langle 0 \rangle = f(y_0) * \langle 0 \rangle}$. Now, we can apply Lemma~\ref{lem:gap_remove_head_weak}, which results in~${f(x_1 + \dots + x_{n-1}) \leq_\weak f(y_1 + \dots + y_{m-1})}$. By induction hypothesis, we have the inequality~${x_1 + \dots + x_{n-1} \leq y_1 + \dots + y_{m-1}}$ and, with~${x_0 = y_0}$, our claim.
	
	Assume that~${x = x_0 + \dots + x_{n-1}}$ is decomposable but~${y = \varphi(\beta, y')}$ is not. We have~${f(x_0) <_\weak f(x) \leq_\weak f(y)}$. By induction hypothesis, this yields~${x_0 < y}$. Now, by definition of the order on~${\varphi(\alpha, 0)}$, we conclude~${x < y}$.
	
	Assume that~${x = \varphi(\beta, x')}$ is indecomposable but~${y = y_0 + \dots + y_{m-1}}$ is decomposable. Similar to before, Lemma~\ref{lem:gap_split_weak_weak_or_strong} produces two sequences~${s_l}$ and~${s_r}$ satisfying~${s_l * s_r = f(x)}$ such that~${s_l \leq_\weak f(y_0)}$ and~${s_r \leq_\weak \langle 0 \rangle * f(y_1 + \dots + y_{m-1})}$ hold. Again, if~${s_l}$ is empty, then~${f(x) \leq_\weak \langle 0 \rangle * f(y_1 + \dots + y_{m-1})}$ already yields the inequality~${f(x) \leq_\weak f(y_1 + \dots + y_{m-1})}$ since~${f(x)}$ does not contain any zeros. By induction hypothesis, we conclude~${x \leq y_1 + \dots + y_{m-1} < y}$. Otherwise, if~${s_r}$ is empty, then~${f(x) \leq_\weak f(y_0)}$ holds and, by induction hypothesis, we get~${x \leq y_0 < y}$. Finally, if~${s_l}$ and~${s_r}$ are both nonempty, Lemma~\ref{lem:gap_split_weak_weak_or_strong} tells us that~${s_r}$ must begin with a zero. However,~${f(x)}$ only consists of positive members, i.e., this leads to a contradiction.
	
	Assume that~${x = \varphi(\beta, x')}$ and~${y = \varphi(\gamma, y')}$ are both indecomposable. We continue by case distinction on the relation between~${\beta}$ and~${\gamma}$. If~${\beta < \gamma}$ holds, then we have~${f(x') \leq_\weak f(x) \leq_\weak f(y)}$. By induction hypothesis, this results in~${x' \leq y}$. Now, by the restrictions imposed on terms in~${\varphi(\alpha, 0)}$ ($x'$ must be~${0}$, decomposable, or a~${\varphi}$-term with a value less of equal to~${\beta}$ in the first argument) it cannot be that~${x' = y}$ holds. We conclude~${x' < y}$. If~${\beta = \gamma}$ holds, then Lemma~\ref{lem:left_addition_ATR}~a) yields~${f(x') \leq_\weak f(y')}$ and, by induction hypothesis,~${x' \leq y'}$. Finally, if~${\beta > \gamma}$ holds, then Lemma~\ref{lem:left_addition_ATR}~b) yields~${f(x) \leq_\weak f(y')}$ and, by induction hypothesis,~${x \leq y'}$. Again, we have~${x < y'}$ since the equality~${x = y'}$ cannot hold by the restrictions imposed on terms in~${\varphi(\alpha, 0)}$.  In all three cases, we arrive at~${x < y}$.
\end{proof}
We continue with quasi-embeddings between orders with weak and strong gap condition. In general, we can only embed~${\gaps_\alpha}$ into~${\gapw_\beta}$ if~${\alpha}$ is strictly smaller than~${\beta}$. This is not a problem for our theorem, however, since our claims are expressed uniformly for all well orders.
\begin{lemma}\label{lem:embed_weak_into_strong_and_strong_into_weak}
	The system~$\rca_0$ proves that for any well order~${\alpha}$, there are quasi-embeddings~${\gapw_{\alpha} \to \gaps_{\alpha}}$ and~${\gaps_{\alpha} \to \gapw_{\alpha + 1}}$.
\end{lemma}
\begin{proof}
	The quasi-embedding from~${\gapw_{\alpha}}$ to~${\gaps_{\alpha}}$ is given by the identity since any sequences~${s, t \in \gaps_{\alpha}}$ with~${s \leq_\strong t}$ also satisfy~${s \leq_\weak t}$.
	
	Let~${\top}$ denote the top element of~${\alpha + 1}$. Then, we consider~${f: \gaps_{\alpha} \to \gapw_{\alpha + 1}}$ defined by~${f(s) := \langle \top \rangle * s}$. We show that~${f}$ is a quasi-embedding: Assume that we have two sequences~${s, t \in \gaps_{\alpha}}$ with~${f(s) \leq_\weak f(t)}$, i.e.,~${\langle 0, \top \rangle * s \leq_\strong \langle 0, \top \rangle * t}$, where~${0}$ is the smallest element in~${\alpha}$. (If~${\alpha}$ is empty, then our claim is trivial.) Now, this inequality entails one of two possibilities: If~${\langle \top \rangle * s \leq_\strong \langle \top \rangle * t}$ holds, then this yields our claim~${s \leq_\strong t}$ by Lemma~\ref{lem:gap_remove_head}. Otherwise, if~${\langle 0, \top \rangle * s \leq_\strong \langle \top \rangle * t}$ holds, then this results, even with respect to Higman's order without any gap condition, in a contradiction since~${\top}$ is strictly greater than any element in~${t}$.
\end{proof}
Next, we embed sequences with weak gap condition into binary trees with weakly ascending labels that strictly ascend for left subtrees. This makes precise what we sketched during the introduction.
\begin{lemma}\label{lem:embed_weak_into_btal}
	The system~$\rca_0$ proves that for any well order~${\alpha}$, there is a quasi-embedding~${\gapw_\alpha \to \btal_{\alpha, 1}}$.
\end{lemma}
Note that this result is also valid in sufficiently strong set theories (which allows us to reuse it in Section~\ref{sec:order_types}). In particular, it also holds with respect to uncountable well orders~$\alpha$. It is important to state this explicitly since it could potentially be the case that our arguments make use of the fact that, in the context of second-order arithmetic, all considered linear orders are countable.
\begin{proof}[Proof of Lemma~\ref{lem:embed_weak_into_btal}]
	We define a map~${f: \gapw_{\alpha} \to \btal_{\alpha, 1}}$ with
	\begin{equation*}
		f(s) :=
		\begin{cases*}
			0 \star [] & \text{if~${s = \langle \rangle}$,}\\
			\beta \star [f(s_l), f(s_r)] & \parbox[t]{20em}{if~${s = s_l * \langle \beta \rangle * s_r}$, where~${\beta}$ is minimal in~${s}$ and~${s_l}$ is as short as possible.}
		\end{cases*}
	\end{equation*}
	A short induction reveals that all inner nodes in~${f(s)}$ must be members of~${s}$, for any sequence~${s \in \gapw_{\alpha}}$. Since, in the second case,~${s_l}$ must be as short as possible, we know that all inner nodes in the left subtree of~${f(s)}$ must be strictly greater than~${\beta}$ and all inner nodes in the right subtree of~${f(s)}$ must be greater than or equal to~$\beta$. Thus,~${f}$ maps to elements in~${\btal_{\alpha, 1}}$.
	
	We show that~${f}$ is a quasi-embedding: Let~${s, t \in \gapw_{\alpha}}$ be such that~${f(s) \leq f(t)}$ holds. Using induction along~${|s| + |t|}$, we prove that this implies~${s \leq_\weak t}$. If~${t = \langle \rangle}$ holds, this immediately implies~${s = \langle \rangle}$. Now, we can assume that~${s = s_l * \langle \beta \rangle * s_r}$ and~${t = t_l * \langle \gamma \rangle * t_r}$ hold for sequences~${s_l, s_r, t_l, t_r \in \gapw_{\alpha}}$ and elements~${\beta, \gamma \in \alpha}$. If the inequality~${f(s) \leq f(t)}$ holds since~${f(s)}$ embeds into one of~${f(t_l)}$ or~${f(t_r)}$, then we have~${s \leq_\weak t_l}$ or~${s \leq_\weak t_r}$, by induction hypothesis. This results in our claim. Otherwise, if~${f(s)}$ does not embed into one of the subtrees of~${f(t)}$, then we have the inequalities~${\beta \leq \gamma}$,~${f(s_l) \leq f(t_l)}$, and~${f(s_r) \leq f(t_r)}$. By induction hypothesis, we get both~${s_l \leq_\weak t_l}$ and~${s_r \leq_\weak t_r}$. With Lemma~\ref{lem:gap_strong_and_weak}, we have~${\langle \beta \rangle * s_r \leq_\strong \langle \gamma \rangle * t_r}$. Finally, with Lemma~\ref{lem:gap_concat_weak_and_strong}, we arrive at our claim since~${\beta}$ is less than any element in~${t_l}$.
\end{proof}
In order to show that the well-foundedness of linear orders implies that our trees are well partial orders, simple quasi-embeddings do not suffice. For this, let us introduce a useful tool (cf.~\cite[p.~85]{SchuetteSimpson} and~\cite[Definition~4.1]{SimpsonHilbertBasis}): \emph{reifications}.
\begin{definition}
	Given a partial order~${X}$, we say that a finite sequence~${s \in X^*}$ is \emph{bad} if~${s_i \nleq s_j}$ holds for any two indices~${i < j < |s|}$. We write~${\badfin(X)}$ for the set of nonempty finite bad sequences in~${X}$. Given a well order~${\alpha}$, we call a map~${r: \badfin(X) \to \alpha}$ a \emph{reification} from~${X}$ into~${\alpha}$ if and only if~${r(s) > r(t)}$ holds for any two sequences~${s, t \in \badfin(X)}$ such that~${t}$ is a proper extension of~${s}$.
\end{definition}
Reifications can be used to prove that a given order is a well partial order (cf.~\cite[p.~85]{SchuetteSimpson} and~\cite[Lemma~4.2]{SimpsonHilbertBasis}):
\begin{lemma}
	The system~$\rca_0$ proves the following:
	Let~${X}$ be a partial order and let~${\alpha}$ be a well order such that there exists a reification from~${X}$ into~${\alpha}$. Then,~${X}$ is a well partial order.
\end{lemma}
\begin{proof}
	Assume that~${(x_n)_{n \in \n}}$ is an infinite bad sequence in~${X}$. Then, the sequence~${(r(\langle x_0, \dots, x_{n} \rangle))_{n \in \n}}$ descends infinitely in~${\alpha}$. This contradicts the assumption that~${\alpha}$ is well-founded.
\end{proof}
Using this tool, we can prove the following proposition:
\begin{proposition}\label{prop:reification_b_alpha_1_into_varphi_2+alpha_1}
	The system~$\rca_0$ proves that for any well order~${\alpha}$, there is a reification from~${\bta_{\alpha, 1}}$ into~${\varphi(2 + \alpha, 0)}$.
\end{proposition}
The proof of this last proposition will be our subject for the rest of this section. First, however, we use this and all the other results that we have gathered so far in order to prove our theorem:
\begin{proof}[Proof of Theorem~\ref{thm:ATR_equivalent_gap_sequences_and_trees}]
	If arithmetical transfinite recursion is available, then~${\varphi(\alpha, 0)}$ is well-founded for any well order~${\alpha}$ using Theorem~\ref{thm:atr_veblen}. Thus,~a) implies~f) using Proposition~\ref{prop:reification_b_alpha_1_into_varphi_2+alpha_1}. Since~${\btal_{\alpha, 1}}$ is a suborder of~${\bta_{\alpha, 1}}$ for any linear order~${\alpha}$,~f) implies~e). Next, with Lemma~\ref{lem:embed_weak_into_btal}, we know that~e) entails~c). Statements~c) and~d) are equivalent by Lemma~\ref{lem:embed_weak_into_strong_and_strong_into_weak}, statements~b) and~c) are equivalent by Lemma~\ref{lem:gapg_gapw_are_identical}. Finally, for the direction from~c) to~a): If~c) holds\footnote{To be very precise,~c) first implies that~${\omega^\alpha}$ is well-founded for any well order~$\alpha$. Hence,~$\gapw_{\omega^\alpha}$ is a well partial order for any well order~$\alpha$. Then, we apply Proposition~\ref{prop:embed_varphi_into_weak_gap_sequence}.}, then~${\varphi(\alpha, 0)}$ is well-founded for any well order~${\alpha}$ by Proposition~\ref{prop:embed_varphi_into_weak_gap_sequence}, which yields arithmetical transfinite recursion again using Theorem~\ref{thm:atr_veblen}.
\end{proof}

For the rest of this section, we will construct the reification claimed in Proposition~\ref{prop:reification_b_alpha_1_into_varphi_2+alpha_1}. We closely follow the construction of a \emph{primitive recursive} reification of finite binary trees into~${\varepsilon_0}$ by Freund in \cite[Section~6]{FreundReification}, which was influenced by \cite{Hasegawa}. Originally, the result on the maximal order types of finite binary trees is due to de Jongh and Schmidt (cf.~the introduction of \cite{SchmidtBounds} and \cite{SchmidtHabilitation}), effective reifications for finitely branching trees were given by Rathjen and Weiermann (cf.~\cite{RW93}).

It should be noted that there also exist reifications of~${\bta_{\alpha, 1}}$ into smaller well orders such as~${\varphi(1 + \alpha, 0)}$. However, the reification into~${\varphi(2 + \alpha, 0)}$ is much simpler conceptually, since it avoids a distinction between decomposable and indecomposable types. Moreover, in the next section we will see that for some well orders~${\alpha}$, we can even find reifications into well orders that are strictly smaller than~${\varphi(1 + \alpha, 0)}$, when we compute all maximal order types precisely.

For the following constructions,~${\alpha}$ always denotes a well order. We begin by defining~\emph{types}. These can be understood as names for partial orders.
\begin{definition}
	We define the following types:
	\begin{enumerate}[label=\roman*)]
		\item ${\labeltype_\beta}$ and~${\emptytype}$ are types for any~${\beta \in \alpha}$.
		\item If~${A}$ is a type and~${\beta \in \alpha}$, then~${\treetype_{\beta, A}}$ is a type.
		\item If~${A}$ and~${B}$ are types, then~${A + B}$ is a type.
		\item If~${A}$ and~${B}$ are types, then~${A \times B}$ is a type.
		\item If~${A}$ is a type, then~${A^*}$ is a type.
	\end{enumerate}
\end{definition}
As usual, products~${\times}$ enjoy a higher precedence than sums~${+}$.
Each type (except for the empty one(s)) is inhabited by \emph{terms}. These are the elements of the partial orders that our types are the names of.
\begin{definition}
	We define the terms of each type~${A}$ together with a height function~${h_A: A \to \n}$.
	\begin{enumerate}[label=\roman*)]
		\item Let~${\beta \in \alpha}$ be arbitrary. Each~${\gamma \in \alpha}$ with~${\gamma < \beta}$ lies in~${\labeltype_\beta}$ with~${h(\gamma) := 0}$.
		\item Let~${\beta \in \alpha}$ be arbitrary and let~${A}$ be a type. Each tree in~${\bta_{\beta, A}}$ is an element\footnote{In~${\bta_{\beta, A}}$, we already interpret~${A}$ as the partial order of terms (of type~${A}$) that we define in the next definition. This can be made precise by considering both definitions as part of the same simultaneously recursive definition.} of~${\treetype_{\beta, A}}$. We set~${h(a \star []) := h(a) + 1}$ for~${a \in A}$. Moreover, given~${\gamma < \beta}$ as well as~${t_l, t_r \in \bta_{\beta, A}}$, we define~${h(\gamma \star [t_l, t_r]) := h(t_l) + h(t_r) + 1}$.
		\item If~${a}$ or~${b}$ are terms of type~${A}$ or~${B}$, respectively, then~${\iota_0(a)}$ or~${\iota_1(b)}$ are terms of type~${A + B}$ with~${h(\iota_0(a)) := h(a) + 1}$ or~${h(\iota_1(b)) := h(b) + 1}$.
		\item If~${a}$ and~${b}$ are terms of type~${A}$ and~${B}$, respectively, then~${\langle a, b \rangle}$ is a term of type~${A \times B}$ with~${h(\langle a, b \rangle) := h(a) + h(b) + 1}$.
		\item If~${a_0, \dots, a_{n-1}}$ are terms of type~${A}$, then~${a := \langle a_0, \dots, a_{n-1} \rangle}$ is a term of type~${A^*}$ with~${h(a) := \max\{a_i \mid i < n\} + 1}$.
	\end{enumerate}
\end{definition}
Of course, in order to arrive at partial orders, we require relations on our terms:
\begin{definition}
	For any term~${A}$, we define a relation~${\leq_A}$ on its elements:
	\begin{enumerate}[label=\roman*)]
		\item Let~${\beta \in \alpha}$ be arbitrary. Given~${\gamma, \delta < \beta}$, we have~${\gamma \leq_{\labeltype_\beta} \delta}$ if~${\gamma \leq \beta}$ holds.
		\item Let~${\beta \in \alpha}$ be arbitrary. Given trees~${s, t \in \bta_{\beta, A}}$, we have~${s \leq_{\treetype_{\beta, A}} t}$ if~${s \leq t}$ holds.
		\item If~${a \leq_{A} a'}$ or~${b \leq_{B} b'}$ hold, we have~${\iota_0(a) \leq_{A + B} \iota_0(a')}$ or~${\iota_1(b) \leq_{A + B} \iota_1(b')}$, respectively.
		\item If~${a \leq_{A} a'}$ and~${b \leq_{B} b'}$ hold, then we have~${\langle a, b \rangle \leq_{A \times B} \langle a', b' \rangle}$.
		\item Let~${a_0, \dots, a_{n-1}}$ and~${a'_0, \dots, a'_{m-1}}$ be terms of type~${A}$ such that there exists a strictly increasing function~${f: n \to m}$ with~${a_i \leq_A a'_{f(i)}}$ for all~${i < n}$. Then we have~${\langle a_0, \dots, a_{n-1} \rangle \leq_{A^*} \langle a'_0, \dots, a'_{m-1} \rangle}$.
	\end{enumerate}
\end{definition}
\begin{lemma}
	The system~$\rca_0$ proves that for any term~${A}$, the relation~${\leq_A}$ is a partial order.
\end{lemma}
\begin{proof}
	This is clear for types~${\labeltype_{\beta}}$ and~${\treetype_{\beta, A}}$ for any~${\beta \in \alpha}$ and any type~${A}$. It is not hard to see that our claim also holds for all other types. For more details, we refer to \cite{FreundReification}.
\end{proof}
Given a type~${A}$ and a term~${a}$ of~${A}$, we define a ``simplified'' or ``smaller'' type~${A(a)}$ such that we can quasi-embed the suborder of all terms~${b}$ in~${A}$ with~${a \nleq b}$ into~${A(a)}$. Informally, this corresponds to quasi-embeddings~${L_A(a) \to A(a)}$, which are used for the computation of maximal order types as we will see in the next section. It might not be immediately clear in which sense~${A(a)}$ is supposed to be a simplification of~$A$. This will be revealed at a later stage of the proof, where we assign elements of~${\varphi(2 + \alpha, 0)}$ to types in order to measure their complexity.
\begin{definition}
	For each type~${A}$ and term~${a}$ of~${A}$, we give a type~${A(a)}$:
	\begin{enumerate}[label=\roman*)]
		\item Let~${\beta, \gamma \in \alpha}$ with~${\gamma < \beta}$ be arbitrary. We set~${\labeltype_\beta(\gamma) := \labeltype_\gamma}$.
		\item Let~${\beta \in \alpha}$ be arbitrary. Given a term~${t}$ of type~${\treetype_{\beta, A}}$, we set
		\begin{equation*}
			\treetype_{\beta, A}(t) :=
			\begin{cases}
				\treetype_{\beta, A(a)} & \text{if~${t = a \star []}$,}\\
				\treetype_{\gamma, (A + \labeltype_\beta \times \treetype_{\beta, A}(t_l) + \labeltype_\beta \times \treetype_{\beta, A}(t_r))^*} & \text{if~${t = \gamma \star [t_l, t_r]}$.}
			\end{cases}
		\end{equation*}
		\item Let~${A}$ and~${B}$ be types. We define~${(A + B)(\iota_0(a)) := A(a) + B}$ if~${a}$ is a term of type~${A}$. If~${b}$ is a term of type~${B}$, we set~${(A + B)(\iota_1(b)) := A + B(b)}$.
		\item Let~${a}$ and~${b}$ be terms of types~${A}$ and~${B}$, respectively. Then, we define~${(A \times B)(\langle a, b \rangle) := A(a) \times B + A \times B(b)}$.
		\item Let~${A}$ be a type. We set
		\begin{equation*}
			A^*(s) :=
			\begin{cases}
				\emptytype & \text{if~${s = \langle \rangle}$,}\\
				A(a)^* + A(a)^* \times A \times A^*(s') & \text{if~${s = \langle a \rangle * s'}$.}
			\end{cases}
		\end{equation*}
	\end{enumerate}
\end{definition}

For convenience, we may use~${\iota_0}$,~${\iota_1}$, and~${\iota_2}$ in order to refer to the three different summands in the type~${A + \labeltype_\beta \times \treetype_{\beta, A}(t_l) + \labeltype_\beta \times \treetype_{\beta, A}(t_r)}$.
The equivalent of the usual quasi-embeddings~${L_A(a) \to A(a)}$ is then given by the following family of functions:
\begin{proposition}\label{prop:e_quasi-embedding}
	The system~$\rca_0$ proves that there is a family of (dependent) functions~${e_A}$ (for each type~${A}$) such that the following properties hold: An element is in the domain of~${e_A}$ if and only if it is a pair~${(a, b) \in A \times A}$ with~${a \nleq b}$. Moreover,~${e_A(a, b)}$ is an element of~${A(a)}$ for any such pair.
	Finally, consider elements~${a, b, c \in A}$ with~${a \nleq b, c}$. Then,~${e_A(a, b) \leq e_A(a, c)}$ entails~${b \leq c}$.
\end{proposition}

\begin{proof}
	The functions~${e_A}$ are defined simultaneously for all types~${A}$ by induction along~${h(a) + h(b)}$ for arguments~${(a, b) \in \dom(e_A)}$. We only do the constructions and proofs for types of the form~${\labeltype_\beta}$ and~${\treetype_{\beta, A}}$. The rest is covered by~\cite{FreundReification}.
	
	Given~${\beta \in \alpha}$, we set
	\begin{equation*}
		e_{\labeltype_\beta}(\gamma, \delta) := \delta
	\end{equation*}
	for~${\gamma, \delta \in \labeltype_\beta}$ with~${\gamma \nleq \delta}$. Clearly, for any further~${\delta' \in \labeltype_\beta}$ with~${\gamma \nleq \delta'}$ satisfying the inequality~${e_{\labeltype_\beta}(\gamma, \delta) \leq e_{\labeltype_\beta}(\gamma, \delta')}$, we have~${\delta \leq \delta'}$.
	Now, given~${\beta \in \alpha}$, some type~${A}$, and a term~${a \in A}$, we set
	\begin{equation*}
		e_{\treetype_{\beta, A}}(a \star [], t) :=
		\begin{cases}
			e_A(a, b) \star [] & \text{if~${t = b \star []}$,}\\
			\gamma \star [e_{\treetype_{\beta, A}}(a \star [], t_l), e_{\treetype_{\beta, A}}(a \star [], t_r)] & \text{if~${t = \gamma \star [t_l, t_r]}$\period}
		\end{cases}
	\end{equation*}
	By assumption on our domain, we have~${a \star [] \nleq t}$. Clearly, this implies~${a \nleq b}$ in the first case and~${a \star [] \nleq t_l, t_r}$ in the second one.
	Now, let~${s = \gamma \star [s_l, s_r]}$. In order to define~${e_{\treetype_{\beta, A}}}$ on pairs that start with~${s}$, we use side induction to define a function~${f}$ that maps any~${t \in \treetype_{\beta, A}}$ (with~${s \nleq t}$ such that all inner nodes of~${t}$ are greater than or equal to~${\gamma}$) to a term in~${(A + \labeltype_\beta \times \treetype_{\beta, A}(s_l) + \labeltype_\beta \times \treetype_{\beta, A}(s_r))^*}$:
	\begin{align*}
		f(t) :=
		\begin{cases}
			\langle \iota_0(a) \rangle & \text{if~${t = a \star []}$,}\\
			\langle \iota_1(\langle \delta, e_{\treetype_{\beta, A}}(s_l, t_l) \rangle) \rangle * f(t_r) &\text{if~${t = \delta \star [t_l, t_r]}$ and~${s_l \nleq t_l}$,}\\
			\langle \iota_2(\langle \delta, e_{\treetype_{\beta, A}}(s_r, t_r) \rangle) \rangle * f(t_l) &\text{if~${t = \delta \star [t_l, t_r]}$ and~${s_l \leq t_l}$.}\\
		\end{cases}
	\end{align*}
	Note that all recursive uses of~${f}$ are valid since the fact that all inner nodes of~${t}$ are greater than or equal to~${\gamma}$, clearly, is transferred to subtrees of~${t}$. Moreover, in the third case, we have~${s_r \nleq t_r}$ since, otherwise,~${\gamma \leq \delta}$ and~${s_l \leq t_l}$ lead to the contradiction~${s \leq t}$. Now, we are ready to define~${e_{\treetype_{\beta, A}}}$ for pairs that start with~${s}$:
	\begin{align*}
		&e_{\treetype_{\beta, A}}(s, t) :=\\
		&\begin{cases}
			\delta \star [e_{\treetype_{\beta, A}}(s, t_l), e_{\treetype_{\beta, A}}(s, t_r)] & \text{if~${t = \delta \star [t_l, t_r]}$ with~${\delta < \gamma}$,}\\
			f(t) \star [] & \text{otherwise.}
		\end{cases}
	\end{align*}
	Let us prove that~${s, t, u \in \treetype_{\beta, A}}$ with~${s \nleq t, u}$ and~${e_{\treetype_{\beta, A}}(s, t) \leq e_{\treetype_{\beta, A}}(s, u)}$ satisfy~${t \leq u}$: First, we assume that~${s = a \star []}$ holds for some term~${a}$ of type~${A}$. Now, if~${e_{\treetype_{\beta, A}}(s, t)}$ embeds into a subtree of~${e_{\treetype_{\beta, A}}(s, u)}$, we know that~${u}$ is of the form~${\delta \star [u_l, u_r]}$ and~${e_{\treetype_{\beta, A}}(s, t)}$ embeds into one of~${e_{\treetype_{\beta, A}}(s, u_l)}$ or~${e_{\treetype_{\beta, A}}(s, u_r)}$. In both cases, we have~${t \leq u_l}$ or~${t \leq u_r}$ and, thus, our claim~${t \leq u}$. Otherwise, assume that~${e_{\treetype_{\beta, A}}(s, t)}$ does not embed into a subtree of~${e_{\treetype_{\beta, A}}(s, u)}$. In the first case,~${t}$ and~${u}$  are of the form~${b \star []}$ and~${c \star []}$, respectively. Then,~${e_A(a, b) \leq e_A(a, c)}$ implies~${b \leq c}$ by induction hypothesis and we arrive at~${t \leq u}$. In the second case,~${t}$ and~${u}$ must be of the form~${\gamma \star [t_l, t_r]}$ and~${\delta \star [u_l, u_r]}$, respectively and we have~${\gamma \leq \delta}$ together with both~${e_{\treetype_{\beta, A}}(s, t_l) \leq e_{\treetype_{\beta, A}}(s, u_l)}$ and~${e_{\treetype_{\beta, A}}(s, t_r) \leq e_{\treetype_{\beta, A}}(s, u_r)}$. By induction hypothesis, this leads to~${t_l \leq u_l}$ and~${t_r \leq u_r}$. Finally, we combine everything and have~${t \leq u}$.
	
	Now, we assume that~${s}$ is of the form~${s = \gamma \star [s_l, s_r]}$. Before we prove our claim for~${e_{\treetype_{\beta, A}}}$, we show that~${f(t) \leq f(u)}$ implies~${t \leq u}$. For this, of course, we need to assume that all inner labels in~${t}$ and~${u}$ are greater than or equal to~${\gamma}$. Now, if~${f(t)}$ is less than or equal to the tail of~${f(u)}$, then~${u}$ must be of the form~${\delta \star [u_l, u_r]}$ and we have~${f(t) \leq f(u_l)}$ or~${f(t) \leq f(u_r)}$. By induction hypothesis, this leads to~${t \leq u_l}$ or~${t \leq u_r}$ and we, clearly, have~${t \leq u}$ in both cases. If~${f(t)}$ is not less than or equal to the tail of~${f(u)}$, then we know that both~${t}$ and~${u}$ must be of the same case with respect to the definition of~${f}$. If we are in the first case, then we have~${t = a \star []}$ and~${u = b \star []}$ for terms~${a}$ and~${b}$ of type~${A}$. The inequality~${f(t) \leq f(u)}$ immediately leads to~${a \leq b}$ and, thus,~${t \leq u}$. If we are in the second case, then~${t}$ and~${u}$ are of the form~${t = \delta \star [t_l, t_r]}$ and~${u = \delta' \star [u_l, u_r]}$, respectively. By definition and induction hypotheses, we arrive at~${\delta \leq \delta'}$,~${t_l \leq u_l}$, and~${t_r \leq u_r}$. Clearly, putting everything together yields~${t \leq u}$. The third case is similar.
	
	Let~${t, u \in \treetype_{\beta, A}}$ such that~${s \nleq t, u}$ holds with~${e_{\treetype_{\beta, A}}(s, t) \leq e_{\treetype_{\beta, A}}(s, u)}$. We want to derive~${t \leq u}$. If~${e_{\treetype_{\beta, A}}(s, t)}$ embeds into a subtree of~${e_{\treetype_{\beta, A}}(s, u)}$, then~${u}$ must be of the form~${u = \delta \star [u_l, u_r]}$ and we have one of the inequalities~${e_{\treetype_{\beta, A}}(s, t) \leq e_{\treetype_{\beta, A}}(s, u_l)}$ or~${e_{\treetype_{\beta, A}}(s, t) \leq e_{\treetype_{\beta, A}}(s, u_r)}$. By induction hypothesis, this yields~${t \leq u_l}$ or~${t \leq u_r}$ and, thus, our claim. Assume that~${e_{\treetype_{\beta, A}}(s, t)}$ does not embed into a subtree of~${e_{\treetype_{\beta, A}}(s, u)}$. Again,~${t}$ and~${u}$ must belong to the same case in the definition of~${e_{\treetype_{\beta, A}}}$. If both~${t}$ and~${u}$ are of the form~${t = \delta \star [t_l, t_r]}$ and~${u = \delta' \star [u_l, u_r]}$, respectively, such that~${\delta, \delta' < \gamma}$ holds, we have~${\delta \leq \delta'}$ and both~${e_{\treetype_{\beta, A}}(s, t_l) \leq e_{\treetype_{\beta, A}}(s, u_l)}$ and~${e_{\treetype_{\beta, A}}(s, t_r) \leq e_{\treetype_{\beta, A}}(s, u_r)}$. The latter inequalities lead to~${t_l \leq u_l}$ and~${t_r \leq u_r}$, by induction hypothesis. Our claim follows immediately. Otherwise, if~${t}$ and~${u}$ are not of this form, we have~${f(t) \leq f(u)}$ and, thus,~${t \leq u}$ from previous considerations.
\end{proof}

For the definition of the complexity of our types, we need natural (Hessenberg) sums as defined in \cite[Definition~4.4]{SimpsonHilbertBasis}:
\begin{definition}
	Consider any two elements of~${\varphi(2 + \alpha, 0)}$ written in form of sums~${x := x_0 + \dots + x_{n-1}}$ and~${y := y_0 + \dots + y_{m-1}}$ (where~$0$ and indecomposable elements are interpreted as empty and singleton sums, respectively). We define the Hessenberg sum
    \begin{equation*}
        x \hess y := z_0 + \dots + z_{n + m - 1}
    \end{equation*}
    where~${\langle z_0, \dots, z_{n+m-1} \rangle}$ is the (unique) weakly descending (re)ordering of the sequence~${\langle x_0, \dots, x_{n-1}, y_0, \dots, y_{m-1} \rangle}$.
\end{definition}
\begin{lemma}
	The system $\rca_0$ proves the following:
    \begin{enumerate}[label=\alph*)]
        \item The Hessenberg sum is both associative and symmetric.
        \item ${x \hess y < z \hess y}$ holds for any~${x, y, z \in \varphi(2 + \alpha, 0)}$ with~${x < z}$.
        \item ${x \hess y < z}$ holds for any~${x, y, z \in \varphi(2 + \alpha, 0)}$ with~${x, y < z}$ if~$z$ is indecomposable.
    \end{enumerate}
\end{lemma}
For proof of these facts, we refer to~\cite[Lemma~4.5]{SimpsonHilbertBasis}.
\begin{definition}
	We define a function~${o}$ that maps types to elements of~${\varphi(2 + \alpha, 0)}$: 
	\begin{equation*}
		o(A) :=
		\begin{cases}
			0 & \text{if~${A = \emptytype}$,}\\
			\varphi(2 + \beta, 0) & \text{if~${A = \labeltype_\beta}$,}\\
			\varphi(2 + \beta, o(B) + 1) & \text{if~${A = \treetype_{\beta, B}}$,}\\
			o(B) \hess o(C) & \text{if~${A = B + C}$,}\\
			\varphi(0, o(B + C)) & \text{if~${A = B \times C}$,}\\
			\varphi(1, o(B) + 1) & \text{if~${A = B^*}$.}
		\end{cases}
	\end{equation*}
\end{definition}
The addition ``$+1$'' in the third and in the last line circumvents fixed points and ensures that all our given terms of~${\varphi(2 + \alpha, 0)}$ are well-defined.

\begin{lemma}\label{lem:simplification_descends_in_order_type}
	The system~$\rca_0$ proves~${o(A(a)) < o(A)}$ for any term~${a}$ of a type~${A}$.
\end{lemma}
\begin{proof}
	We proceed by case distinction on all defined types. For~${A = \emptytype}$, the claim is immediate since this term does not have any terms. Consider~${\labeltype_\beta}$ for some~${\beta \in \alpha}$ and~${\gamma < \beta}$. Then, we have
	\begin{equation*}
		o(\labeltype_\beta(\gamma)) = o(\labeltype_\gamma) = \varphi(2 + \gamma, 0) < \varphi(2 + \beta, 0) = o(\labeltype_\beta)\period
	\end{equation*}
	Consider~${\treetype_{\beta, A}}$ for some~${\beta \in \alpha}$ and a type~${A}$. Let~${t}$ be a term in~${\treetype_{\beta, A}}$. We proceed by side induction on the amount of nodes in~${t}$. If~${t = a \star []}$ holds for some~${a}$ of type~${A}$, then we have
	\begin{align*}
		o(\treetype_{\beta, A}(t)) &= o(\treetype_{\beta, A(a)}) = \varphi(2 + \beta, o(A(a)) + 1)\\
		&< \varphi(2 + \beta, o(A) + 1) = o(\treetype_{\beta, A})
	\end{align*}
	by induction. Otherwise, if~${t = \gamma \star [t_l, t_r]}$ holds for~${\gamma < \beta}$ and trees~${t_l, t_r \in \treetype_{\beta, A}}$, then we derive
	\begin{align*}
		o(\treetype_{\beta, A}(t)) &{}= o(\treetype_{\gamma, (A + \labeltype_\beta \times \treetype_{\beta, A}(t_l) + \labeltype_\beta \times \treetype_{\beta, A}(t_r))^*})\\
		&{}= \varphi(2 + \gamma, \varphi(1, o(A) \hess \varphi(0, \varphi(2 + \beta, 0) \hess o(\treetype_{\beta, A}(t_l)))\\
		&\phantom{{}={} \varphi(2 + \gamma, \varphi(1, o(A)} \hess \varphi(0, \varphi(2 + \beta, 0)\hess o(\treetype_{\beta, A}(t_r))) + 1) + 1)\\
		&{}< \varphi(2 + \beta, o(A) + 1) = o(\treetype_{\beta, A})\period
	\end{align*}
	For the inequality, we are using the facts~${0, 1, 2 + \gamma < 2 + \beta}$ together with the inequality~${\varphi(2 + \beta, 0) < o(\treetype_{\beta, A})}$ and, with the help of our induction hypothesis,~${o(\treetype_{\beta, A}(t_l)), o(\treetype_{\beta, A}(t_r)) < o(\treetype_{\beta, A})}$.
	
	We conclude with the cases for sums, products, and sequences. Here, we cannot simply refer to \cite{FreundReification} since we use different values for~${o}$ in order to avoid the distinction between decomposable and indecomposable types.
	
	Let~${A = B + C}$ for types~${B}$ and~${C}$. W.l.o.g., consider some term~${\iota_0(b)}$ of~${B + C}$, where~${b}$ is a term of type~${B}$. We have
	\begin{equation*}
		o((B + C)(\iota_0(b)) = o(B(b) + C) = o(B(b)) \hess o(C) < o(B) \hess o(C) = o(B + C)\period
	\end{equation*}
    The case for~$\iota_1(c)$ with~$c$ in~$C$ works analogously.
	Let~${A = B \times C}$ for types~${B}$ and~${C}$. Consider some term~${\langle b, c \rangle}$ of~${B \times C}$, where~${b}$ and~${c}$ are terms of type~${B}$ and~${C}$, respectively. We have
	\begin{align*}
		o((B \times C)(\langle b, c \rangle)) &= o(B(b) \times C + B \times C(c)) = o(B(b) \times C) \hess o(B \times C(c))\\
		&= \varphi(0, o(B(b)) \hess o(C)) \hess \varphi(0, o(B) \hess o(C(c))\\
		&< \varphi(0, o(B) \hess o(C)) = o(B \times C)\period
	\end{align*}
	Finally, let~${A = B^*}$ for some type~${B}$. Consider a sequence~${s}$ of type~${A}$. We proceed by side induction on the length of~${s}$. If~${s}$ is empty, then we have the inequality~${o(A(s)) = o(\emptytype) = 0 < \varphi(1, o(B) + 1) = o(A)}$. Otherwise, if~${s}$ is of the form~${s = \langle b \rangle * s'}$ for some term of type~${B}$ and a sequence~${s'}$ in~${A}$, we have the following:
	\begin{align*}
		o(A(s)) &= o(B(b)^* + B(b)^* \times B \times A(s')) = o(B(b)^*) \hess o(B(b)^* \times B \times A(s'))\\
		&= o(B(b)^*) \hess \varphi(0, o(B(b)^*) \hess o(B \times A(s')))\\
		&= o(B(b)^*) \hess \varphi(0, o(B(b)^*) \hess \varphi(0, o(B) \hess o(A(s'))))\\
		&< \varphi(1, o(B) + 1) = o(A)\comma
	\end{align*}
	where we use~${o(B(b)^*) = \varphi(1, o(B(b)) + 1) < \varphi(1, o(B) + 1)}$ and~${o(A(s')) < o(A)}$, which hold by induction hypothesis, in the inequality.
\end{proof}

Next, we extend our definition of~${A(a)}$ from single terms~${a}$ of type~${A}$ to whole sequences:
\begin{definition}
	Given a type~${A}$ and a finite bad sequence~${s}$ consisting of terms in~${A}$, we define the type~${A[s]}$ as follows:
	\begin{equation*}
		A[s] :=
		\begin{cases}
			A & \text{if~${s = \langle \rangle}$,}\\
			A(s_0)[e_A(s_0, s_1), \dots, e_A(s_0, s_{|s|-1})] & \text{otherwise.}
		\end{cases}
	\end{equation*}
\end{definition}
We quickly argue that this definition is valid in~${\rca_0}$: Given a nonempty bad sequence~${s}$ of terms in~${A}$, we clearly have~${s_0 \nleq s_i}$ for any index~${i}$ with~${0 < i < |s|}$. Thus, we can consider~${e_A(s_0, s_i) \in A(s_0)}$. Moreover, from Proposition~\ref{prop:e_quasi-embedding}, it is immediate that~${e_A(s_0, s_1), \dots, e_A(s_0, s_{|s|-1})}$ is, again, a finite bad sequence.

\begin{lemma}\label{lem:simplification_descends_sequence}
    The system~$\rca_0$ proves~${o(A[s * \langle t \rangle]) < o(A[s])}$ for any finite bad sequence~${s* \langle t \rangle}$ of terms in~$A$.
\end{lemma}
\begin{proof}
    We prove our claim by induction along the length of~$s$ simultaneously for all types~$A$. First, if it is empty, then we have
    \begin{equation*}
        o(A[\langle t \rangle]) = o(A(t)[\langle \rangle]) = o(A(t)) < o(A) = o(A[\langle \rangle])
    \end{equation*}
    by Lemma~\ref{lem:simplification_descends_in_order_type}. Now, assume that~${n := |s|}$ is positive. Then, we derive
    \begin{align*}
        o(A[s * \langle t \rangle]) &= o(A(s_0)[\langle e_A(s_0, s_1), \dots, e_A(s_0, s_{n-1}), e_A(s_0, t) \rangle])\\
        &< o(A(s_0)[\langle e_A(s_0, s_1), \dots, e_A(s_0, s_{n-1}) \rangle])\\
        &= o(A[s])
    \end{align*}
    using our induction hypothesis applied to~${\langle e_A(s_0, s_1), \dots, e_A(s_0, s_{n-1}), e_A(s_0, t) \rangle}$ with respect to the term~$A(s_0)$.
\end{proof}

\begin{proof}[Proof of Proposition~\ref{prop:reification_b_alpha_1_into_varphi_2+alpha_1}]
	Let~${f}$ be the quasi-embedding from~${\bta_{\alpha, 1}}$ into~${\treetype_{\alpha, \emptytype^*}}$, which is canonically given by mapping the unique element in~$1$ to the empty sequence in~${\emptytype^*}$. We lift this function to sequences in~${(\bta_{\alpha, 1})^*}$ and call the resulting map~${g}$.
	
	Now, let our reification~${r}$ be the function that maps sequences~${s}$ from~${\badfin(\bta_{\alpha, 1})}$ to~${o(\treetype_{\alpha, \emptytype^*}[g(s)])}$ in~${\varphi(2 + \alpha, 0)}$. We prove that~${r}$ satisfies the condition for reifications: Consider a nonempty sequence~${s}$ and an element~${t}$ in~${\bta_{\alpha, 1}}$ such that~${s * \langle t \rangle}$ is a finite bad sequence. We need to verify that~${r(s * \langle t \rangle) < r(s)}$ holds.
    This can be seen using Lemma~\ref{lem:simplification_descends_sequence}:
    Let~${n := |s|}$. Then, we derive:
    \begin{align*}
        r(s * \langle t \rangle) &= o(\treetype_{\alpha, \emptytype^*}[g(s * \langle t \rangle)])\\
        &= o(\treetype_{\alpha, \emptytype^*}[\langle f(s_0), \dots, f(s_{n-1}), f(t) \rangle])\\
        &< o(\treetype_{\alpha, \emptytype^*}[\langle f(s_0), \dots, f(s_{n-1}) \rangle])\\
        &= o(\treetype_{\alpha, \emptytype^*}[g(s)])\\
        &= r(s)\text{.}
    \end{align*}
    This concludes our argument.
\end{proof}

\section{Maximal order types}\label{sec:order_types}

In this section, we compute the maximal order types claimed by Theorem~\ref{thm:maximal_order_types}. From now on, we are not working in the setting of second order arithmetic anymore. Instead, the following proofs assume the axioms of a sufficiently strong set theory such as~${\zfc}$. With enough effort, most of these results should be transferrable to weaker systems considered in reverse mathematics. However, already for the Cantor normal forms used by our theorem, the system of~${\atr_0}$ is required, in general (cf.~\cite{Hirst}).

\subsection{Upper bounds}

In this part, we calculate the upper bounds of our claims in Theorem~\ref{thm:maximal_order_types}. The general approach works as follows: Given a well partial order~${X}$ and an ordinal~${\alpha}$, we can prove~${o(X) \leq \alpha}$ by considering suborders
\begin{equation*}
	L_X(x) := \{y \in X \mid x \nleq y\} \subseteq X
\end{equation*}
for all~${x \in X}$.
We prove our claim simply by showing that~${l_X(x) := o(L_X(x)) < \alpha}$ holds for all~${x \in X}$. Note that the inequality is strict this time. At first glance, this might not seem that useful. However, it turns out that, in general, we can say a lot about the structure of~${L_X(x)}$ for each~${x \in X}$. Using quasi-embeddings from~${L_X(x)}$ into compositions of simpler well partial orders, for which we already know the respective order types, we arrive at~${l_X(x) < \alpha}$ for each~${x \in X}$. This approach is due to de Jongh and Parikh; more on this topic can be found in \cite{deJonghParikh}.

The ingredient that is vital for this approach is the knowledge of order types for simpler well partial orders and how order types behave with respect to different kinds of composition. E.g., given two well partial orders~${X}$ and~${Y}$, their disjoint union~${X \oplus Y}$ (this corresponds to the type~${A + B}$ from Section~\ref{sec:ATR}) has order type~${o(X \oplus Y) = X \hess Y}$, where~${\hess}$ denotes the natural (Hessenberg) sum. Also, their product~${X \otimes Y}$ (this corresponds to the type~${A \times B}$ from Section~\ref{sec:ATR}) has order type~${o(X \otimes Y) = X \hessMul Y}$, where~${\hessMul}$ denotes the natural product. Again, see \cite{deJonghParikh} for proofs.

Since we are working with sequences, it is not surprising that we will make use of the order types of sequences (ordered as in Higman's lemma). See \cite{deJonghParikh} and \cite{SchmidtHabilitation} for the following result:
\begin{theorem}\label{thm:Higman_order_types}
	Let~${X}$ be a well partial order, then~${X^*}$ is also a well partial order and we have the following maximal order types:
	\begin{equation*}
		o(X^*) :=
		\begin{cases}
			1 & \text{if~${X}$ is empty,}\\
			\omega^{\omega^{o(X) - 1}} & \text{if~${X}$ is nonempty but finite,}\\
			\omega^{\omega^{o(X) + 1}} & \text{if~${o(X) = \alpha + n}$ for an~${\varepsilon}$-number\footnotemark{}~${\alpha}$ and~${n \in \n}$,}\\
			\omega^{\omega^{o(X)}} & \text{otherwise.}
		\end{cases}
	\end{equation*}
    \footnotetext{{By this, we mean an ordinal~$\alpha$ satisfying~$\omega^\alpha = \alpha$.}}
\end{theorem}
In many of the following proofs, we need to subtract from ordinals in order to reach earlier induction hypotheses:
\begin{definition}[Left subtraction]
	Given ordinals~${\alpha \leq \beta}$, we define the \emph{left subtraction} of~${\alpha}$ from~${\beta}$, written as~${-\alpha + \beta}$, to be the unique ordinal~${\gamma}$ such that~${\alpha + \gamma = \beta}$ holds. For convenience, we set~${-\alpha + \beta := 0}$ for~${\alpha > \beta}$.
\end{definition}
Together with left addition, this can be extended to the members of sequences and inner nodes of trees:
\begin{definition}
	Let~${\alpha}$ be an ordinal and~${X}$ a well partial order.
	
	Given an ordinal~${\beta}$ and a sequence~${s \in \gapw_\alpha}$ (a tree~${t \in \bta_{\alpha, X}}$), we write~${\beta + s}$ ($\beta + t$) for the sequence in~${\gapw_{\beta + \alpha}}$ (tree in~${\bta_{\beta + \alpha, X}}$) that results if we replace any member~${\gamma}$ in~${s}$ (inner label~${\gamma}$ in~${t}$) with~${\beta + \gamma}$. Similarly, we write~${-\beta + s}$ ($-\beta + t$) for the sequence in~${\gapw_\alpha}$ (tree in~${\bta_{\alpha, X}}$) that results if we replace any member~${\gamma}$ is~${s}$ (inner label~${\gamma}$ in~${t}$) with~${-\beta + \gamma}$.
\end{definition}
Finally, we gather some basic properties about these new operations. This corresponds to Lemma~\ref{lem:left_addition_ATR} from the previous section.
\begin{lemma}\label{lem:left_subtraction}
	Let~${\alpha}$ be an ordinal and~${X}$ a well partial order. The following properties hold:
	\begin{enumerate}[label=\alph*)]
		\item ${-\beta + s_l \leq_\weak -\gamma + s_r}$ implies~${s_l \leq_\weak s_r}$ for sequences~${s_l, s_r \in \gapw_\alpha}$ and ordinals~${\beta \leq \gamma}$ such that~${\beta}$ is less than or equal to any member of~${s_l}$ and~${\gamma}$ is less than or equal to any member of~${s_r}$.
		\item ${\beta + s_l \leq_\weak \beta + s_r}$ implies~${s_l \leq_\weak s_r}$ for sequences~${s_l, s_r \in \gapw_\alpha}$ and ordinals~${\beta}$.
		\item ${\omega^\beta + s_l \leq_\weak \gamma + s_r}$ implies~${\omega^\beta + s_l \leq_\weak s_r}$ for sequences~${s_l, s_r \in \gapw_\alpha}$, and ordinals~${\beta}$ and~${\gamma}$ with~${\omega^\beta > \gamma}$.
	\end{enumerate}
	Properties a)-c) also hold for trees (if in~a) we replace ``member'' by ``inner label'') and sequences using the strong gap condition.
\end{lemma}
\begin{proof}
	Since our orders on trees and sequences only use \emph{positive} occurrences of the relation on ordinals, we simply have to show that these claims hold for ordinals.
	\begin{enumerate}[label=\alph*)]
		\item Assume that the inequality~${-\beta + \lambda \leq -\beta + \rho}$ holds for ordinals~${\beta \leq \lambda, \rho}$. Then, we have~${\lambda = \beta + (-\beta + \lambda) \leq \beta + (-\beta + \rho) = \rho}$.
		\item Assume that~${-\beta + \lambda \leq -\gamma + \rho}$ holds for~${\beta \leq \lambda}$ and~${\gamma \leq \rho}$ with~${\beta \leq \gamma}$. Then, we have~${\lambda = \beta + (-\beta + \lambda) \leq \gamma + (-\beta + \lambda) \leq \gamma + (-\gamma + \rho) = \rho}$.
		\item Assume~${\beta + \lambda \leq \beta + \rho}$. We have~${\lambda = -\beta + (\beta + \lambda) \leq -\beta + (\beta + \rho) = \rho}$ using b).
		\item Assume the inequality~${\omega^\beta + \lambda \leq \gamma + \rho}$ for ordinals with~${\omega^\beta > \gamma}$. First, notice that~${-\gamma + (\omega^\beta + \lambda) = \omega^\beta + \lambda}$ holds since~${\omega^\beta}$ is indecomposable and strictly greater than~${\gamma}$. Thus~${\omega^\beta + \lambda = -\gamma + (\omega^\beta + \lambda) \leq -\gamma + (\gamma + \rho) = \rho}$ holds by b).\qedhere
	\end{enumerate}
\end{proof}
We are ready to compute the first bounds for trees with finitely many inner labels:
\begin{lemma}\label{lem:small_upper_bounds_btal}
	For natural numbers~${n \in \n}$, the maximal order type of~${\btal_{n, 1}}$ is bounded as follows:
	\begin{enumerate}[label=\alph*)]
		\item ${o(\gapw_0) = o(\btal_{0, 1}) = 1}$,
		\item ${o(\gapw_1) = o(\btal_{1, 1}) = \omega}$,
		\item ${o(\btal_{n + 1, 1}) \leq \omega^{\omega^{o(\btal_{n,1})}}}$,
		\item ${o(\btal_{n + 1, X}) \leq \omega^{\omega^{o(\btal_{n, X}) + 1}}}$ for well partial orders~${X}$.
	\end{enumerate}
\end{lemma}

\begin{proof}
	The proofs for a) and b) are simple:~${\gapw_0}$ and~${\btal_{0, 1}}$ have exactly one element. Furthermore,~${\gapw_1}$ and~${\btal_{1, 1}}$ can be identified with the usual order on natural numbers. For c) and d), we construct a quasi-embedding~${f: \btal_{n+1, X} \to (\btal_{n, X})^*}$. In the case of~${X = 1}$, using a quick induction, we can see that the order type of~${\btal_{n, 1}}$ stays below~${\varepsilon_0}$. Thus, we have c) and d) by Theorem~\ref{thm:Higman_order_types} using the following map:
	\begin{equation*}
		f(t) :=
		\begin{cases}
			\langle -1 + t_l \rangle * f(t_r) & \text{if } t = 0 \star [t_l, t_r]\comma\\
			\langle -1 + t \rangle & \text{if each inner node in~${t}$ is positive.}
		\end{cases}
	\end{equation*}
	Note that in the first case, each inner node of~${t_l}$ must be positive since labels on left subtrees must strictly increase for~${t}$.
	
	We prove that~${f}$ is a quasi-embedding: Let~${s, t \in \btal_{n+1, 1}}$ with~${f(s) \leq f(t)}$. Using induction along the combined amount of nodes in~${s}$ and~${t}$, we show that this entails~${s \leq t}$.
	Assume that~$f(s)$ is less than or equal to the tail of~$f(t)$. Then, we know that~$t$ must be of the form~${t = 0 \star [t_l, t_r]}$ and we have~${f(s) \leq f(t_r)}$. By induction hypothesis, this yields~${s \leq t_r \leq t}$.

    Assume that~$f(s)$ is not less than or equal to the tail of~$f(t)$. There are three possible cases: In the first case,~$s$ is of the form~${s = x \star []}$ and~$t$ is of the form~${t = y \star []}$ for elements~${x, y \in X}$. Then,~${f(s) \leq f(t)}$ immediately leads to~${x \leq y}$ and, hence, to~${s \leq t}$. In the second case,~$s$ is still of this form but~${t = 0 \star [t_l, t_r]}$ holds for some trees~${t_l, t_r \in \btal_{n+1, X}}$. By assumption, we have~${-1 + s \leq -1 + t_l}$. Using Lemma~\ref{lem:left_subtraction}~a), this leads to~${s \leq t_l \leq t}$. Finally, assume that~$t$ is still of this form but we also have~${s = 0 \star [s_l, s_r]}$ for trees~${s_l, s_r \in \btal_{n+1, X}}$. By assumption, we have both~${-1 + s_l \leq -1 + t_l}$ and~${f(s_r) \leq f(t_r)}$. By Lemma~\ref{lem:left_subtraction}~a), the former entails~${s_l \leq t_l}$, while the latter leads to~${s_r \leq t_r}$ by our induction hypothesis. We conclude~${s \leq t}$.
\end{proof}
In order to distinguish regular sums from Cantor normal forms, we introduce the following notation:
\begin{definition}[Cantor normal form]
	Given ordinals~${\alpha}$ and~${\gamma_0, \dots, \gamma_{n-1}}$, we write
	\begin{equation*}
		\alpha \eqnf \omega^{\gamma_0} + \dots + \omega^{\gamma_{n-1}}
	\end{equation*}
	if~${\alpha}$ is equal to the right hand side and~${\gamma_i \geq \gamma_j}$ holds for any indices~${i < j < n}$. We may also abbreviate~${\alpha \eqnf \omega^{\gamma_0} + \delta}$, where~${\delta}$ stands for the term~${\omega^{\gamma_1} + \dots + \omega^{\gamma_{n-1}}}$.
\end{definition}
At the very end of this article, we will also use the more general version of this normal form, where base~${\omega}$ is replaced by some other ordinal~${\alpha \geq 2}$. In this case, we denote the ensuing normal form using~${\alphanf}$ instead of~${\eqnf}$.
We continue with the first quasi-embeddings from partial orders~${L_X(x)}$ as discussed in the introduction of this part:
\begin{lemma}\label{lem:left_set_quasi_embeddings}
	Given an ordinal~${\alpha}$ and a well partial order~${X}$, we can construct the following quasi-embeddings:
	\begin{enumerate}[label=\alph*)]
		\item ${L_{\bta_{\alpha, X}}(x \star []) \to \bta_{\alpha, L_X(x)}}$ for any~${x \in X}$
		\item ${L_{\bta_{\alpha, X}}(0 \star [t_l, t_r]) \to (X \oplus \alpha \otimes L_{\bta_{\alpha, X}}(t_l) \oplus \alpha \otimes L_{\bta_{\alpha, X}}(t_r))^*}$ for~${t_l, t_r \in \bta_{\alpha, X}}$
		\item ${L_{\bta_{\alpha, X}}(t) \to \bta_{\omega^{\gamma_0}, Y}}$ with~${Y := L_{\bta_{\alpha, X}}(-\omega^{\gamma_0} + t)}$ for trees~${t \in \bta_{\alpha, X} \setminus \bta_{1, X}}$, where we have~${\gamma \eqnf \omega^{\gamma_0} + \dots + \omega^{\gamma_{n-1}}}$ for the largest label~${\gamma}$ of an inner node occurring in~${t}$.
	\end{enumerate}	
	The quasi-embeddings still hold, if we replace~${\bta}$ with~${\btal}$.
\end{lemma}

\begin{remark}
	Equivalents of the quasi-embeddings from a) and b) can also be proved for trees with non-ascending labels. However, our proof for c) will not work for this more general class of trees: Such a result could be used to prove that binary trees with (not necessarily ascending) labels in~${2}$ have a maximal order type below~${\Gamma_0}$. This, however, does not hold (cf.~\cite[Theorem~14.11]{SchmidtvdMeerenWeiermann}).
\end{remark}

\begin{proof}
	In the following, we make implicit use of the fact that both~${\bta_{\alpha, X}}$ and~${\btal_{\alpha, X}}$ are closed under sub-trees.
	\begin{enumerate}[leftmargin=*, label=\alph*)]
		\item This (quasi-)embedding is simply given by the identity function: Consider some tree~${t \in L_{\bta_{\alpha, X}}(x \star [])}$ and assume that~${t}$ has a leaf label~${y \in X}$ with~${x \leq_X y}$. Then, we also have~${x \star [] \leq_{\bta_{\alpha, X}} t}$, which is a contradiction. Thus,~${x \nleq_X y}$, i.e.~$y \in L_X(x)$, holds for any leaf label~${y}$ in~${t}$.
		\item We construct~${f: L_{\bta_{\alpha, X}}(0 \star [t_l, t_r]) \to (X \oplus \alpha \otimes L_{\bta_{\alpha, X}}(t_l) \oplus \alpha \otimes L_{\bta_{\alpha, X}}(t_r))^*}$ as follows: Let~${s \in L_{\bta_{\alpha, X}}(0 \star [t_l, t_r])}$ be arbitrary, then we define
		\begin{equation*}
			f(s) :=
			\begin{cases}
				\langle \iota_0(x) \rangle & \text{ if } s = x \star [] \text{ for some } x \in X\comma\\
				\langle \iota_1((\beta, s_l)) \rangle * f(s_r) & \text{ if } s = \beta \star [s_l, s_r] \text{ with } t_l \nleq_{\bta_{\alpha, X}} s_l\comma\\
				\langle \iota_2((\beta, s_r)) \rangle * f(s_l) & \text{ if } s = \beta \star [s_l, s_r] \text{ with } t_l \leq_{\bta_{\alpha, X}} s_l\period
			\end{cases}
		\end{equation*}
		Note that in the third case,~${t_l \leq_{\bta_{\alpha, X}} s_l}$ implies~${t_r \nleq_{\bta_{\alpha, X}} s_r}$ since, otherwise, we arrive at the contradiction~${0 \star [t_l, t_r] \leq_{\bta_{\alpha, X}} s}$.
		By a simple induction on the height of trees, we can see that~${f}$ is well-defined.
		
		We prove that~${f}$ reflects orders: Consider trees~${s, s' \in L_{\bta_{\alpha, X}}(0 \star [t_l, t_r])}$ satisfying~${f(s) \leq f(s')}$. We want to derive~${s \leq s'}$ by induction on the sum of the heights of~${s}$ and~${s'}$. First, assume that~${f(s) \leq f(s')}$ holds since~${f(s)}$ is already less than or equal to the tail of~${f(s')}$. In this case,~${s'}$ must be of the form~${s' = \beta \star [s'_l, s'_r]}$ and we have one of~${f(s) \leq f(s'_l)}$ or~${f(s) \leq f(s'_r)}$. In both cases, the induction hypothesis yields~${s \leq s'_l}$ or~${s \leq s'_r}$, which results in~${s \leq s'}$.
		
		Now, assume that~${f(s) \leq f(s')}$ holds since the head and tail of~${f(s)}$ are each less than or equal to the head and tail of~${f(s')}$, respectively. Since both heads have to be comparable, we know that we are in the same of the three cases of the definition of~${f}$ for both~${f(s)}$ and~${f(s')}$. Assume that we are in the first case. Then,~${s = x \star []}$ and~${s' = y \star []}$ hold for elements~${x, y \in X}$. From~${\iota_0(x) \leq \iota_0(y)}$, we derive~${x \leq y}$ and, thus,~${s \leq s'}$.
		
		In the remaining cases, we have~${s = \beta \star [s_l, s_r]}$ and~${s' = \gamma \star [s'_l, s'_r]}$ for ordinals~${\beta, \gamma < \alpha}$ and sequences~${s_l, s_r, s'_l, s'_r \in \bta_{\alpha, X}}$. Assume that we are in the second case with~${t_l \nleq s_l, s'_l}$. First, since the head of~${f(s)}$ is less than or equal to the head of~${f(s')}$, we derive~${\iota_1((\beta, s_l)) \leq \iota_1((\gamma, s'_l))}$ and, thus, both~${\beta \leq \gamma}$ and~${s_l \leq s'_l}$ hold. Secondly, since the tail of~${f(s)}$ is less than or equal to the tail of~${f(s')}$, we get~${f(s_r) \leq f(s'_r)}$ and, hence, by induction hypothesis~${s_r \leq s'_r}$. Combining everything yields~${s \leq s'}$.
		The third case with~${t_l \leq s_l, s'_l}$ works analogously.
		\item
		Notice that~${n \geq 1}$ holds and, therefore,~${\gamma_0}$ is well-defined since we assumed that~${t}$ is not an element of~${\bta_{1, X}}$.
		We construct~${g: L_{\bta_{\alpha, X}}(t) \to \bta_{\omega^{\gamma_0}, Y}}$ as follows. Let~${s \in L_{\bta_{\alpha, X}}(t)}$ be arbitrary, then we define
		\begin{equation*}
			g(s) :=
			\begin{cases}
				(-\omega^{\gamma_0} + s) \star [] & \text{ if~${\beta \geq \omega^{\gamma_0}}$ holds for all inner labels~${\beta}$ in~${s}$,}\\
				\beta \star [g(s_l), g(s_r)] & \text{ if } s = \beta \star [s_l, s_r] \text{ and } \beta < \omega^{\gamma_0}\period
			\end{cases}
		\end{equation*}
		In the context of~${\bta_{\alpha, X}}$, this case distinction reaches all possible cases: If~${s}$ has an inner label below~${\omega^{\gamma_0}}$, then so does the root label since~${s}$ has weakly ascending labels.
		
		We check that~${g}$ is well-defined by induction on the height of the argument:
		In the first case, where all inner labels of~${s}$ are greater or equal to~${\omega^{\gamma_0}}$, we consider the tree~${t'}$ that is defined like~${t}$ but where each inner node that lies strictly below~${\omega^{\gamma_0}}$ is replaced by~${\omega^{\gamma_0}}$. From the assumption~${t \nleq s}$, we easily derive~${t' \nleq s}$. Thus, by Lemma~\ref{lem:left_subtraction}~a), we have~${-\omega^{\gamma_0} + t = -\omega^{\gamma_0} + t' \nleq -\omega^{\gamma_0} + s}$. We conclude that~${-\omega^{\gamma_0} + s}$ is an element of~${Y = L_{\bta_{\alpha, X}}(-\omega^{\gamma_0} + t)}$ and~${g(s)}$ is an element of~${\bta_{\omega^{\gamma_0}, Y}}$. In the second case, where the root label~${\beta}$ of~${s = \beta \star [s_l, s_r]}$ is strictly below~${\omega^{\gamma_0}}$, we know by induction hypothesis that both~${g(s_l)}$ and~${g(s_r)}$ are elements of~${\bta_{\omega^{\gamma_0}, Y}}$.
		
		Let us also note at this point that all trees produced by~${g}$ have weakly ascending inner labels. Likewise, in the case where we consider~${\btal}$ instead of~${\bta}$, inner nodes strictly increase when we consider the subtree to the left. Both can be verified by a short induction.
		
		Now, we verify that~${g}$ reflects orders. Let~${s, s' \in L_{\bta_{\alpha, X}}(t)}$ with~${g(s) \leq g(s')}$. By induction along the sum of the heights of~${s}$ and~${s'}$, we prove~${s \leq s'}$. Assume that~${g(s) \leq g(s')}$ holds since~${g(s)}$ embeds into a subtree of~${g(s')}$. Then,~${s'}$ must be of the form~${\beta' \star [s'_l, s'_r]}$ such that~${g(s)}$ embeds into~${g(s'_l)}$ or~${g(s'_r)}$. By induction hypothesis, we have~${s \leq s'_l}$ or~${s \leq s'_r}$. In both cases, we conclude~${s \leq s'}$. From now on, we assume that~${g(s)}$ does not embed into a subtree of~${g(s')}$.
		
		Assume that all inner labels in~${s}$ are greater than or equal to~${\omega^{\gamma_0}}$. Since~${g(s)}$ does not embed into a subtree of~${g(s')}$, we know that~${g(s')}$ cannot have any inner nodes. Thus, all inner labels in~${s'}$ must be greater than or equal to~${\omega^{\gamma_0}}$ and we have~${-\omega^{\gamma_0} + s \leq -\omega^{\gamma_0} + s'}$. By Lemma~\ref{lem:left_subtraction}~a), this entails~${s \leq s'}$.
		
		Assume that~${s}$ is of the form~${\beta \star [s_l, s_r]}$ with~${\beta < \omega^{\gamma_0}}$. Since~${g(s)}$ can only embed into a tree with inner nodes, we derive that~${s'}$ is of the form~${\beta' \star [s'_l, s'_r]}$. We assumed that~${g(s)}$ does not embed into a subtree of~${g(s')}$. Thus, we have~${\beta \leq \beta'}$,~${g(s_l) \leq g(s'_l)}$, and~${g(s_r) \leq g(s'_r)}$. By induction hypothesis, we derive~${s_l \leq s'_l}$ and~${s_r \leq s'_r}$ from the latter two. Finally, we conclude~${s \leq s'}$.\qedhere
	\end{enumerate}
\end{proof}
Already the nesting behavior in Lemma~\ref{lem:left_set_quasi_embeddings} strongly indicates that the maximal order types of our trees can be expressed using the Veblen hierarchy. However, because of their fixed points, we will first continue with the following function~${\psi}$.
\begin{definition}
	We define a function~${\psi: (\on \setminus \{0\}) \times (\on \setminus \{0\}) \to \on}$. For ordinals~${\alpha, \beta > 0}$, we have
	\begin{equation*}
		\psi(\alpha, \beta) :=
		\begin{cases}
			\varphi_{\alpha}(\beta + 1) & \text{ if } \beta \geq \varphi_{\alpha + 1}(0)\comma\\
			\varphi_{\alpha}(0) & \text{ if } \beta = 1 \text{ and } \alpha < \varphi_{\alpha}(0)\comma\\
			\varphi_{\alpha}(\beta) & \text{ otherwise.}
		\end{cases}
	\end{equation*}
\end{definition}

\begin{lemma}\label{lem:psi_properties}
	The function~${\psi}$ satisfies the following for ordinals~${\alpha, \beta, \gamma, \delta > 0}$:
	\begin{enumerate}[label=\alph*)]
		\item ${\psi(\alpha, \beta) < \psi(\gamma, \delta)}$ if
		\begin{itemize}
			\item ${\alpha = \gamma}$ and~${\beta < \delta}$, or
			\item ${\alpha < \gamma}$ and~${\psi(\alpha, \beta) < \delta}$, or
			\item ${\alpha > \gamma}$ and~${\beta < \psi(\gamma, \delta)}$.
		\end{itemize}
		\item ${\psi(\alpha, \beta) > 0}$.
		\item ${\psi(\alpha, \beta)}$ is closed under sums, products, and exponentials.
		\item ${\alpha < \psi(\alpha, \beta)}$.
		\item ${\beta < \psi(\alpha, \beta)}$.
	\end{enumerate}
\end{lemma}

\begin{proof}\mbox{}
	\begin{enumerate}[leftmargin=*, label=\alph*)]
		\item Assume that we have~${\alpha = \gamma}$ and~${\beta < \delta}$. If~${\beta < \varphi_{\alpha + 1}(0)}$ holds, then this implies~${\psi(\alpha, \beta) \leq \varphi_{\alpha}(\beta) < \varphi_{\gamma}(\delta) \leq \psi(\gamma, \delta)}$. For the last inequality, we used~${\delta > 1}$, which follows from~${\beta > 0}$.
		
		If~${\beta \geq \varphi_{\alpha + 1}(0)}$ holds, then the same does hold for~${\delta}$. Thus, we have the inequality~${\psi(\alpha, \beta) = \varphi_{\alpha}(\beta + 1) < \varphi_{\gamma}(\delta + 1) = \psi(\gamma, \delta)}$.
		
		Assume that we have~${\alpha < \gamma}$ and~${\beta < \psi(\gamma, \delta)}$. The value of~${\psi(\gamma, \delta)}$ is an~${\varepsilon}$-ordinal. This implies~${\beta + 1 < \psi(\gamma, \delta)}$. Since~${\psi(\gamma, \delta)}$ is the value of~${\varphi_{\gamma}(\zeta)}$ for some~${\zeta}$, this yields~${\psi(\alpha, \beta) \leq \varphi_{\alpha}(\beta + 1) < \varphi_{\gamma}(\zeta) = \psi(\gamma, \delta)}$.
		
		Assume that we have~${\alpha > \gamma}$ and~${\psi(\alpha, \beta) < \delta}$. Similar to before,~${\psi(\alpha, \beta)}$ is equal to~${\varphi_{\alpha}(\zeta)}$ for some~${\zeta}$. Moreover, we know that~${\delta}$ must be greater than~${1}$. This implies~${\varphi_{\gamma}(\delta) \leq \psi(\gamma, \delta)}$. In conclusion, we arrive at this inequality:~${\psi(\alpha, \beta) = \varphi_{\alpha}(\zeta) < \varphi_{\gamma}(\delta) \leq \psi(\gamma, \delta)}$.
		\item All values in the Veblen-hierarchy are positive.
		\item Since we only consider~${\varphi_{\alpha}(\beta)}$ for positive~${\alpha}$, any value of~${\psi}$ is an~${\varepsilon}$-number and, therefore, closed under sums, products, and exponentials.
		\item If~${\alpha < \varphi_{\alpha}(0)}$ holds, then we have~${\alpha < \varphi_{\alpha}(0) = \psi(\alpha, 1) \leq \psi(\alpha, \beta)}$ by a). Otherwise, if~${\alpha = \varphi_{\alpha}(0)}$ holds, this implies~${\alpha = \varphi_{\alpha}(0) < \varphi_{\alpha}(\beta) \leq \psi(\alpha, \beta)}$.
		\item If~${\beta \geq \varphi_{\alpha + 1}(0)}$ holds, then we have~${\beta < \beta + 1 \leq \varphi_{\alpha}(\beta + 1) = \psi(\alpha, \beta)}$. Otherwise, if~${\beta < \varphi_{\alpha + 1}(0)}$ holds, this implies that~${\beta}$ is not a fixed point of~${\gamma \mapsto \varphi_{\alpha}(\gamma)}$. Thus, if~${\beta}$ is greater than~${1}$, we conclude~${\beta < \varphi_{\alpha}(\beta) = \psi(\alpha, \beta)}$. Finally, if~${\beta}$ is equal to~${1}$, then~${\beta < \psi(\alpha, \beta)}$ trivially follows from c), i.e., the fact that~${\psi(\alpha, \beta)}$ is an~${\varepsilon}$-number.\qedhere
	\end{enumerate}
\end{proof}
Now we are ready to combine several of our previous results and compute upper bounds for trees with labels from an infinite ordinal:
\begin{lemma}\label{lem:tree_bound_with_psi}
	For any ordinal~${\alpha}$ and nonempty well partial order~${X}$, the following bounds hold:
	\begin{enumerate}[label=\alph*)]
		\item ${o(\bta_{\omega^\alpha, X}) \leq \psi(1 + \alpha, o(X))}$.
		\item ${o(\btal_{\omega^\alpha, X}) \leq \psi(\alpha, o(X))}$ if~${\alpha > 0}$.
	\end{enumerate}
\end{lemma}
\begin{proof}
	By induction along~${\alpha}$, we claim both by (inner) induction along~$o(X)$. We do so by showing that the set~${L_{\bta_{\omega^\alpha, X}}(t)}$ (resp.~$L_{\btal_{\omega^\alpha, X}}(t)$) lies strictly below the ordinal on the right hand side for any~${X}$ and any considered tree~${t}$. This is done by side induction along the natural numbers using the following assignment~${i: \bta_{\omega^\alpha, X} \to \n}$:
	\begin{equation*}
		i(t) :=
		\begin{cases}
			0 & \text{ if } t = x \star [] \text{ for } x \in X\comma\\
			i(t_l) + i(t_r) + n + 1 & \text{ if } t = \beta \star [t_l, t_r] \text{ for } \beta \eqnf \omega^{\beta_0} + \dots + \omega^{\beta_{n-1}}\period
		\end{cases}
	\end{equation*}
	Given~${n \in \n}$, we assume that the claim of our side induction has been proved for all trees~${t}$ with~${i(t) < n}$ and show that it holds for all trees~${t}$ with~${i(t) = n}$. Thus, we have induction in three layers: The first layer proceeds along~${\alpha}$, the second along~${o(X)}$, and the third along~${i(t)}$.
	
	If~${n = 0}$, then any tree~${t}$ with~${i(t) = n}$ must be a leaf~${t = x \star []}$ for some~${x \in X}$. We use the quasi-embedding~${L_{\bta_{\omega^\alpha, X}}(x \star []) \to \bta_{\omega^\alpha, L_X(x)}}$ from Lemma~\ref{lem:left_set_quasi_embeddings}~a) (which also holds for~${\btal_{\alpha, X}}$) and prove:
	\begin{enumerate}[label=\alph*)]
		\item ${l_{\bta_{\omega^\alpha, X}}(x \star []) \leq o(\bta_{\omega^\alpha, L_X(x)}) \leq \psi(1 + \alpha, o(L_X(x)))< \psi(1 + \alpha, o(X))}$.
		\item ${l_{\btal_{\omega^\alpha, X}}(x \star []) \leq o(\btal_{\omega^\alpha, L_X(x)}) \leq \psi(\alpha, o(L_X(x))) < \psi(\alpha, o(X))}$.
	\end{enumerate}
	Here, we used Lemma~\ref{lem:psi_properties}~a) for the last inequalities. Moreover, we applied the main (inner) induction since~${\alpha}$ stays fixed but we consider~${L_X(x)}$ satisfying the inequality~${l_X(x) < o(X)}$. Of course, the whole argument is only valid as long as~$L_X(x)$ is nonempty. Should it be empty, then the same holds for~$\bta_{\omega^\alpha, L_X(x)}$ and~$\btal_{\omega^\alpha, L_X(x)}$, which renders our claim trivial as both~$\psi(1 + \alpha, o(X))$ and~$\psi(\alpha, o(X))$ must always be positive by definition.
	
	If~${n > 0}$, then any tree with~${i(t) = n}$ must be of the form~${t = \beta \star [t_l, t_r]}$ for~${\beta < \alpha}$. Let us first consider the case where~${\beta = 0}$ holds. Here, we make use of the quasi-embedding ~${L_{\bta_{\omega^\alpha, X}}(0 \star [t_l, t_r]) \to (X \oplus \omega^\alpha \otimes L_{\bta_{\omega^\alpha, X}}(t_l) \oplus \omega^\alpha \otimes L_{\bta_{\omega^\alpha, X}}(t_r))^*}$, which we constructed in Lemma~\ref{lem:left_set_quasi_embeddings}~b). Using the side induction for the smaller trees~${t_l}$ and~${t_r}$, which satisfy~${i(t_l), i(t_r) < i(t)}$, we prove our claim. During the following inequalites, we make use of Lemma~\ref{lem:psi_properties}~c), d), and e) for~${\omega^\alpha < \psi(\alpha, o(X))}$ and~${o(X) < \psi(\alpha, o(X))}$.
	\begin{enumerate}[label=\alph*)]
		\item
        $\begin{aligned}[t]
            l_{\bta_{\omega^\alpha, X}}(0 \star [t_l, t_r]) &\leq \varphi_0(\varphi_0(o(X) \hess \omega^\alpha \hessMul l_{\bta_{\omega^\alpha, X}}(t_l) \hess \omega^\alpha \hessMul l_{\bta_{\omega^\alpha, X}}(t_r) + 1))\\&< \psi(1 + \alpha, o(X))
        \end{aligned}$\vspace{0.5em}
		\item
        $\begin{aligned}[t]
            l_{\btal_{\omega^\alpha, X}}(0 \star [t_l, t_r]) &\leq \varphi_0(\varphi_0(o(X) \hess \omega^\alpha \hessMul l_{\btal_{\omega^\alpha, X}}(t_l) \hess \omega^\alpha \hessMul l_{\btal_{\omega^\alpha, X}}(t_r) + 1))\\&< \psi(\alpha, o(X))
        \end{aligned}$
	\end{enumerate}
	(We write~${\varphi_0(\dots)}$ instead of~${\omega^{\dots}}$ in order to avoid too many superscripts.)
	The last inequalites hold since both~${\psi(1 + \alpha, o(X))}$ and~${\psi(a, o(X))}$ are positive and closed under sums, products, and exponentials by Lemma~\ref{lem:psi_properties}~b) and~c).
	
	Let us continue with the case where~${\beta}$ is positive. This time, we use the quasi-embedding~${L_{\bta_{\omega^\alpha, X}}(t) \to \bta_{\omega^{\gamma_0}, Y}}$ with~${Y := L_{\bta_{\omega^\alpha, X}}(-\omega^{\gamma_0} + t)}$ where~${\gamma_0}$ is the highest power occurring in the Cantor normal form of an inner node in~${t}$. We use the main (outer) induction, which yields~${o(\bta_{\omega^{\gamma_0}, Y}) \leq \psi(1 + \gamma_0, o(Y))}$.	
	Moreover, we have~${i(-\omega^{\gamma_0} + t) < i(t)}$. Thus, we can apply our side induction to~${-\omega^{\gamma_0} + t}$.
	\begin{enumerate}[label=\alph*)]
		\item ${l_{\bta_{\omega^\alpha, X}}(t) \leq \psi(1 + \gamma_0, l_{\bta_{\omega^\alpha, X}}(-\omega^{\gamma_0} + t)) < \psi(1 + \alpha, o(X))}$.
		Here, we apply Lemma~\ref{lem:psi_properties}~a) to~${\gamma_0 < \alpha}$ and~${l_{\bta_{\omega^\alpha, X}}(-\omega^{\gamma_0} + t) < \psi(1 + \alpha, o(X))}$.
		\item For~${\gamma_0 > 0}$, the argument is analogous to a). However, for~${\gamma_0 = 0}$, we cannot apply our main induction hypothesis as this case is explicitly excluded. Using Lemma~\ref{lem:left_set_quasi_embeddings}~c), we have a quasi-embedding~${L_{\btal_{\omega^\alpha, X}}(t) \to \btal_{1, Y}}$ with~${Y := L_{\btal_{\omega^\alpha, X}}(-1 + t)}$. With Lemma~\ref{lem:small_upper_bounds_btal}~d), we get
		\begin{equation*}
			l_{\btal_{\omega^\alpha, X}}(t) \leq \varphi_0(\varphi_0(l_{\btal_{\omega^\alpha, X}}(-1 + t) + 1)) < \psi(\alpha, o(X))\period
		\end{equation*}
		Again, we used that~${\psi(\alpha, o(X))}$ is closed under sums, products, and exponentials.
	\end{enumerate}
    Of course, similar to before, it is possible that~$Y$ is empty. In this case, we are not allowed to apply the induction hypothesis. Still, our claims trivially follow as~${\psi(1 + \alpha, o(X))}$ and~${\psi(\alpha, o(X))}$ must always be positive.
\end{proof}
Using the definition of our function~${\psi}$ now yields results in terms of the Veblen hierarchy:
\begin{corollary}\label{cor:upper_bound_indecomposable_varphi}
	For ordinals~${\alpha}$ and well partial orders~${X}$, the following bounds hold:
	\begin{enumerate}[label=\alph*)]
		\item ${o(\bta_{\omega^\alpha, X}) \leq \varphi_{1 + \alpha}(o(X))}$ if~${o(X) < \varphi_{1 + \alpha + 1}(0)}$.
		\item ${o(\btal_{\omega^\alpha, X}) \leq \varphi_{\alpha}(o(X))}$ if~${o(X) < \varphi_{\alpha + 1}(0)}$ and~${\alpha > 0}$.
	\end{enumerate}
	Moreover, for any ordinal~${\alpha}$ with~${\alpha < \varphi_{\alpha}(0)}$, the following bounds hold:
	\begin{enumerate}[label=\alph*), start=3]
		\item ${o(\bta_{\omega^\alpha, 1}) \leq \varphi_{1 + \alpha}(0)}$.
		\item ${o(\btal_{\omega^\alpha, 1}) \leq \varphi_{\alpha}(0)}$ if~${\alpha > 0}$.
	\end{enumerate}
\end{corollary}

\begin{proof}
	For the empty well partial order~${X = \emptyset}$, both claims a) and b) are clear since then, both~${\bta_{\omega^\alpha, X}}$ and~${\btal_{\omega^\alpha, X}}$ are empty. For~${X \neq \emptyset}$, we assume~${o(X) < \varphi_{1 + \alpha + 1}(0)}$, resp.~$o(X) < \varphi_{\alpha + 1}(0)$, and instantiate Lemma~\ref{lem:tree_bound_with_psi}.
	\begin{enumerate}[label=\alph*)]
		\item ${o(\bta_{\omega^\alpha, X}) \leq \psi(1 + \alpha, o(X)) \leq \varphi_{1 + \alpha}(o(X))}$.
		\item ${o(\btal_{\omega^\alpha, X}) \leq \psi(\alpha, o(X)) \leq \varphi_{\alpha}(o(X))}$.
	\end{enumerate}
	Now, we assume that~${\alpha < \varphi_\alpha(0)}$ holds. Clearly, this entails~${1 + \alpha < \varphi_{1 + \alpha}(0)}$. Together with~${1 < \varphi_{\alpha + 1}(0)}$, this yields both equalities~${\psi(1 + \alpha, 1) = \varphi_{1 + \alpha}(0)}$ as well as~${\psi(\alpha, 1) = \varphi_{\alpha}(0)}$:
	\begin{enumerate}[label=\alph*), start=3]
		\item ${o(\bta_{\omega^\alpha, 1}) \leq \psi(1 + \alpha, 1) = \varphi_{1 + \alpha}(0)}$.
		\item ${o(\btal_{\omega^\alpha, 1}) \leq \psi(\alpha, 1) = \varphi_{\alpha}(0)}$.\qedhere
	\end{enumerate}
\end{proof}
Until now, our upper bounds for labels from an infinite ordinal required this ordinal to be indecomposable. In order to extend this to arbitrary infinite ordinals, we use the next lemma:
\begin{lemma}\label{lem:decomposable_inner_label}
	Given an ordinal~${\alpha \eqnf \omega^{\gamma} + \delta}$, we can prove the following quasi-embeddings:
	\begin{enumerate}[label=\alph*)]
		\item ${\bta_{\alpha, 1} \to \bta_{\omega^{\gamma}, X}}$ for~${X := \bta_{\delta, 1}}$.
		\item ${\btal_{\alpha, 1} \to \btal_{\omega^{\gamma}, X}}$ for~${X := \btal_{\delta, 1}}$.
	\end{enumerate}
\end{lemma}

\begin{proof}
	We prove a): Consider the map~${f: \bta_{\alpha, 1} \to \bta_{\omega^\gamma, X}}$ with
	\begin{equation*}
		f(t) :=
		\begin{cases}
			(-\omega^\gamma + t) \star [] & \text{if all inner labels in~${t}$ are greater than or equal to~${\omega^\gamma}$,}\\
			\beta \star [f(t_l), f(t_r)] & \text{if~${t = \beta \star [t_l, t_r]}$ with~${\beta < \omega^\gamma}$.}
		\end{cases}
	\end{equation*}
	We have already encountered a quasi-embedding that is quite similar to~${f}$ in the proof of Lemma~\ref{lem:left_set_quasi_embeddings}~c) and, therefore, omit the proof that~${f}$ reflects orders.
	The proof of b) is analogous. We restrict the domain of~${f}$ to~${\btal_{\alpha, 1}}$ and only have to make sure that its range lies within~${\btal_{\omega^\gamma, Y}}$ for~${Y := \btal_{\delta, 1}}$, which is clear by definition of~$f$.
\end{proof}
Now, everything is ready for computing all upper bounds of binary trees with weakly ascending labels:
\begin{proposition}
	For any ordinal~${\alpha}$, we have~${o(\bta_{\alpha, 1}) \leq F(\alpha)}$.
\end{proposition}

\begin{proof}
	If~${\alpha = 0}$, then~${\bta_{\alpha, 1}}$ only contains the tree~${0 \star []}$. Thus,~${o(\bta_{\alpha, 1}) = 1}$ holds.
	If~${\alpha = \omega^{\gamma}}$ holds for an ordinal~${\gamma}$ satisfying~${\gamma < \varphi_\gamma(0)}$, then we use Corollary~\ref{cor:upper_bound_indecomposable_varphi}~c), which yields~${o(\bta_{\alpha, 1}) \leq \varphi_{1 + \gamma}(0)}$.
    
	Otherwise, write~${\alpha \eqnf \omega^{\gamma} + \delta}$. From Lemma~\ref{lem:decomposable_inner_label}~a) and Corollary~\ref{cor:upper_bound_indecomposable_varphi}~a), we get~${o(\bta_{\alpha, 1}) \leq o(\bta_{\omega^{\gamma}, X}) \leq \varphi_{1 + \gamma}(o(X)) \leq \varphi_{1 + \gamma}(F(\delta))}$ for~${X := \bta_{\delta, 1}}$.
\end{proof}
Next, we consider the upper bounds for both sequences with weak gap condition as well as trees with ascending labels that ascend strictly for left subtrees:
\begin{proposition}
	For any ordinal~${\alpha}$, we have~${o(\gapw_{\alpha}) \leq o(\btal_{\alpha, 1}) \leq G(\alpha)}$.
\end{proposition}

\begin{proof}
	By Lemma~\ref{lem:embed_weak_into_btal}, we know that there is a quasi-embedding from~${\gapw_{\alpha}}$ into~${\btal_{\alpha, 1}}$. Thus, we conclude~${o(\gapw_\alpha) \leq o(\btal_{\alpha, 1})}$. We prove~${o(\btal_{\alpha, 1}) \leq G(\alpha)}$ by induction along~${\alpha}$: First, if~${\alpha = 0}$ or~${\alpha = 1}$, we have~${o(\btal_{\alpha, 1}) = 1}$ or~${o(\btal_{\alpha, 1}) = \omega}$, respectively, by Lemma~\ref{lem:small_upper_bounds_btal}~a) and~b). If~${\alpha = n + 1}$ for~${n \in \n}$ with~${n > 0}$, then Lemma~\ref{lem:small_upper_bounds_btal}~c) yields~${o(\btal_{\alpha, 1}) \leq \omega^{\omega^{o(\btal_{n, 1})}} \leq \omega^{\omega^{G(n)}}}$.
	
	If~${\alpha}$ is infinite and we have~${\alpha = \omega^\gamma}$ for some ordinal~${\gamma}$ with~${0 < \gamma < \varphi_\gamma(0)}$, we invoke Corollary~\ref{cor:upper_bound_indecomposable_varphi}~d), which yields~${o(\btal_{\alpha, 1}) \leq \varphi_\gamma(0)}$. Finally, in any other case where~${\alpha \eqnf \omega^\gamma + \delta}$ holds for ordinals~${\gamma}$ and~${\delta}$, we have~${o(\btal_{\alpha, 1}) \leq o(\btal_{\omega^\gamma, X})}$ for~${X := \btal_{\delta, 1}}$ by Lemma~\ref{lem:decomposable_inner_label}~b). With Corollary~\ref{cor:upper_bound_indecomposable_varphi}~b), we arrive at the inequalities~${o(\btal_{\omega^\gamma, X}) \leq \varphi_\gamma(o(X)) \leq \varphi_\gamma(G(\delta))}$.
\end{proof}
We continue with sequences using the strong gap condition. The first step for their upper bounds lies in the next lemma. (Compare~a) to the proof of~\cite[Lemma~5.5]{SchuetteSimpson} as well as~\cite[Corollary~1 and Lemma~2]{RMWGapCondition}.)
\begin{lemma}\label{lem:quasi-embedding_strong_gap_condition}
	We have the following quasi-embeddings:
	\begin{enumerate}[label=\alph*)]
		\item ${\gaps_{n+1} \to \gaps_{n} \otimes \gapw_{n+1}}$ for any~${n \in \n}$.
		\item ${\gaps_{\alpha} \to \gaps_{\delta} \otimes (\gapw_{\alpha} \oplus (-1 + \gamma))^*}$ for any infinite ordinal~${\alpha \eqnf \omega^{\gamma} + \delta}$.
	\end{enumerate}
\end{lemma}

\begin{proof}
	For a), we define a map~${f: \gaps_{n+1} \to \gaps_{n} \otimes \gapw_{n+1}}$ with~${f(s) := (-1 + s_l, s_r)}$ for a sequence~${s = s_l * s_r}$ where~${s_l}$ only has positive members and it is as long as possible. Clearly, the decomposition of~${s}$ into~${s_l}$ and~${s_r}$ is unique and~${f}$ is, therefore, well-defined. We prove that~${f}$ is a quasi-embedding: Let~${s, t \in \gaps_{n+1}}$ with~${s = s_l * s_r}$ and~${t = t_l * t_r}$ such that~${s_l}$ and~${t_l}$ only contain positive members and are as long as possible. Assume~${f(s) \leq f(t)}$. This entails~${-1 + s_l \leq_\strong -1 + t_l}$ and~${s_r \leq_\weak t_r}$. Using Lemma~\ref{lem:left_subtraction}~a), the former results in~${s_l \leq_\strong t_l}$. By assumption,~${s_r}$ must be empty or begin with~${0}$. Otherwise,~${s_l}$ could have been chosen to be longer. Note that this already entails~${s_r \leq_\strong t_r}$ with respect to the strong gap condition: If~${s_r}$ is nonempty and, therefore, begins with~${0}$, the additional condition for the outer gap is trivially satisfied.	
	By Lemma~\ref{lem:gap_concat_weak_and_strong}, we arrive at our claim~${s_l * s_r \leq_{\strong} t_l * t_r}$.
	
	For b), let~${T}$ be the suborder of~${\gaps_{\alpha}}$ that consists of the empty sequence and also those that begin with an element strictly below~${\omega^\gamma}$.
	We continue with the definition of a map~${f: T \to (\gapw_{\alpha} \oplus (-1 + \gamma))^*}$ as follows:
	\begin{equation*}
		f(s) :=
		\begin{cases}
			\langle\rangle & \text{if } s = \langle\rangle\comma\\
			\langle \iota_0(s_r) \rangle * f(-1 + s_l) & \parbox[t]{17em}{if~${s = s_l * s_r}$ if~${s_l}$ only contains positive members and~$s_r$ begins with~$0$,}\\
            \langle \iota_0(\langle \rangle) * f(-1 + s) & \parbox[t]{17em}{if~$s$ is nonempty and its minimum is finite but positive,}\\
			\langle \iota_1(-1 + \rho_0) \rangle * f(-\omega^{\rho_0} + s) & \parbox[t]{17em}{if~${s}$ is nonempty and~${\rho \eqnf \omega^{\rho_0} + \rho'}$ is its minimum with~${\rho_0 > 0}$.}
		\end{cases}
	\end{equation*}
	Of course, we need to verify that~${f}$ is well defined: In the second, third, and fourth case,~$f$ is applied to a sequence that is either empty of begins with an element that clearly lies (weakly) below the first member of~$s$ and, hence, strictly below~$\omega^\gamma$. In the fourth case,~$\omega^{\rho_0}$ must be strictly below~$\omega^\gamma$ and, therefore,~${-1 + \rho_0}$ is strictly below~${-1 + \gamma}$.
	
	We show that~${f}$ is a quasi-embedding. Let~${s, t \in T}$ with~${f(s) \leq f(t)}$. The cases where one of~${s}$ or~${t}$ is the empty sequence are clear. Next, we assume that our assumed inequality holds since~${f(s)}$ is less than or equal to the tail of~${f(t)}$. Here, we have one of~${s \leq_{\strong} -1 + t_l}$,~${s \leq_{\strong} -1 + t}$, or~${s \leq_{\strong} -\omega^{\rho_0} + t}$. Clearly, each case entails~${s \leq_{\strong} t}$.

    Otherwise, assume that~${f(s)}$ is not less than or equal to the tail of~${f(t)}$. First, we consider the case in which both~$s$ and~$t$ have a finite minimum. If both contain~$0$, we derive~${s_r \leq_{\weak} t_r}$ and~${-1 + s_l \leq_\strong -1 + t_l}$, which entails~${s_l \leq_\strong t_l}$ by Lemma~\ref{lem:left_subtraction}~a). In fact, the former leads to~${s_r \leq_{\strong} t_r}$ since~$s_r$ is either empty or begins with~$0$, which has the effect that the outer gap condition becomes trivial. Now, we can concatenate both results using Lemma~\ref{lem:gap_concat_weak_and_strong} in order to arrive at~${s_l * s_r \leq_{\strong} t_l * t_r}$. It is not possible that~$s$ contains a~$0$ while~$t$ does not, since~$s_r$ can never be empty. Assume that~$s$ has a positive minimum and~$t$ contains a~$0$. Then, by induction hypothesis, we get~${-1 + s \leq_{\strong} -1 + t_l}$, which leads to~${s \leq_{\strong} t}$ by Lemma~\ref{lem:left_subtraction}~a) and~${t_l \leq_\strong t}$. If both~$s$ and~$t$ have a positive minimum, then~${-1 + s \leq_\strong -1 + t}$ leads to~${s \leq_\strong t}$ by Lemma~\ref{lem:left_subtraction}~a).
	
	If~$s$ and~$t$ both belong to the last case in the definition of~$f$, then this implies that~${s}$ and~${t}$ have~${\sigma \eqnf \omega^{\sigma_0} + \sigma'}$ and~${\tau \eqnf \omega^{\tau_0} + \tau'}$ as their minimum, respectively, for positive~$\sigma_0$ and~$\tau_0$. Also, we have~${-1 + \sigma_0 \leq -1 + \tau_0}$ and, by induction hypothesis,~${-\omega^{\sigma_0} + s \leq_\strong -\omega^{\tau_0} + t}$. Invoking Lemma~\ref{lem:left_subtraction}~a), this yields~${s \leq_\strong t}$.
	
	Finally, we proceed similar to the proof of a): We construct~${g: \gaps_{\alpha} \to \gaps_{\delta} \otimes T}$ with~${g(s) := (-\omega^{\gamma} + s_l, s_r)}$ for any sequence~${s = s_l * s_r}$ where~${s_l}$ only has members greater than or equal to~${\omega^{\gamma}}$ and is as long as possible. This immediately entails that the first element of~${s_r}$, if it exists, must be strictly below~${\omega^{\gamma}}$. We show that~${g}$ is a quasi-embedding. Then, combining it with~${f}$ will yield our claim.
	Let~${s, t \in \gaps_{\alpha}}$ with~${s = s_l * s_r}$ and~${t = t_l * t_r}$ such that~${s_l}$ and~${t_l}$ only have members greater than or equal to~${\omega^{\gamma_0}}$ and are as long as possible. We assume~${g(s) \leq g(t)}$. This entails~${-\omega^{\gamma_0} + s_l \leq_\strong -\omega^{\gamma_0} + t_l}$ and~${s_r \leq_\strong t_r}$. By Lemma~\ref{lem:left_subtraction}~a), the former yields~${s_l \leq_\strong t_l}$. Finally, with Lemma~\ref{lem:gap_concat_weak_and_strong}, we conclude our claim~${s_l * s_r \leq_\strong t_l * t_r}$.
\end{proof}
Finally, we compute the upper bounds for sequences with strong gap condition by simply calculating (upper bounds of) the maximal order types of the well partial orders that acted as codomains for the quasi-embeddings of the previous lemma:
\begin{proposition}
	For any ordinal~${\alpha}$, we have~${o(\gaps_{\alpha}) \leq H(\alpha)}$.
\end{proposition}

\begin{proof}
	If~${\alpha = 0}$, then~${\gaps_{\alpha}}$ only consists of the empty sequence. Therefore, we have~${o(\gaps_{\alpha}) = 1}$. If~${\alpha = n + 1}$ for some~${n \in \n}$, we can invoke Lemma~\ref{lem:quasi-embedding_strong_gap_condition}~a), which yields~${o(\gaps_{n+1}) \leq o(\gaps_{n}) \hessMul o(\gapw_{n+1}) \leq H(n) \hessMul G(n+1) = G(n+1) \cdot H(n)}$. During the last step, we used~${H(n) \leq G(n+1)}$ together with the fact that~${G(n+1)}$ is of the form~${\omega^{\omega^\beta}}$ for some ordinal~${\beta}$.
	Otherwise, if~${\alpha \eqnf \omega^{\gamma} + \delta}$ with~${\gamma > 0}$, then we invoke Lemma~\ref{lem:quasi-embedding_strong_gap_condition}~b), which yields
	\begin{align*}
		o(\gaps_{\alpha}) \leq{} &o(\gaps_{\delta}) \hessMul o((\gapw_{\alpha} \oplus (-1 + \gamma))^*)
		\leq{} H(\delta) \hessMul \omega^{\omega^{G(\alpha) \hess \gamma}}
		\leq{} \omega^{\omega^{G(\alpha) + \gamma}} \cdot H(\delta)\\
		={} &\omega^{G(\alpha) \cdot \omega^{\gamma}} \cdot H(\delta)
		={} G(\alpha)^{\omega^{\gamma}} \cdot H(\delta)\period
	\end{align*}
    In the second inequality, we used the following case distinction: If~${\gamma > 0}$ is finite, then~${-1 + \gamma = \gamma - 1}$ holds, which entails
    \begin{equation*}
        o((\gapw_{\alpha} \oplus (-1 + \gamma))^*) \leq \omega^{\omega^{(G(\alpha) \hess (-1 + \gamma)) + 1}} = \omega^{\omega^{G(\alpha) \hess \gamma}}\text{.}
    \end{equation*}
    Otherwise, if~$\gamma$ is infinite, then~${-1 + \gamma = \gamma}$ holds and, since~$G(\alpha)$ is infinite due to~${\alpha > 0}$, it is also guaranteed that~${o((\gapw_{\alpha} \oplus (-1 + \gamma))^*)}$ cannot be the sum of an~$\varepsilon$-number and a natural one. Hence, we derive~${o((\gapw_{\alpha} \oplus (-1 + \gamma))^*) \leq \omega^{\omega^{G(\alpha) \hess \gamma}}}$.
    
	In the third inequality, we used~${H(\delta) \leq \omega^{\omega^{G(\alpha) \hess \gamma}}}$ together with the fact that the ordinal~${\omega^{\omega^{G(\alpha) \hess \gamma}}}$ is of the form~${\omega^{\omega^{\zeta}}}$ for some~${\zeta}$. Moreover, we replaced~${G(\alpha) \hess \gamma}$ with~${G(\alpha) + \gamma}$ since~${\gamma \leq G(\alpha)}$ holds and~${G(\alpha)}$ is additively indecomposable. In the last two equalities, we used the fact that~$G(\alpha)$ is an~$\varepsilon$-number since~${\alpha > 0}$ holds.
\end{proof}

\subsection{Lower bounds}

In this part, we calculate all lower bounds of Theorem~\ref{thm:maximal_order_types}. This will be done by simply constructing quasi-embeddings from the claimed ordinals into the respective well partial orders.
We begin by defining notation for a variation of the Cantor normal form using the Veblen hierarchy (cf.~\cite[Section~3.4.3]{Pohlers}:

\begin{definition}[Veblen normal form]
	Given ordinals~${\alpha}$, with~${\gamma_0, \dots, \gamma_{n-1}}$, as well as~${\delta_0, \dots, \delta_{n-1}}$, we write
	\begin{equation*}
		\alpha \phinf \varphi_{\gamma_0}(\delta_0) + \dots + \varphi_{\gamma_{n-1}}(\delta_{n-1})
	\end{equation*}
	if~${\alpha}$ is equal to the right hand side and~${\varphi_{\gamma_i}(\delta_i) \geq \varphi_{\gamma_j}(\delta_j)}$ holds for all~${i < j < n}$ as well as~${\delta_i < \varphi_{\gamma_i}(\delta_i)}$ for all~${i < n}$. We may also abbreviate~${\alpha \phinf \varphi_{\gamma_0}(\delta_0) + \delta}$, where~${\delta}$ stands for the term~${\varphi_{\gamma_1}(\delta_1) + \dots + \varphi_{\gamma_{n-1}}(\delta_{n-1})}$.
\end{definition}
We will use the well-known result that any ordinal~${\alpha}$ can be written in such a normal form.
Now, we begin with the first embeddings into sequences ordered using the weak gap condition.
The following lemma corresponds to Proposition~\ref{prop:embed_varphi_into_weak_gap_sequence} (see also \cite{Gordeev}).
\begin{lemma}\label{lem:quasi-embedding_varphi_into_weak}
	Assume that there is a quasi-embedding~${f: \beta \to \gapw_{\rho} \setminus \{\langle \rangle\}}$. Then, we can construct a quasi-embedding~${g: \varphi_\alpha(\beta) \to \gapw_{\omega^\alpha + \rho}}$.
\end{lemma}
\begin{proof}
	We define~${g}$ as follows:
	\begin{equation*}
		g(\sigma) :=
		\begin{cases}
			\langle 0 \rangle & \text{if~${\sigma = 0}$,}\\
			g(\gamma) * \langle 0 \rangle * g(\delta) &
			\parbox[t]{19em}{
				if~${\sigma \phinf \gamma + \delta}$ where~${\gamma}$ is indecomposable and~${\delta}$ is positive,
			}\\
			\omega^{\gamma} + g(\delta) & \text{if } \sigma \phinf \varphi_\gamma(\delta) \text{ and } \gamma < \alpha\comma\\
			\omega^{\alpha} + f(\delta) & \text{if } \sigma \phinf \varphi_\alpha(\delta)\period
		\end{cases}
	\end{equation*}
	It can easily be seen that the range of~${g}$ lies within~${\gapw_{\omega^{\alpha} + \rho}}$.
	We prove that~${g}$ is a quasi-embedding by induction on Veblen normal forms. Let~${\sigma, \tau < \varphi_\alpha(\beta)}$ with~${g(\sigma) \leq_\weak g(\tau)}$. First, assume that~${\tau}$ is equal to~${0}$. Since~${f}$ and, therefore, also~${g}$ never map to empty sequences, this entails~${\sigma = 0}$. In the following, we assume that both~${\sigma}$ and~${\tau}$ are positive.
	
	First, assume that~${\sigma \phinf \gamma + \delta}$ and~${\tau \phinf \gamma' + \delta'}$ are decomposable. By Lemma~\ref{lem:gap_split_weak_weak_or_strong}, there are sequences~${s_l}$ and~${s_r}$ with~${s_l * s_r = g(\sigma)}$ such that both the inequalities~${{s_l \leq_\weak g(\gamma')}}$ and~${s_r \leq_\weak \langle 0 \rangle * g(\delta')}$ hold. If~${s_l}$ is empty, then the inequality~${g(\sigma) = s_r \leq_\weak \langle 0 \rangle * g(\delta')}$ even implies~${g(\sigma) \leq_\weak g(\delta')}$ since~${g(\gamma)}$ is nonempty and contains no zeros. The induction hypothesis yields~${\sigma \leq \delta' < \tau}$. If~${s_r}$ is empty, then we have~${g(\sigma) = s_l \leq_\weak g(\gamma')}$ and, by induction hypothesis,~${\sigma \leq \gamma' < \tau}$.
	Finally, if~${s_l}$ and~${s_r}$ are both nonempty, then Lemma~\ref{lem:gap_split_weak_weak_or_strong} tells us that already~${s_r \leq_\strong \langle 0 \rangle * g(\delta')}$ holds, i.e., with respect to the strong gap condition. This entails that~${s_r}$ must begin with~${0}$. Therefore,~${g(\gamma)}$ must lie within~${s_l}$, i.e.,~${g(\gamma) \leq_\weak s_l \leq_\weak g(\gamma')}$ holds and, by induction hypothesis,~${\gamma \leq \gamma'}$. If this relation is strict, we are finished. Otherwise, if~${\gamma = \gamma'}$ holds, then this yields~${g(\gamma) * \langle 0 \rangle = g(\gamma') * \langle 0 \rangle}$. Using Lemma~\ref{lem:gap_remove_head_weak}, we arrive at~${g(\delta) \leq_\weak g(\delta')}$ and, by induction hypothesis, at~${\delta \leq \delta'}$.
	
	Assume that~${\sigma \phinf \gamma + \delta}$ is decomposable but~${\tau \phinf \varphi_{\gamma'}(\delta')}$ is not. Using the inequalities~${g(\gamma) \leq_\weak g(\tau)}$ and~${g(\delta) \leq_\weak g(\tau)}$, we get both~${\gamma \leq \tau}$ and~${\delta \leq \tau}$. In particular, both must be strict as~$g(\gamma)$ and~$g(\delta)$ are proper subsequences of~$g(\sigma)$ by definition of~$g$. Since~${\varphi_{\gamma'}(\delta')}$ is closed under sums, we conclude~${\sigma \leq \tau}$.
	
	Assume that~${\sigma \phinf \varphi_\gamma(\delta)}$ is indecomposable but~${\tau \phinf \gamma' + \delta'}$ is decomposable. Similar to before, we find~${s_l}$ and~${s_r}$ such that~${s_l * s_r = g(\sigma)}$,~${s_l \leq_\weak g(\gamma')}$, and~${s_r \leq_\weak \langle 0 \rangle * g(\delta')}$ hold. If~${s_l}$ is empty, then~${g(\sigma) = s_r \leq_\weak \langle 0 \rangle * g(\delta')}$ yields the inequality~${g(\sigma) \leq_\weak g(\delta')}$ since~${g(\sigma)}$ does not contain any zeros. By induction hypothesis, we arrive at~${\sigma \leq \delta' < \tau}$. If~${s_r}$ is empty, then we have~${g(\sigma) = s_l \leq_\weak g(\gamma')}$ and, by induction hypothesis,~${\sigma \leq \gamma' < \tau}$.
	Finally, if both~${s_l}$ and~${s_r}$ are nonempty, then Lemma~\ref{lem:gap_split_weak_weak_or_strong} yields a strong gap condition for~${s_r \leq_\strong \langle 0 \rangle * g(\delta')}$, which entails that~${s_r}$ must begin with~${0}$. However,~${g(\sigma)}$ does not contain any zeros, which leads to a contradiction.
	
	In the following,~${\sigma \phinf \varphi_\gamma(\delta)}$ and~${\tau \phinf \varphi_{\gamma'}(\delta')}$ are both indecomposable. Assume that~${\gamma}$ and~${\gamma'}$ are both strictly less than~${\alpha}$. If~${\gamma < \gamma'}$ holds, then we have~${g(\delta) \leq_\weak \omega^{\gamma} + g(\delta) \leq_\weak g(\tau)}$. By induction hypothesis, this yields~${\delta \leq \tau}$ and, by definition of the Veblen hierarchy, we conclude~${\varphi_\gamma(\delta) \leq \varphi_{\gamma'}(\delta')}$ (we even have a strict inequality since we are working with normal forms). If~${\gamma = \gamma'}$ holds, then~${\omega^\gamma + g(\delta) \leq_\weak \omega^{\gamma'} + g(\delta')}$ implies~${g(\delta) \leq_\weak g(\delta')}$ by Lemma~\ref{lem:left_subtraction}~b). Our induction hypothesis yields~${\delta \leq \delta'}$ and we arrive at our claim. If~${\gamma > \gamma'}$, then the inequality~${\omega^{\gamma} + g(\delta) \leq_\weak \omega^{\gamma'} + g(\delta')}$ entails~${g(\sigma) = \omega^{\gamma} + g(\delta) \leq_\weak g(\delta')}$ by Lemma~\ref{lem:left_subtraction}~c). Combining this with our induction hypothesis produces~${\sigma \leq \delta' \leq \varphi_{\gamma'}(\delta')}$.
	
	If~${\gamma < \alpha}$ and~${\gamma' = \alpha}$ hold, then we have~${g(\delta) \leq_\weak \omega^{\gamma} + g(\delta) = g(\sigma) \leq_\weak g(\tau)}$. By induction hypothesis, this yields~${\delta \leq \tau}$ and, thus, our claim~${\varphi_\gamma(\delta) \leq \varphi_\alpha(\delta')}$. If~${\gamma = \alpha}$ and~${\gamma' < \alpha}$ hold, then~${\omega^{\alpha} + f(\delta) \leq_\weak \omega^{\gamma'} + g(\delta')}$ already leads us to the inequality~${g(\sigma) = \omega^{\alpha} + f(\delta) \leq_\weak g(\delta')}$ by Lemma~\ref{lem:left_subtraction}~c). The induction hypothesis then yields~${\sigma \leq \delta' \leq \tau}$. Finally, if~${\gamma = \gamma' = \alpha}$ holds, then~${g(\sigma) \leq_\weak g(\tau)}$ implies the inequality~${\omega^{\alpha} + f(\delta) \leq_\weak \omega^{\alpha} + f(\delta')}$ and, therefore,~${f(\delta) \leq_\weak f(\delta')}$ by Lemma~\ref{lem:left_subtraction}~b). Since~${f}$ is a quasi-embedding, we get~${\delta \leq \delta'}$, which results in our claim.
\end{proof}
We continue with a similar result for trees:
\begin{lemma}\label{lem:quasi-embedding:varphi_into_bta}
	Let~${f: \beta \to \bta_{\rho, 1} \setminus \{0 \star []\}}$ be a quasi-embedding. Then, we can construct a quasi-embedding~${g: \varphi_{1 + \alpha}(\beta) \to \bta_{\omega^\alpha + \rho, 1}}$.
\end{lemma}

\begin{proof}
	We define\footnote{Notice the slightly cumbersome notation in the second case. While one might be tempted to simply write~``$\sigma \phinf \omega^\gamma + \delta$'', this would exclude the cases in which~$\gamma$ is an~$\varepsilon$-number.}
	\begin{equation*}
		g(\sigma) :=
		\begin{cases}
			0 \star [0 \star [], 0 \star []] &\text{if~${\sigma = 0}$,}\\
			0 \star [g(\gamma), g(\delta)] &\text{if~${\sigma \phinf \tilde{\gamma} + \delta}$ with~${\tilde{\gamma} = \omega^\gamma}$ and~${\delta > 0}$,}\\
			\omega^\gamma + g(\delta) &\text{if~${\sigma \phinf \varphi_{1 + \gamma}(\delta)}$ and~${\gamma < \alpha}$,}\\
			\omega^\alpha + f(\delta) &\text{if~${\sigma \phinf \varphi_{1 + \alpha}(\delta)}$.}
		\end{cases}
	\end{equation*}
	We argue that~${g}$ only produces valid trees, i.e., with weakly ascending labels: Clearly, if~${\sigma = 0}$ holds, this is true. For~${\sigma \phinf \tilde{\gamma} + \delta}$, this follows by induction. Finally, in the case of~${\sigma \phinf \varphi_{1 + \gamma}(\delta)}$ our claim holds for~${g(\delta)}$ (and~${f(\delta)}$ in the case~${\gamma = \alpha}$). We know that this property stays intact even after adding~${\omega^{\gamma}}$ from the left.

	We prove that~${g}$ is a quasi-embedding by showing that~${g(\sigma) \leq g(\tau)}$ implies~${\sigma \leq \tau}$ using induction along Veblen normal forms:
	
	First, assume that~${\tau = 0}$ holds. Clearly,~${\sigma}$ cannot be of the form~${\sigma \phinf \tilde{\gamma} + \delta}$ with~${\tilde{\gamma} = \omega^\gamma}$ as this would lead to the contradiction~${0 \star [g(\gamma), g(\delta)] \leq 0 \star [0 \star [], 0 \star []]}$. Now, assume that~${\sigma \phinf \varphi_{1 + \gamma}(\delta)}$ holds. Using a simple induction, it is clear that any node in the range of~${g}$ has at least one inner node. By definition, the same holds for~${f}$. Thus,~${g(\sigma)}$ must have a positive inner node and can, therefore, not be embedded into~${g(\tau) = 0 \star [0 \star [], 0 \star []]}$.
	
	Assume~${\tau \phinf \tilde{\gamma}' + \delta'}$ with~${\tilde{\gamma}' = \omega^{\gamma'}}$. If~${g(\sigma) \leq \tau}$ holds because of~${g(\sigma) \leq g(\gamma')}$ (or~${g(\sigma) \leq g(\delta')}$), then our induction hypothesis yields~${\sigma \leq \gamma' \leq \tau}$ (or~${\sigma \leq \delta' \leq \tau}$). Therefore, assume that the root label of~${g(\sigma)}$ is equal to zero and the left and right subtree of~${g(\sigma)}$ embed into the left and right subtree of~${g(\tau)}$, respectively. By previous considerations, this means that~${\sigma}$ is of the form~${\sigma \phinf \tilde{\gamma} + \delta}$ with~${\tilde{\gamma} = \omega^\gamma}$ and we have both~${g(\gamma) \leq g(\gamma')}$ and~${g(\delta) \leq g(\delta')}$. By induction hypothesis, this yields~${\gamma \leq \gamma'}$ and~${\delta \leq \delta'}$. We conclude~${\sigma \leq \tau}$.
	
	Assume~${\tau \phinf \varphi_{1 + \gamma'}(\delta')}$ with~${\gamma' < \alpha}$. If~${\sigma \phinf \tilde{\gamma} + \delta}$ holds with~${\tilde{\gamma} = \omega^\gamma}$, then we have~${g(\gamma), g(\delta) \leq g(\sigma) \leq g(\tau)}$. By induction hypothesis, this yields~${\gamma, \delta \leq \tau}$ and also~${\tilde{\gamma} \leq \tau}$. In particular, these inequalities are strict. Now, since~${\tau}$ is indecomposable, our claim follows. If~${\sigma \phinf \varphi_{1 + \gamma}(\delta)}$ holds with~${\gamma < \alpha}$, there are three cases. If~${\gamma < \gamma'}$, then~${g(\delta) \leq g(\sigma) \leq g(\tau)}$ together with our induction hypothesis yields~${\delta \leq \tau}$ and, thus, our claim. If~${\gamma = \gamma'}$ holds, then~${\omega^{\gamma} + g(\delta) \leq \omega^{\gamma} + g(\delta')}$ entails~${\delta \leq \delta'}$ by Lemma~\ref{lem:left_subtraction}~b) and, hence, our claim. If~${\gamma > \gamma'}$ holds, then the inequality~${\omega^{\gamma} + g(\delta) \leq \omega^{\gamma'} + g(\delta')}$ implies~${\omega^{\gamma} + g(\delta) \leq g(\delta')}$ by Lemma~\ref{lem:left_subtraction}~c) and, therefore, our claim. For the same reason, if~${\sigma \phinf \varphi_{1 + \alpha}(\delta)}$ holds, then the inequality~${\omega^{\alpha} + f(\delta) \leq \omega^{\gamma'} + g(\delta')}$ entails~${\omega^{\alpha} + f(\delta) \leq g(\delta')}$ and we arrive at our claim.
	
	Assume~${\tau \phinf \varphi_{1 + \alpha}(\delta')}$. If~${\sigma \phinf \tilde{\gamma} + \delta}$ holds with~${\tilde{\gamma} = \omega^\gamma}$, then our argument is as in the previous paragraph. If~${\sigma \phinf \varphi_{1 + \gamma}(\delta)}$ holds for~${\gamma < \alpha}$, then~${g(\delta) \leq g(\sigma) \leq g(\tau)}$ leads to our claim using the induction hypothesis. Finally, if we have~${\sigma \phinf \varphi_{1 + \alpha}(\delta)}$, then~${g(\sigma) \leq g(\tau)}$ entails~${f(\delta) \leq f(\delta')}$ by Lemma~\ref{lem:left_subtraction}~b). Since~${f}$ is a quasi-embedding, this leads to~${\delta \leq \delta'}$ and, hence, our claim.
\end{proof}
In the case where~${\alpha}$ is a~${\Gamma}$-ordinal\footnote{By this, we mean an ordinal~$\alpha$ satisfying~${\varphi_\alpha(0) = \alpha}$.}, we can not only embed~${\varphi_\alpha(0)}$ into~${\gapw_\alpha}$ but even~${\varphi_\alpha(1)}$. This stems from the fact that the Veblen hierarchy has fixed points whereas the order types of~${\gapw_\alpha}$ do not.
\begin{lemma}\label{lem:quasi-embedding_varphi_1_into_weak}
	Let~${\alpha}$ be a~${\Gamma}$-ordinal. Then, there exists a map~${f: \varphi_\alpha(1) \to \gapw_{\alpha}}$ that is a quasi-embedding.
\end{lemma}

\begin{proof}
	We define~${f}$ as follows:
	\begin{equation*}
		f(\sigma) :=
		\begin{cases}
			\langle \sigma \rangle & \text{if } \sigma < \alpha\comma\\
			\langle 1, 1 \rangle & \text{if } \sigma = \alpha\comma\\
			f(\gamma) * \langle 0 \rangle * f(\delta) & \parbox[t]{15em}{if~${\sigma \phinf \gamma + \delta > \alpha}$ where~${\gamma}$ is indecomposable and~${\delta}$ is positive,}\\
			\omega^{\gamma} + f(\delta) & \text{if } \sigma \phinf \varphi_\gamma(\delta) > \alpha\period
		\end{cases}
	\end{equation*}
	We verify that the range of~${f}$ lies within~${\gapw_{\alpha}}$. Since~${\alpha}$ is a~${\Gamma}$-ordinal, we can use the fact~${\alpha = \varphi_\alpha(0)}$. The case for~${\sigma < \alpha}$ is trivial. Also,~${1 < \varphi_\alpha(0)}$ clearly holds. If~${\sigma}$ is decomposable, we apply induction and the fact~${0 < \varphi_\alpha(0)}$. Finally, if~${\sigma \phinf \varphi_\gamma(\delta) > \alpha}$, we know that~${\gamma}$ must be strictly smaller than~${\alpha}$: Assume that it is equal to~${\alpha}$. Then,~${\delta = 0}$ must hold in order to satisfy~${\sigma < \varphi_\alpha(1)}$. However,~${\sigma = \varphi_\alpha(0)}$ contradicts~${\sigma > \alpha}$. Now, left-addition of~${\omega^{\gamma}}$ for~${\gamma < \alpha}$ to any element in~${\alpha}$ results, again, in an element in~${\alpha}$ since, as a~${\Gamma}$-ordinal,~${\alpha}$ must be indecomposable.
	
	We prove that~${f}$ is a quasi-embedding. Let~${\sigma, \tau \in \varphi_\alpha(1)}$ be such that~${f(\sigma) \leq f(\tau)}$ holds. Using induction along Veblen normal forms, we prove that this implies~${\sigma \leq \tau}$. If both~${\sigma}$ and~${\tau}$ strictly lie below~${\alpha}$, then~${\langle \sigma \rangle = f(\sigma) \leq f(\tau) = \langle \tau \rangle}$ immediately results in our claim. Assume for contradiction that~${\sigma \geq \alpha}$ and~${\tau < \alpha}$ hold. Now,~${\sigma = \alpha}$ cannot hold since~${\langle 1, 1 \rangle}$ can never be less than or equal to~${\langle \sigma \rangle}$. If~${\sigma \phinf \gamma + \delta > \alpha}$ holds, then~${f(\sigma)}$ has at least a length of~${3}$. Thus, for the same reason as before,~${f(\sigma)}$ can never be less than or equal to~${f(\tau)}$. Finally, if~${\sigma \phinf \varphi_\gamma(\delta) > \alpha}$ holds, then~${\gamma}$ must be strictly below~${\alpha}$ (as we have seen in the previous paragraph). Now, in order to satisfy~${\varphi_\gamma(\delta) > \alpha = \varphi_\alpha(0)}$, already~${\delta}$ must be greater than~${\alpha}$. Thus,~${f(\delta)}$ cannot be less than or equal to~${f(\tau)}$. The assumption~${f(\sigma) \leq f(\tau)}$ leads then to a contradiction with~${f(\delta) \leq \omega^{\gamma} + f(\delta) \leq f(\tau)}$.
	
	Assume for contradiction that~${\sigma > \alpha}$ and~${\tau = \alpha}$ hold. If~${\sigma \phinf \gamma + \delta}$ is decomposable, then~${\gamma}$ must be greater than~${\alpha}$. If~${\gamma = \alpha}$ holds, then~${f(\sigma)}$ clearly has more than~${2}$ members. If~${\gamma > \alpha}$ holds, then we apply our induction hypothesis to~${\gamma}$ and yield the contradiction~${f(\gamma) \leq f(\sigma) \leq f(\tau)}$. If~${\sigma \phinf \varphi_\gamma(\delta)}$ is not decomposable, then (using the same argument as in the previous paragraph),~${\delta}$ must be greater than or equal to~${\alpha}$. With~${\delta > \alpha}$, we apply the induction hypothesis to~${\delta}$ and arrive at the contradiction~${f(\delta) < f(\sigma) \leq f(\tau)}$.
	Finally, for the case where both~${\sigma}$ and~${\tau}$ are strictly above~${\alpha}$, we simply reuse the reasoning from Lemma~\ref{lem:quasi-embedding_varphi_into_weak}.
\end{proof}
Similar to before, this result can be transferred to trees:
\begin{corollary}\label{cor:quasi-embedding_varphi_into_bta_gamma}
	Let~${\alpha}$ be a~${\Gamma}$-ordinal. Then, there exists a map~${\varphi_\alpha(1) \to \bta_{\alpha, 1}}$ that is a quasi-embedding.
\end{corollary}

\begin{proof}
	By Lemma~\ref{lem:embed_weak_into_btal}, there is a quasi-embedding~${\gapw_{\alpha} \to \bta_{\alpha, 1}}$. Thus, we can reuse the result from Lemma~\ref{lem:quasi-embedding_varphi_1_into_weak}.
\end{proof}
At this point, most of the work for the lower bounds of~${\gapw_\alpha}$,~${\btal_{\alpha, 1}}$, and~${\bta_{\alpha, 1}}$ is done. For those of the former two, the cases where~${\alpha}$ is a natural number, however, require further inspection. We introduce a new order on sequences that is quite similar to the order by Higman. The only difference is a new fourth case:
\begin{definition}
	Let~${X}$ be a partial order. We define~${X^\bullet}$ to be the relation on (the underlying set of)~${X^*}$ that is the smallest relation satisfying the following properties:
	\begin{enumerate}[label=\roman*)]
		\item ${\langle \rangle \leq s}$ for any~${s \in X^\bullet}$,
		\item ${\langle x \rangle * s \leq \langle y \rangle * t}$ for any~${x, y \in X}$ with~${x \leq y}$ and~${s, t \in X^\bullet}$ with~${s \leq t}$,
		\item ${s \leq \langle y \rangle * t}$ for any~${y \in X}$ and~${s, t \in X^\bullet}$ with~${s \leq t}$,
		\item ${s_l * s_r \leq \langle y \rangle * t}$ for any~${y \in X}$ and~${s_l, s_r, t \in X^\bullet}$ with~${(s_l)_i < y}$ for all~${i < |s_l|}$ and~${s_r \leq t}$.
	\end{enumerate}
\end{definition}
Using the gap condition, we want to quasi-embed~${\alpha^*}$ into~${\gapw_\beta}$ for large~${\alpha}$ and small~$\beta$. However, it seems that for any such quasi-embedding, the partial order on~${\alpha^*}$ is too strict. We will see this in the proof of Lemma~\ref{lem:quasi-embedding:alpha_dot_into_weak}. The additional comparisons that we need in our proof are given by the new case iv) in the definition of~${\alpha^\bullet}$. Later, in Remark~\ref{rem:cannot_embed_omega_star_into_gapw_2}, we show that, while Lemma~\ref{lem:quasi-embedding:alpha_dot_into_weak} yields a quasi-embedding~${\omega^\bullet \to \gapw_2}$, there cannot be a quasi-embedding~${\omega^* \to \gapw_2}$.

Now, we prove that our new partial order is transitive. For this, we use the following lemma:
\begin{lemma}\label{lem:split_dot}
	Let~${X}$ be a partial order. For any~${s, t_l, t_r \in X^\bullet}$ with~${s \leq t_l * t_r}$, we can find sequences~${s_l}$ and~${s_r}$ with~${s = s_l * s_r}$ such that~${s_l \leq t_l}$ and~${s_r \leq t_r}$ hold.
\end{lemma}

\begin{proof}
	We proceed by induction on the length of~${t_l}$. If~${t_l}$ is the empty sequence, then we simply choose~${s_l := \langle \rangle}$ and~${s_r := s}$. In the following, we assume that~${t_l = \langle y \rangle * t'_l}$ holds for some element~${y \in X}$ and a sequence~${t'_l}$.
	
	If~${s \leq t_l * t_r}$ holds because of i), we have~${s = \langle \rangle}$ and are immediately finished. If it holds because of ii), then~${s = \langle x \rangle * s'}$ holds for some element~${x \in X}$ and a sequence~${s'}$ together with~${x \leq y}$ and~${s' \leq t'_l * t_r}$. By induction hypothesis, we find~${s'_l}$ and~${s'_r}$ with~${s' = s'_l * s'_r}$ and both~${s'_l \leq t'_l}$ and~${s'_r \leq t_r}$. We choose~${s_l := \langle x \rangle * s'_l}$ and~${s_r := s'_r}$. Now,~${s_l \leq t_l}$ holds by ii) and~${s_r \leq t_r}$ holds by definition.
	
	If~${s \leq t_l * t_r}$ holds because of iii), then we have~${s \leq t'_l * t_r}$. By induction hypothesis, we find~${s_l}$ and~${s_r}$ with~${s = s_l * s_r}$ and both~${s_l \leq t'_l}$ and~${s_r \leq t_r}$. Now,~${s_l \leq t_l}$ holds by iii) and we are finished.
	
	Finally, if~${s \leq t_l * t_r}$ holds because of iv), then there are sequences~${s'_l}$ and~${s'_r}$ with~${s = s'_l * s'_r}$ such that we have~${(s'_l)_i < y}$ for all~${i < |s'_l|}$ together with~${s'_r \leq t'_l * t_r}$. By induction hypothesis, we find sequences~${s'_{rl}}$ and~${s'_{rr}}$ with~${s'_r = s'_{rl} * s'_{rr}}$ such that both~${s'_{rl} \leq t'_l}$ and~${s'_{rr} \leq t_r}$ hold. We set~${s_l := s'_l * s'_{rl}}$ and~${s_r := s'_{rr}}$. Now,~${s_l \leq t_l}$ holds by iv) and~${s_r \leq t_r}$ holds by definition.
\end{proof}
Next, we prove transitivity:
\begin{lemma}
	Let~${X}$ be a partial order. Then,~${X^\bullet}$ is transitive.
\end{lemma}

\begin{proof}
	Let~${s, t, u \in X^\bullet}$ with~${s \leq t}$ and~${t \leq u}$. We prove by induction on~${|s| + |t| + |u|}$ that this entails~${s \leq u}$. We omit all cases in which both~${s \leq t}$ and~${t \leq u}$ hold because of i)-iii) as these cases are already contained in the usual transitivity proof of the order~${X^*}$.
	
	Assume that~${s \leq t}$ holds and~${t \leq u}$ holds because of iv). We have~${t = t_l * t_r}$ and~${u = \langle z \rangle * u'}$ for sequences~${t_l, t_r, u' \in X^\bullet}$ and~${z \in X}$ such that~${(t_l)_i < z}$ holds for all~${i < |t_l|}$ together with~${t_r \leq u'}$. By Lemma~\ref{lem:split_dot}, we find sequences~${s_l}$ and~${s_r}$ with both~${s_l \leq t_l}$ and~${s_r \leq t_r}$. Clearly,~${(s_l)_i < z}$ holds for all~${i < |s_l|}$ since~${s_l \leq t_l}$ entails that any element in~${s_l}$ is less than or equal to some element in~${t_l}$, which a short induction shows. Moreover, we have~${s_r \leq u'}$ by transitivity, i.e., our induction hypothesis. Finally, by iv), this yields~${s \leq u}$.
	
	Assume that~${s \leq t}$ holds because of iv) and~${t \leq u}$ holds because of i)-iii). Clearly, the case where~${t \leq u}$ follows from iv) is already covered by the previous paragraph. There are sequences~${s_l}$ and~${s_r}$ with~${s = s_l * s_r}$ and a sequence~${t'}$ with an element~${y \in X}$ such that~${t = \langle y \rangle * t'}$ holds together with~${(s_l)_i < y}$ for all~${i < |s_l|}$ and~${s_r \leq t'}$.
	Assume that~${t \leq u}$ holds because of i). This immediately contradicts the fact that~${t}$ cannot be empty. Assume that~${t \leq u}$ holds because of ii), i.e., there is a sequence~${u'}$ with an element~${z \in X}$ such that~${u = \langle z \rangle * u'}$,~${y \leq z}$, and~${t' \leq u'}$ hold. Now, we have~${(s_l)_i < y \leq z}$ for all~${i < |s_l|}$ and~${s_r \leq t' \leq u'}$ by transitivity, i.e., our induction hypothesis. By iv), this yields~${s \leq u}$.
	
	Finally, assume that~${t \leq u}$ holds because of iii). Then, there is an element~${z \in X}$ together with a sequence~${u'}$ such that both~${u = \langle z \rangle * u'}$ and~${t \leq u'}$ hold. Now, we have~${s \leq u'}$ by induction hypothesis. Thus, iii) yields~${s \leq u}$.
\end{proof}
Next, we construct lower bounds for the order types of our new partial order on sequences. If the set of allowed values is singular, this yields the following expected result:
\begin{lemma}\label{lem:quasi-embedding_omega_dot}
	There is a quasi-embedding~${\omega \to 1^\bullet}$.
\end{lemma}

\begin{proof}
	We consider the usual map~${f: \omega \to 1^\bullet}$ that maps any natural number to the unique sequence in~${1^\bullet}$ of the same length. We verify that this constitutes a quasi-embedding: Let~${n, m \in \n}$ be such that~${f(n) \leq f(m)}$ holds. By induction on the sum~${n + m}$ we prove that this implies~${n \leq m}$. If~${m}$ equals zero, then~${f(n) \leq f(m) = \langle \rangle}$ immediately implies the same for~${n}$. We consider positive numbers~${n}$ and~${m}$ in the following. We proceed by case distinction on the properties i)-iv): Since~${f(n)}$ is not the empty sequence, i) cannot be the reason for our inequality. If it holds because of ii), then we have~${f(n-1) \leq f(m-1)}$ and, by induction hypothesis~${n-1 \leq m-1}$, which results in~${n \leq m}$. If it holds because of iii), then we have~${f(n) \leq f(m-1)}$ and by induction hypothesis~${n \leq m-1 \leq m}$. The last case iv) is identical to iii) since~${1}$ only consists of a single element.
\end{proof}
Before we continue with the lower bounds for sequences~${\alpha^\bullet}$, we need some fundamental properties:
\begin{lemma}\label{lem:properties_dot}
	Let~${X}$ be a partial order.
	\begin{enumerate}[label=\alph*)]
		\item If~$X$ has a top element~$\top$ and there are two sequences~${s = s_l * \langle \top \rangle * s_r}$ and~${t = t_l * \langle \top \rangle * t_r}$ in~${X^\bullet}$ with~${s \leq t}$ such that~${s_l}$ and~${t_l}$ do not contain~${\top}$, then already~${s_r \leq t_r}$ holds.
		\item Consider sequences~$r$,~${s = s_l * r}$, and~${t = t_l * r}$ in~${X^\bullet}$ with~${s \leq t}$. Then, we already have~${s_l \leq t_l}$.
	\end{enumerate}
\end{lemma}

\begin{proof}\mbox{}
	\begin{enumerate}[leftmargin=*, label=\alph*)]
		\item Assume that~${t_l}$ is empty. If~${s \leq t}$ holds because of ii) or iii), then~${s_r \leq t_r}$ follows by transitivity. If~${s \leq t}$ holds because of iv), then there are sequences~${s_{ll}}$ and~${s_{lr}}$ with~${s_l = s_{ll} * s_{lr}}$ and~${s_{lr} * \langle \top \rangle * s_r \leq t_r}$. The reason why this split must happen before the first~${\top}$ in~${s}$ lies in the fact that~${\top}$ cannot be strictly smaller than any element in~${X}$. From this, we get~${s_r \leq t_r}$ by transitivity.
		
		Assume that~${t_l = \langle y \rangle * t'_l}$ holds for some element~${y \in X}$ and sequence~${t'_l}$. If~${s \leq t}$ holds because of ii), then~${s_l}$ cannot be empty, i.e., it must be of the form~${s_l = \langle x \rangle * s'_l}$ for some element~${x \in X}$ and sequence~${s'_l}$. We have~${x \leq y}$ and~${s'_l * \langle \top \rangle * s_r \leq t'_l * \langle \top \rangle * t_r}$. Our claim follows by induction. If~${s \leq t}$ holds because of iii), then we have~${s \leq t'_l * \langle \top \rangle * t_r}$ and our claim follows by induction. Finally, assume that~${s \leq t}$ holds because of iv). Then, we split~${s_l}$ like before, i.e., we have~${s_l = s_{ll} * s_{lr}}$ with~${s_{lr} * \langle \top \rangle * s_r \leq t'_l * \langle \top \rangle * t_r}$. Again, the split must happen before the first~${\top}$ in~${s}$ because of the same considerations as before. Our claim follows by induction.
		\item First, we prove that for an element~${x \in X}$ and a sequence~${s \in X^\bullet}$, we get~${s \nleq \langle x \rangle}$ if~$x$ appears in~$s$ and this sequence consists of more than one member. We proceed by case distinction: It is clear that this relation does not hold because of~i). If it holds because of~ii) or~iii), then this entails that a nonempty sequence is less than or equal to the empty sequence. It is easy to see that this cannot be the case. Finally, if it holds because of~iv), then there are sequences~$s_l$ and~$s_r$ with~${s = s_l * s_r}$ as well as both~${(s_l)_i < x}$ for all~$i < |s_l|$ and~${s_r \leq \langle \rangle}$. Clearly, the latter entails~${s_r = \langle\rangle}$ and, hence,~${s_i < x}$ for all~${i < |s|}$. This contradicts our assumption that~$x$ appears in~$s$.

        Now, consider~${r, s, s_l, t, t_l \in X^\bullet}$ with~${s = s_l * r}$,~${t = t_l * r}$, and~${s \leq t}$. We show~${s_l \leq t_l}$ by induction on the length of~$r$. Our claim is trivial if~$s_l$ is empty. Therefore, assume that it is of the form~${s_l = \langle x \rangle * s'_l}$ for some~${x \in X}$ and~${s'_l \in X^\bullet}$. Now, we continue by side induction along the length of~$t_l$ and lead the assumption~${t_l = \langle \rangle}$ to a contradiction. Since~$s$ is nonempty, the same must hold for~$t$ and we can find~${y \in X}$ with~$r'$ such that~${r = \langle y \rangle * r'}$ holds. By our induction hypothesis, we know that this entails~${s_l * \langle y \rangle \leq \langle y \rangle}$. However, as~$s_l$ is nonempty, this leads to a contradiction by our previous thoughts.

        Assume that~$t_l$ is nonempty. Then, there must be~${y \in X}$ and~${t'_l \in X^\bullet}$ with~$t_l = \langle y \rangle * t'_l$. If~${s \leq t}$ holds because of~ii), we have~${x \leq y}$ and~${s'_l * r \leq t'_l * r}$. The latter leads to~${s'_l \leq t'_l}$ by our side induction hypothesis and, together with the former, we arrive at~${s \leq t}$. The argument is quite similar if~${s \leq t}$ holds because of~iii).
        
        Finally, if~${s \leq t}$ holds due to~iv), then there are~${\bar{s}_l, \bar{s}_r \in X^\bullet}$ with~${s = \bar{s}_l * \bar{s}_r}$ satisfying~${(\bar{s}_l)_i < y}$ for all~${i < |\bar{s}_l|}$ as well as~${\bar{s}_r \leq t'_l * r}$. If~$r$ is a final segment of~$\bar{s}_r$, then there is a sequence~$s_{lr}$ with~${s_l = \bar{s}_l * s_{lr}}$ and~${\bar{s}_r = s_{lr} * r}$. By our side induction hypothesis, we get~${s_{lr} \leq t'_l}$. This yields our claim~${s_l \leq t_l}$. Otherwise, if~$r$ is \emph{not} a final segment of~$\bar{s}_r$, then~$s_l$ must be an initial segment of~$\bar{s}_l$. In particular, we have~${(s_l)_i < y}$ for all~${i < |s_l|}$ and, therefore,~${s_l \leq t_l}$.\qedhere
	\end{enumerate}
\end{proof}
With this, we can construct quasi-embeddings into our new sequences. This corresponds to the lower bounds for infinite~${\alpha}$ given by Theorem~\ref{thm:Higman_order_types}.
\begin{lemma}\label{lem:quasi-embedding_omega_omega_alpha_into_dot_alpha}
	For any infinite ordinal~${\alpha}$, there is a quasi-embedding~${f: \omega^{\omega^{\alpha}} \to \alpha^\bullet}$.
\end{lemma}

\begin{proof}
	We begin by showing that if there is a quasi-embedding~${g: \beta \to \alpha^\bullet}$, then there is also a quasi-embedding~${h: \beta^\omega \to (\alpha + 1)^\bullet}$. In light of the fact due to de Jongh and Parikh, that maximal order types are always attained, it suffices to construct quasi-embeddings~${h_n: \beta^n \to (\alpha + 1)^\bullet}$ for every~${n \in \n}$. The definition of~${h_0}$ is simply done by mapping~${0}$ to, e.g., the empty sequence. Now, assume that~${h_n}$ has been defined for some~${n \in \n}$. We set
	\begin{equation*}
		h_{n+1}(\beta \cdot \gamma + \delta) := g(\delta) * \langle \alpha \rangle * h_n(\gamma)
	\end{equation*}
	for~${\gamma < \beta^n}$ and~${\delta < \beta}$. Let us verify that~${h_{n+1}}$ is, indeed, a quasi-embedding. For this, let~${\gamma, \gamma' < \beta^n}$ and~${\delta, \delta' < \beta}$ be such that~${h_{n+1}(\sigma) \leq h_{n+1}(\tau)}$ holds for ordinals~${\sigma := \beta \cdot \gamma + \delta}$ and~${\tau := \beta \cdot \gamma' + \delta'}$. We prove that this entails~${\sigma \leq \tau}$. From Lemma~\ref{lem:properties_dot}~a), it follows that~${h_n(\gamma) \leq h_n(\gamma')}$ holds. Since we already know that~${h_n}$ is a quasi-embedding, this entails~${\gamma \leq \gamma'}$. Now, if~${\gamma < \gamma'}$ holds, we immediately have~${\sigma < \tau}$ and are finished. Otherwise, assume~${\gamma = \gamma'}$. Clearly, this implies~${h_n(\gamma) = h_n(\gamma')}$. Thus, by Lemma~\ref{lem:properties_dot}~b), we have~${g(\delta) \leq g(\delta')}$. We conclude~${\delta \leq \delta'}$ and, finally,~${\sigma \leq \tau}$.
	
	We continue by showing the claim of our lemma for the smallest infinite ordinal, i.e.,~${\alpha = \omega}$. Starting from the simple quasi-embedding~${\omega \to 1^\bullet}$ (Lemma~\ref{lem:quasi-embedding_omega_dot}), consecutive applications of the previous paragraph yield maps~${\omega^{\omega^n} \to (n + 1)^\bullet}$ that are quasi-embeddings for any~${n \in \n}$. Thus, by de Jongh and Parikh, there is a quasi-embedding~${\omega^{\omega^{\omega}} \to \omega^\bullet}$.
	
	Next, we consider some infinite successor ordinal~${\alpha + 1}$ and assume that our claim has already been shown for~${\alpha}$. This step is immediate by the considerations of the first paragraphs if we put~${\beta := \omega^{\omega^{\alpha}}}$.
	
	Finally, let~${\alpha > \omega}$ be a limit ordinal. Thus, for every~${\beta < \alpha}$, there is a quasi-embedding~${\omega^{\omega^{\gamma}} \to \gamma^\bullet}$ if we set~${\gamma := \max(\beta, \omega)}$. Clearly, this can be extended into a quasi-embedding with codomain~${\alpha^\bullet}$. Therefore, by de Jongh and Parikh, we conclude that there must be a quasi-embedding~${\omega^{\omega^{\alpha}} \to \alpha^\bullet}$.
\end{proof}
Next, we prove the following quasi-embeddings from our new partial order on sequences into sequences with gap condition:
\begin{lemma}\label{lem:quasi-embedding:alpha_dot_into_weak}
	For any ordinal~${\alpha}$ and natural number~${n \in \n}$ such that there exists a quasi-embedding~${f: \alpha \to \gapw_{n}}$, we can find a quasi-embedding~${g: \alpha^\bullet \to \gapw_{n+1}}$.
\end{lemma}

\begin{proof}
	We define
	\begin{equation*}
		g(s) :=
		\begin{cases}
			\langle \rangle & \text{if } s = \langle \rangle\comma\\
			(1 + f(\alpha_0)) * \langle 0 \rangle * g(s') & \text{if } s = \langle \alpha_0 \rangle * s'\period
		\end{cases}
	\end{equation*}
	Let us prove that~${g}$ is a quasi-embedding. Assume~${g(s) \leq_\weak g(t)}$ holds for~${s, t \in \alpha^\bullet}$. By induction on the length of~${s}$ and~${t}$, we show that~${s \leq t}$ holds. Clearly,~${t = \langle \rangle}$ entails~${s = \langle \rangle}$. Thus, we assume that~${s}$ and~${t}$ are both not empty. We can write~${t = \langle \beta_0 \rangle * t'}$ for some~${\beta_0 \in \alpha}$ and~${t' \in \alpha^\bullet}$. Also, we write~${s = \langle \alpha_0, \dots, \alpha_{n-1} \rangle}$ with~${n > 0}$.
	
	By Lemma~\ref{lem:gap_split_weak_weak_or_strong}, we can find sequences~${u, v \in \gapw_{n+1}}$ with~${g(s) = u * v}$ and both inequalities~${u \leq_{\weak} 1 + f(\beta_0)}$ and~${v \leq_{\weak} \langle 0 \rangle * g(t')}$.
	Assume that~${u}$ is empty. If~${f(\alpha_0)}$ is empty as well, this leads to~${\langle 0 \rangle * g(\langle \alpha_1, \dots, \alpha_{n-1} \rangle) \leq_\weak \langle 0 \rangle * g(t')}$. Hence, we arrive at~${g(\langle \alpha_1, \dots, \alpha_{n-1} \rangle) \leq_\weak g(t')}$ and, by induction hypothesis, at~${\langle \alpha_1, \dots, \alpha_{n-1} \rangle \leq t'}$. This yields~${s \leq t}$ since~${\alpha_0 = 0}$ must hold as~${f(\alpha_0)}$ is the smallest element in~${\gapw_n}$. If~${f(\alpha_0)}$ is nonempty, then~${g(s) \leq_\weak \langle 0 \rangle * g(t')}$ leads to~${g(s) \leq g(t')}$ since~${g(s)}$ begins with a positive member. By induction hypothesis, we conclude~${s \leq t' \leq t}$.
	
	If~${u}$ is nonempty, Lemma~\ref{lem:gap_split_weak_weak_or_strong} yields~${v \leq_{\strong} \langle 0 \rangle * g(t')}$ with respect to the strong gap condition.
	Thus, we conclude that~${v}$ is either empty or starts with~${0}$. Combining this knowledge with the definition of~${g}$ yields some number~${0 < k < n}$ such that
	\begin{align*}
		u &={} (1 + f(\alpha_0)) * \langle 0 \rangle * \dots * \langle 0 \rangle * (1 + f(\alpha_{k-1}))\\
		\text{and } v &={} \langle 0 \rangle * (1 + f(a_k)) * \dots * \langle 0 \rangle * (1 + f(a_{n-1})) * \langle 0 \rangle
	\end{align*}
	hold. Note that~${u}$ contains~$(k-1)$-many zeros, while~${v}$ contains~${(n + 1 -k)}$-many zeros. By Lemma~\ref{lem:gap_strong_and_weak}, we know that already~${g(\langle a_k, \dots, a_{n-1} \rangle) \leq_\weak g(t')}$ holds. Thus, by induction hypothesis, we arrive at~${\langle a_k, \dots, a_{n-1} \rangle \leq t'}$.
    
    If~${k = 1}$ holds, we have~${1 + f(\alpha_0) = u \leq_\weak 1 + f(\beta_0)}$ and, thus,~${\alpha_0 \leq \beta_0}$. Together with~${\langle a_1, \dots, a_{n-1} \rangle \leq t'}$, this yields our claim.
	Finally, if~${k > 1}$ holds, notice that~${u}$ has at least one zero. Thus,~${1 + f(\alpha_i) <_\weak u}$ holds (with a strict inequality) for every~${i < k}$. This yields~${1 + f(\alpha_i) <_\weak u \leq_\weak f(\beta_0)}$ and, hence,~${\alpha_i < \beta_0}$ for every~${i < k}$. Together with~${\langle a_k, \dots, a_{n-1} \rangle \leq t'}$ and the additional case iv) in the definition of the order for~${\alpha^\bullet}$ (in contrast to~${\alpha^*}$), this yields our claim.
\end{proof}
Combining everything yields the lower bounds for sequences with gap condition of members from finite sets:
\begin{corollary}\label{cor:quasi-embedding_omega_omega_alpha_weak}
	For any infinite ordinal~${\alpha}$ and natural number~${n \in \n}$, if there exists a quasi-embedding~${\alpha \to \gapw_{n}}$, then there is also a map~${\omega^{\omega^{\alpha}} \to \gapw_{n+1}}$ that is a quasi-embedding.
\end{corollary}

\begin{proof}
	Simply combine the previous two lemmas, i.e., Lemma~\ref{lem:quasi-embedding_omega_omega_alpha_into_dot_alpha} and Lemma~\ref{lem:quasi-embedding:alpha_dot_into_weak}.
\end{proof}

In \cite{RMWGapCondition}, Rathjen, van der Meeren, and Weiermann give a much simpler proof for the lower bounds of the \emph{original} version of~${\gapw_{n}}$ for natural numbers~${n \in \n}$ as it was studied by Sch\"{u}tte and Simpson in \cite{SchuetteSimpson}. There, they simply construct direct quasi-embeddings from (nested) sequences into~${\gapw_{n}}$ and use the well-known maximal order types for Higman's relation on sequences in order to derive the claim. However, we can show that the reasoning from \cite{RMWGapCondition} cannot simply be transferred to our situation and already fails in the case of~${\gapw_{2}}$:

\begin{remark}\label{rem:cannot_embed_omega_star_into_gapw_2}
	There is no quasi-embedding~${f: \omega^* \to \gapw_{2}}$. Assume, for contradiction, that~${f}$ exists. Moreover, let~${g: \gapw_{2} \to \n}$ be a map such that for any~${s \in \gapw_{2}}$ of length~${n \in \n}$, the value of~${g(s)}$ is given by the largest length~${l \in \n}$ such that there exists a start index~${i \in \n}$ with~${i + l \leq n}$ and~${s_j = 1}$ for all~${j}$ with~${i \leq j < i + l}$. Intuitively,~${g(s)}$ gives us the length of the longest interval in~${s}$ that only consists of~ones. Equivalently, we can also define~${g(s)}$ to be equal to~${|s'|}$ for the longest sequence~${s'}$ that only consists of~${1}$s and satisfies~${s' \leq_\weak s}$.
	
	Let~${n > 0}$ be some upper bound on the length of~${f(\langle 0, 0 \rangle)}$. This value yields a bound such that~${(g \circ f)(\langle m \rangle) < n}$ holds for any~${m \in \n}$: Assume that there were some~${m \in \n}$ with~${(g \circ f)(\langle m \rangle) \geq n}$. Then,~${f(\langle 0, 0 \rangle) \leq_\weak f(\langle m \rangle)}$ holds as we can fit~${f(\langle 0, 0 \rangle)}$ into the subinterval of~${f(\langle m \rangle)}$ that only consists of~${1}$s. Since~${f}$ is a quasi-embedding, this implies the contradiction~${\langle 0, 0 \rangle \leq \langle m \rangle}$. The value of~${n}$ will stay fixed for the rest of this proof.
	
	Let~${T := f((n+1)^*)}$, i.e.,~${T}$ is the suborder of~${\gapw_2}$ consisting of all sequences~${f(s)}$ for~${s \in (n+1)^*}$. Clearly, we have~${\omega^{\omega^n} \leq o(T)}$ since~${f}$ is a quasi-embedding and the lower bound~${\omega^{\omega^n} \leq o((n+1)^*)}$ holds by Theorem~\ref{thm:Higman_order_types}.
	We will lead this to a contradiction by showing~${o(T) < \omega^{\omega^n}}$. First, we define a function~${h: \gapw_2 \to \n}$ and then argue why its range is bounded on~${T}$:
	\begin{equation*}
		h(s) :=
		\begin{cases}
			0 & \text{if~${s}$ only consists of ones and~${|s| < n}$ holds,}\\
			1 & \text{if~${s}$ only consists of ones and~${|s| \geq n}$ holds,}\\
			h(s_l) + h(s_r) & \text{if~${s = s_l * \langle 0 \rangle * s_r}$ and~${s_l}$ is as short as possible.}
		\end{cases}
	\end{equation*}
    Intuitively,~$h(s)$ counts the amount of intervals consisting only of ones in~$s$ that are separated by zeros and have a length of at least~$n$.
	We prove that for any sequences~${s, t \in \gapw_2}$ with~${g(s) \leq n}$ and~${h(t) \geq |s|}$, we have~${s \leq_\weak t}$ by induction on the amount of zeros in~${s}$. If~${s}$ has no zero, then the assumption~${g(s) \leq n}$ implies~${|s| \leq n}$. Assume there is some~${t}$ with~${h(t) \geq |s|}$ such that~${s \leq_\weak t}$ does not hold. Clearly,~${s}$ cannot be the empty sequence. Thus, we have~${h(t) > 0}$. By definition of~${h}$, the sequence~${t}$ must include an interval of length~${n}$ that only consists of ones. Since~${|s|}$ is less than or equal to~${n}$, we conclude~${s \leq_\weak t}$.
	
	Otherwise, if~${s}$ has~${k+1}$-many zeros for~${k \in \n}$, then~${s}$ can be written in the form~${s = s_l * \langle 0 \rangle * s_r}$, where~${s_l}$ does not contain any zeros and~${s_r}$ contains~${k}$-many zeros. Assume that there is some~${t}$ of minimal length with~${h(t) \geq |s|}$ such that~${s \leq t}$ does not hold. By assumption on~${s}$, we have~${|s| > 0}$ and, therefore,~${h(t) > 0}$. If~${h(t) = 1}$ holds, then this implies~${s = \langle 0 \rangle}$ and~${s \leq_\weak t}$ follows easily. Thus, we assume~${h(t) > 1}$. With this, we can write~${t = t_l * \langle 0 \rangle * t_r}$ where~${t_l}$ is as short as possible, i.e.,~${t_l}$ only consists of ones. Assume that~${|t_l| < n}$ holds. Then, by definition of~$h$, we have~${h(\langle 0 \rangle * t_r) = h(t) \geq |s|}$. Together with~${s \nleq_\weak \langle 0 \rangle * t_r}$, this contradicts the minimality of~${t}$. Therefore,~${|t_l| \geq n}$ and~${h(t_r) = h(t) - 1 \geq |s| - 1 \geq |s_r|}$ must hold. From the former together with our assumption~${g(s) \leq n}$, which yields~${|s_l| \leq n}$ since~${s_l}$ only contains ones, we conclude~${s_l \leq_\weak t_l}$ similar to before. From the latter together with~${g(s_r) \leq g(s) \leq n}$, we arrive at~${s_r \leq_\weak t_r}$. Combining everything using the usual properties of the gap condition leads to our claim~${s \leq_\weak t}$.
	
	Now, assume that~${h}$ is unbounded on~${T}$. Then, there exists some~${t \in T}$ satisfying~${h(t) \geq |f(\langle n+1 \rangle)|}$. By previous considerations, we know that~${(g \circ f)(\langle m \rangle) < n}$ holds for all numbers~${m}$, in particular for~${m := n+1}$. Thus, we can apply the result from the previous paragraph, which yields~${f(\langle n+1 \rangle) \leq_\weak t}$. By definition of~$T$, we find~${s \in (n+1)^*}$ such that~${f(s) = t}$ holds. Since~${f}$ is a quasi-embedding, we conclude~${\langle n+1 \rangle \leq_\weak s}$. This, of course, is a contradiction since all members of~${s}$ must lie strictly below~${n+1}$.
	
	For every~${m \in \n}$, let~${T_m}$ be the restriction of~${T}$ to sequences~${t}$ with~${h(t) \leq m}$. We define a family of quasi-embeddings~${k_m}$:
	The map~${k_0: T_0 \to n^*}$ is given as follows:
	\begin{equation*}
		k_0(s) :=
		\begin{cases}
			\langle |s| \rangle & \text{if~${s}$ only consists of ones,}\\
			k_0(s_l) * k_0(s_r) & \text{if~${s = s_l * \langle 0 \rangle * s_r}$ and~${s_l}$ is as short as possible.}
		\end{cases}
	\end{equation*}
	A short induction shows that~${k_0}$ is a quasi-embedding.
	
	For~${m > 0}$, we define quasi-embeddings~${k_m: T_m \to T_{m-1} \otimes \omega \otimes T_{m-1}}$ as follows: Let~${t \in T_m}$ be arbitrary and write~${t = t_l * t_m * t_r}$ where~${t_m}$ only consists of ones,~${|t_m| = g(t)}$ holds, and~${t_l}$ is as short as possible. We set~${k_m(t) := \langle t_l, |t_m|, t_r \rangle}$. It~can be shown that~${h(t_l), h(t_r) < m}$ hold.
	Assume that there is a second sequence~${s = s_l * s_m * s_r \in T_m}$ split in an analogous way such that~${k_m(s) \leq k_m(t)}$ holds. We conclude, by induction on~${m}$, that~${s_l \leq_\weak t_l}$,~${s_m \leq_\weak t_m}$, and~${s_r \leq_\weak t_r}$ hold. Now,~${s_r}$ and~${t_r}$ are either empty or start with~${0}$. Thus, we have~${s_m * s_r \leq_\weak t_m * t_r}$. Similarly,~${s_l}$ and~${t_l}$ are either empty or end with~${0}$. Hence, we have~${s \leq_\weak t}$ (using the identity between~$\leq_\weak$ and~$\leq_\gordeev$, which entails that all of our results on the weak gap condition also hold if we consider sequences in reversed order). We conclude that~${k_m}$ is a quasi-embedding for any~${m \in \n}$.
	
	By induction, we show that~${o(T_m) < \omega^{\omega^n}}$ holds. Clearly, we have the upper bound~${o(T_0) \leq \omega	^{\omega^{n-1}} < \omega^{\omega^n}}$ using~${k_0}$ and our knowledge on the maximal order types of~${n^*}$. Assume that~${o(T_m) < \omega^{\omega^n}}$ holds for some~${m \in \n}$. We conclude~${o(T_{m+1}) \leq o(T_m) \hessMul \omega \hessMul o(T_m) < \omega^{\omega^n}}$ using~${k_{m+1}}$,~${n > 0}$, and the fact that the ordinal~${\omega^{\omega^n}}$ is multiplicatively indecomposable. Finally, by previous considerations, we know that there is some~${m \in \n}$ with~${T = T_m}$. Therefore, the inequalities~${\omega^{\omega^n} \leq o(T) = o(T_m) < \omega^{\omega^n}}$ lead to a contradiction.
\end{remark}

We return to the lower bounds of Theorem~\ref{thm:maximal_order_types}:

\begin{proposition}\label{prop:lower_bound_bta}
	For any ordinal~${\alpha}$, we have~${F(\alpha) \leq o(\bta_{\alpha, 1})}$.
\end{proposition}

\begin{proof}
	The case for~${\alpha = 0}$ is clear since there is exactly one element~${0 \star []}$ in~$\bta_{0, 1}$. Let~${\alpha = \omega^\gamma}$. By Lemma~\ref{lem:quasi-embedding:varphi_into_bta}, we can find a quasi-embedding~${\varphi_{1 + \gamma}(0) \to \bta_{\alpha, 1}}$. Moreover, if~${\gamma}$ is a~${\Gamma}$-ordinal, then Corollary~\ref{cor:quasi-embedding_varphi_into_bta_gamma} yields that there is a quasi-embedding~${\varphi_{1 + \gamma}(1) \to \bta_{\alpha, 1}}$. Finally, we consider the case where~${\alpha = \omega^\gamma + \delta}$ holds for~${\delta > 0}$. Consider the quasi-embedding~${F(\delta) \to \bta_{\delta, 1}}$ that exists by induction hypothesis. Since~${\delta > 0}$ holds,~${F(\delta)}$ must be an infinite ordinal. We conclude that there already exists a quasi-embedding~${F(\delta) \to \bta_{\delta, 1} \setminus \{0 \star []\}}$ since~${0 \star []}$ is the smallest element in~${\bta_{\delta, 1}}$. Now, by Lemma~\ref{lem:quasi-embedding:varphi_into_bta}, this yields~${\varphi_{1 + \gamma}(F(\delta)) \to \bta_{\omega^\gamma + \delta, 1}}$.
\end{proof}

\begin{proposition}
	For any ordinal~${\alpha}$, we have~${G(\alpha) \leq o(\gapw_\alpha) \leq o(\btal_{\alpha, 1})}$.
\end{proposition}

\begin{proof}
	Any lower bound for~${\gapw_{\alpha}}$ is also a lower bound for~${\btal_{\alpha, 1}}$ via Lemma~\ref{lem:embed_weak_into_btal}.
	The order types of~${\gapw_{0, 1}}$ and~${\gapw_{1, 1}}$ are already given by Lemma~\ref{lem:small_upper_bounds_btal}~a) and~b). If~${\alpha > 1}$ is a natural number, we invoke Corollary~\ref{cor:quasi-embedding_omega_omega_alpha_weak}. For infinite~${\alpha}$, the argument is similar to that of the proof for Proposition~\ref{prop:lower_bound_bta}. We simply substitute Lemma~\ref{lem:quasi-embedding:varphi_into_bta} and Corollary~\ref{cor:quasi-embedding_varphi_into_bta_gamma} with Lemma~\ref{lem:quasi-embedding_varphi_into_weak} and Lemma~\ref{lem:quasi-embedding_varphi_1_into_weak}, respectively.
\end{proof}
The lower bounds for sequences using the strong gap condition are a bit more involved. We begin with a result that lets us increase a lower bound if we switch from the weak gap condition to the strong one.
\begin{lemma}
	Let~${f: \alpha \to \gapw_{\omega^\gamma + \delta}}$ be a quasi-embedding. Then, we can also construct a quasi-embedding~${g: \alpha^{(\omega^\gamma)} \to \gaps_{\omega^\gamma + \delta}}$ such that all sequences in the range of~${g}$ are empty or begin with an ordinal strictly below~${\omega^\gamma}$.
\end{lemma}

\begin{proof}
	If~${\alpha < 2}$ holds, the statement is trivial. For~${\alpha \geq 2}$, we know that any ordinal can be written in base-$\alpha$ notation.
	We define
	\begin{equation*}
		g(\sigma) :=
		\begin{cases}
			\langle \rangle & \text{if } \sigma = 0\comma\\
			(\beta + \langle 0 \rangle * f(\kappa)) * g(\lambda)& \text{if } \sigma \alphanf \alpha^\beta \cdot \kappa + \lambda\period
		\end{cases}
	\end{equation*}
	Note that since~${\omega^\gamma}$ is indecomposable, adding~${\beta}$ to~${f(\kappa)}$ produces a sequence whose members are, again, strictly below~${\omega^\gamma + \delta}$. Moreover, in the second line of our definition, it becomes clear that sequences in the range of~${g}$ can only begin with ordinals strictly below~${\omega^\gamma}$.
	
	We prove that~${g}$ is a quasi-embedding. Let~${\sigma, \tau < \alpha^{(\omega^\gamma)}}$ be ordinals satisfying~${g(\sigma) \leq_\strong g(\tau)}$. By induction along~${\sigma \hess \tau}$, we prove that this entails~${\sigma \leq \tau}$. If~${\tau = 0}$, our assumption implies~${\sigma = 0}$. Thus, we assume that~${\sigma}$ and~${\tau}$ are both positive. We write~${\sigma \alphanf \alpha^{\beta} \cdot \kappa + \lambda}$ and~${\tau \alphanf \alpha^{\beta'} \cdot \kappa' + \lambda'}$. Since we are working with the \emph{strong} variant of gap-sequences, our assumption~${g(\sigma) \leq_\strong g(\tau)}$ implies that the first element in~${g(\sigma)}$ must be less than or equal to the first element in~${g(\tau)}$. We conclude~${\beta \leq \beta'}$. Now, if~${\beta < \beta'}$ holds, we immediately have our claim~${\sigma \leq \tau}$. Assume~${\beta = \beta'}$. Either,~${g(\lambda')}$ is empty or it must begin with an element that lies strictly below~${\beta}$ by definition of our normal forms. Therefore, Lemma~\ref{lem:gap_remove_tail} yields that~${\beta + \langle 0 \rangle * f(\kappa) \leq_\strong \beta + \langle 0 \rangle * f(\kappa')}$ holds. With Lemma~\ref{lem:left_subtraction}~b) and Lemma~\ref{lem:gap_strong_and_weak}, this results in~${f(\kappa) \leq_\weak f(\kappa')}$ and, by assumption on~${f}$, in~${\kappa \leq \kappa'}$. Again, if this inequality is strict, we are done. Assume~${\kappa = \kappa'}$. Notice that~${g(\sigma)}$ and~${g(\tau)}$ begin with the same sequence followed by~${g(\lambda)}$ and~${g(\lambda')}$, respectively. Since both~${g(\lambda)}$ and~${g(\lambda')}$ begin with a member that lies strictly below all members in~${\beta + \langle 0 \rangle * f(\kappa)}$, we can invoke Lemma~\ref{lem:gap_remove_head}, which yields~${g(\lambda) \leq_\strong g(\lambda')}$. Finally, by induction hypothesis, we have~${\lambda \leq \lambda'}$.
\end{proof}
In the previous lemma, we did not make use of the whole codomain of our quasi-embedding. With the next lemma, we fill these last gaps:
\begin{lemma}
	Let~${\alpha, \beta, \gamma, \delta}$ be such that there are quasi-embeddings~${f: \alpha \to \gaps_{\omega^{\gamma} + \delta}}$ and~${g: \beta \to \gaps_{\delta}}$ where all sequences in the range of~${f}$ can only begin with ordinals strictly below~${\omega^\gamma}$. Then, there is also a map~${h: \alpha \cdot \beta \to \gaps_{\omega^\gamma + \delta}}$ that is a quasi-embedding.
\end{lemma}

\begin{proof}
	We simply put~${h(\alpha \cdot \beta' + \alpha') := (\omega^\gamma + g(\beta')) * f(\alpha')}$ and prove that this constitutes a quasi-embedding: Let there be two elements of~${\alpha \cdot \beta}$ given by~${\alpha \cdot \beta_0 + \alpha_0}$ and~${\alpha \cdot \beta_1 + \alpha_1}$ for ordinals~${\alpha_0, \alpha_1 < \alpha}$ and~${\beta_0, \beta_1 < \beta}$. We assume that the inequality~${h(\alpha \cdot \beta_0 + \alpha_0) \leq_\strong h(\alpha \cdot \beta_1 + \alpha_1)}$ holds and show~${\alpha \cdot \beta_0 + \alpha_0 \leq \alpha \cdot \beta_1 + \alpha_1}$:
	
	First, we notice that, by definition,~${f(\alpha_1)}$ must be empty or begin with an element that lies strictly below~${\omega^\gamma}$. Moreover,~${\omega^\gamma + g(\beta_0)}$ clearly only consists of elements greater than or equal to~${\omega^\gamma}$. Thus, by Lemma~\ref{lem:gap_remove_tail}, we conclude the inequality~${\omega^\gamma + g(\beta_0) \leq_\strong \omega^\gamma + g(\beta_1)}$ and, therefore, also~${g(\beta_0) \leq_\strong g(\beta_1)}$ and, finally,~${\beta_0 \leq \beta_1}$. Now, if this inequality is strict, we are done. Otherwise, we continue with~${\beta_0 = \beta_1}$, which clearly entails~${\omega^\gamma + g(\beta_0) = \omega^\gamma + g(\beta_1)}$. Hence, by Lemma~\ref{lem:gap_remove_head}, we conclude~${f(\alpha_0) \leq_\strong f(\alpha_1)}$. Finally, since~${f}$ is a quasi-embedding, this entails~${\alpha_0 \leq \alpha_1}$ and we arrive at our claim~${\alpha \cdot \beta_0 + \alpha_0 \leq \alpha \cdot \beta_1 + \alpha_1}$.
\end{proof}
Finally, we combine everything and verify the lower bounds~${H(\alpha) \leq o(\gaps_\alpha)}$, for any ordinal~${\alpha}$:
\begin{corollary}
	Let~${\alpha, \beta, \gamma, \delta}$ be ordinals as well as quasi-embeddings~${\alpha \to \gapw_{\omega^\gamma + \delta}}$ and~${\beta \to \gaps_{\delta}}$. Then, we can also construct a quasi-embedding~${\alpha^{(\omega^\gamma)} \cdot \beta \to \gaps_{\omega^\gamma + \delta}}$.
\end{corollary}

\begin{proof}
	Combine the previous two lemmas.
\end{proof}

\bibliographystyle{amsplain}
\bibliography{reification}

\providecommand{\bysame}{\leavevmode\hbox to3em{\hrulefill}\thinspace}
\providecommand{\MR}{\relax\ifhmode\unskip\space\fi MR }
\providecommand{\MRhref}[2]{%
  \href{http://www.ams.org/mathscinet-getitem?mr=#1}{#2}
}
\providecommand{\href}[2]{#2}
\begin{thebibliography}{10}

\bibitem{ChopraPakhomov}
Alakh~Dhruv Chopra and Fedor Pakhomov, \emph{Well-quasi-orders on finite trees and transfinite sequences}, 2024, Forthcoming.

\bibitem{deJonghParikh}
D.H.J {de Jongh} and Rohit Parikh, \emph{Well-partial orderings and hierarchies}, Indagationes Mathematicae (Proceedings) \textbf{80} (1977), no.~3, 195--207.

\bibitem{FDJoosten}
David Fern\'andez-Duque and Joost~J. Joosten, \emph{Hyperations, {V}eblen progressions and transfinite iteration of ordinal functions}, Ann. Pure Appl. Logic \textbf{164} (2013), no.~7-8, 785--801.

\bibitem{FreundReification}
Anton Freund, \emph{A mathematical commitment without computational strength}, Rev. Symb. Log. \textbf{15} (2022), no.~4, 880--906.

\bibitem{FreundKruskal}
\bysame, \emph{Reverse mathematics of a uniform {K}ruskal-{F}riedman theorem}, 2022, preprint available as arXiv:2112.08727.

\bibitem{FriedmanATR}
Harvey Friedman, Antonio Montalb\'{a}n, and Andreas Weiermann, \emph{Phi function}, draft, 2007.

\bibitem{FriedmanWeiermann}
Harvey Friedman and Andreas Weiermann, \emph{Some independence results related to finite trees}, Philos. Trans. Roy. Soc. A \textbf{381} (2023), no.~2248, Paper No. 20220017, 18.

\bibitem{Girard}
Jean-Yves Girard, \emph{Proof theory and logical complexity}, Studies in Proof Theory. Monographs, vol.~1, Bibliopolis, Naples, 1987.

\bibitem{GirardVauzeilles}
Jean-Yves Girard and Jacqueline Vauzeilles, \emph{Functors and ordinal notations. {I}. {A} functorial construction of the {V}eblen hierarchy}, J. Symbolic Logic \textbf{49} (1984), no.~3, 713--729.

\bibitem{Gordeev}
Lev Gordeev, \emph{Generalizations of the one-dimensional version of the {K}ruskal-{F}riedman theorems}, J. Symbolic Logic \textbf{54} (1989), no.~1, 100--121.

\bibitem{Hasegawa}
Ryu Hasegawa, \emph{Well-ordering of algebras and {K}ruskal's theorem}, Logic, language and computation (Neil~D. Jones, Masami Hagiya, and Masahiko Sato, eds.), Lecture Notes in Comput. Sci., vol. 792, Springer, Berlin, 1994, pp.~133--172.

\bibitem{Higman}
Graham Higman, \emph{Ordering by divisibility in abstract algebras}, Proc. London Math. Soc. (3) \textbf{2} (1952), 326--336.

\bibitem{Hirst}
Jeffry~L. Hirst, \emph{Reverse mathematics and ordinal exponentiation}, Ann. Pure Appl. Logic \textbf{66} (1994), no.~1, 1--18.

\bibitem{Kruskal}
Joseph Kruskal, \emph{Well-quasi-ordering, the {T}ree {T}heorem, and {V}azsonyi's conjecture}, Trans. Amer. Math. Soc. \textbf{95} (1960), 210--225.

\bibitem{Kriz}
Igor K\v{r}\'{\i}\v{z}, \emph{Well-quasiordering finite trees with gap-condition. {P}roof of {H}arvey {F}riedman's conjecture}, Ann. of Math. (2) \textbf{130} (1989), no.~1, 215--226.

\bibitem{MarconeMinimalBad}
Alberto Marcone, \emph{{On the Logical Strength of Nash-Williams’ Theorem on Transfinite Sequences}}, {Logic: from Foundations to Applications: European logic colloquium} (Wilfrid Hodges, Martin Hyland, Charles Steinhom, and John Truss, eds.), Oxford University Press, 1996.

\bibitem{NashWilliams}
C.~St. J.~A. Nash-Williams, \emph{On well-quasi-ordering finite trees}, Proc. Cambridge Philos. Soc. \textbf{59} (1963), 833--835.

\bibitem{Pohlers}
Wolfram Pohlers, \emph{Proof theory}, Universitext, Springer-Verlag, Berlin, 2009, The first step into impredicativity.

\bibitem{RMWGapCondition}
Michael Rathjen, Jeroen van~der Meeren, and Andreas Weiermann, \emph{Ordinal notation systems corresponding to {F}riedman's linearized well-partial-orders with gap-condition}, Arch. Math. Logic \textbf{56} (2017), no.~5-6, 607--638.

\bibitem{RW93}
Michael Rathjen and Andreas Weiermann, \emph{Proof-theoretic investigations on {K}ruskal's theorem}, Ann. Pure Appl. Logic \textbf{60} (1993), no.~1, 49--88.

\bibitem{RW11}
\bysame, \emph{Reverse mathematics and well-ordering principles}, Computability in context (S~Barry Cooper and Andrea Sorbi, eds.), Imp. Coll. Press, London, 2011, pp.~351--370.

\bibitem{RobertsonSeymour}
Neil Robertson and P.~D. Seymour, \emph{Graph minors. {XX}. {W}agner's conjecture}, J. Combin. Theory Ser. B \textbf{92} (2004), no.~2, 325--357.

\bibitem{SchmidtBounds}
Diana Schmidt, \emph{Bounds for the closure ordinals of replete monotonic increasing functions}, J. Symbolic Logic \textbf{40} (1975), no.~3, 305--316.

\bibitem{SchmidtHabilitation}
\bysame, \emph{Well-partial orderings and their maximal order types}, 1979, Habilitationsschrift.

\bibitem{SchmidtvdMeerenWeiermann}
Diana Schmidt, Jeroen van~der Meeren, and Andreas Weiermann, \emph{Calculating maximal order types for finite rooted unstructured labeled trees}, The legacy of {K}urt {S}ch\"utte, Springer, Cham, [2020] \copyright 2020, pp.~253--264.

\bibitem{SchuetteSimpson}
Kurt Sch\"{u}tte and Stephen~G. Simpson, \emph{Ein in der reinen {Z}ahlentheorie unbeweisbarer {S}atz \"{u}ber endliche {F}olgen von nat\"{u}rlichen {Z}ahlen}, Arch. Math. Logik Grundlag. \textbf{25} (1985), no.~1-2, 75--89.

\bibitem{SimpsonGap}
Stephen~G. Simpson, \emph{Nonprovability of certain combinatorial properties of finite trees}, Harvey {F}riedman's research on the foundations of mathematics (L.A. Harrington, M.D. Morley, A.~Sĉêdrov, and S.G. Simpson, eds.), Stud. Logic Found. Math., vol. 117, North-Holland, Amsterdam, 1985, pp.~87--117.

\bibitem{SimpsonHilbertBasis}
\bysame, \emph{Ordinal numbers and the {H}ilbert basis theorem}, J. Symbolic Logic \textbf{53} (1988), no.~3, 961--974.

\bibitem{Simpson}
\bysame, \emph{Subsystems of second order arithmetic}, second ed., Perspectives in Logic, Cambridge University Press, Cambridge; Association for Symbolic Logic, Poughkeepsie, NY, 2009.

\bibitem{UftringPhD}
Patrick Uftring, \emph{Well-ordering principles across reverse mathematics}, Ph.D. thesis, University of W\"{u}rzburg, 2025.

\end{thebibliography}

\end{document}